\documentclass[11pt,twoside]{amsart}

\makeatletter
\@namedef{subjclassname@2020}{%
  \textup{2020} Mathematics Subject Classification}
\makeatother

\DeclareRobustCommand{\SkipTocEntry}[4]{}
\makeatletter
\def\@tocline#1#2#3#4#5#6#7{\relax
  \ifnum #1>\c@tocdepth 
  \else
    \par \addpenalty\@secpenalty\addvspace{#2}%
    \begingroup \hyphenpenalty\@M
    \@ifempty{#4}{%
      \@tempdima\csname r@tocindent\number#1\endcsname\relax
    }{%
      \@tempdima#4\relax
    }%
    \parindent\z@ \leftskip#3\relax \advance\leftskip\@tempdima\relax
    \rightskip\@pnumwidth plus4em \parfillskip-\@pnumwidth
    #5\leavevmode\hskip-\@tempdima
      \ifcase #1
       \or\or \hskip 1em \or \hskip 2em \else \hskip 3em \fi%
      #6\nobreak\relax
    \dotfill\hbox to\@pnumwidth{\@tocpagenum{#7}}\par
    \nobreak
    \endgroup
  \fi}
\makeatother

\makeatletter
\def\subsubsection{\@startsection{subsubsection}{3}%
  \z@{.5\linespacing\@plus.7\linespacing}{-.5em}%
  {\normalfont\bfseries}}
 \makeatother

\usepackage[all]{xy}
\usepackage{amsmath}
\usepackage{amsfonts}
\usepackage{amssymb}
\usepackage{amsthm}
\usepackage{xcolor}
\usepackage{enumitem}
\usepackage{hyperref}
\usepackage{fullpage}
\usepackage{graphicx}
\usepackage{mathtools}
\usepackage{url}
\usepackage{tikz}

\usepackage{listings}

\usepackage{thm-restate}

\setlist[enumerate,1]{label=(\roman*)}
\setlist[enumerate,2]{label=(\alph*)}
\setlist[enumerate,3]{label=(\Roman*)}
\setlist[enumerate,4]{label=(\Alph*)}

\theoremstyle{definition}
\newtheorem{defn}{Definition}[section]
\newtheorem{ex}[defn]{Example}
\newtheorem{rmk}[defn]{Remark}

\theoremstyle{plain}
\newtheorem{theorem}[defn]{Theorem}

\newtheorem{lem}[defn]{Lemma}
\newtheorem{prop}[defn]{Proposition}
\newtheorem{cor}[defn]{Corollary}

\def\A{\ensuremath{\mathbb{A}}}
\def\C{\ensuremath{\mathbb{C}}}

\def\N{\ensuremath{\mathbb{N}}}
\def\P{\ensuremath{\mathbb{P}}}
\def\Q{\ensuremath{\mathbb{Q}}}
\def\R{\ensuremath{\mathbb{R}}}
\def\Z{\ensuremath{\mathbb{Z}}}

\def\FF{\ensuremath{\mathcal F}}

\def\HH{\ensuremath{\mathcal H}}
\def\II{\ensuremath{\mathcal I}}

\def\LL{\ensuremath{\mathcal L}}

\def\OO{\ensuremath{\mathcal O}}

\def\RR{\ensuremath{\mathcal R}}

\def\TT{\ensuremath{\mathcal T}}

\def\ch{\mathop{\mathrm{ch}}\nolimits}

\def\Coh{\mathop{\mathrm{Coh}}\nolimits}

\def\Conv{\mathop{\mathrm{Conv}}\nolimits}

\def\Db{\mathop{\mathrm{D}^{\mathrm{b}}}\nolimits}

\def\diag{\mathop{\mathrm{diag}}\nolimits}
\def\dim{\mathop{\mathrm{dim}}\nolimits}

\def\End{\mathop{\mathrm{End}}\nolimits}

\def\Ext{\mathop{\mathrm{Ext}}\nolimits}
\def\lExt{\mathop{\mathcal Ext}\nolimits}

\def\GL{\mathop{\mathrm{GL}}\nolimits}
\def\Gr{\mathop{\mathrm{Gr}}\nolimits}

\def\Hom{\mathop{\mathrm{Hom}}\nolimits}

\def\RHom{\mathop{\mathbf{R}\mathrm{Hom}}\nolimits}

\def\initial{\mathop{\mathrm{in}}\nolimits}

\def\Mov{\mathop{\mathrm{Mov}}\nolimits}
\def\mod{\mathop{\mathrm{mod}}\nolimits}
\def\min{\mathop{\mathrm{min}}\nolimits}

\def\Nef{\mathop{\mathrm{Nef}}\nolimits}

\def\PGL{\mathop{\mathrm{PGL}}}

\def\Pic{\mathop{\mathrm{Pic}}\nolimits}

\def\Rep{\mathop{\mathrm{Rep}}}

\def\SL{\mathop{\mathrm{SL}}\nolimits}

\def\State{\mathop{\mathrm{State}}\nolimits}

\def\Stab{\mathop{\mathrm{Stab}}\nolimits}

\def\into{\ensuremath{\hookrightarrow}}
\def\onto{\ensuremath{\twoheadrightarrow}}

\begin{document}

\title{
Variation of stability for moduli spaces of unordered points in the plane
}

\author{Patricio Gallardo}
\address{Department of Mathematics, University of California, Riverside, CA 92501, USA }
\email{pgallard@ucr.edu}
\urladdr{https://sites.google.com/site/patriciogallardomath/}

\author{Benjamin Schmidt}
\address{Gottfried Wilhelm Leibniz Universit\"at Hannover, Institut f\"ur Algebraische Geometrie, Welfengarten 1, 30167 Hannover, Germany}
\email{bschmidt@math.uni-hannover.de}
\urladdr{https://benjaminschmidt.github.io}



\begin{abstract}
We study compactifications of the moduli space of unordered points in the plane via variation of GIT quotients of their corresponding Hilbert scheme. Our VGIT considers linearizations outside the ample cone and within the movable cone. For that purpose, we use the description of the Hilbert scheme as a Mori dream space, and the moduli interpretation of its birational models via Bridgeland stability. We determine the GIT walls associated with curvilinear zero-dimensional schemes, collinear points, and schemes supported on a smooth conic. For seven points, we study a compactification associated with an extremal ray of the movable cone, where stability behaves very differently from the Chow quotient. Lastly, a complete description for five points is given.
\end{abstract}

\maketitle

\section{Introduction}

A central insight of algebraic geometry is that a moduli space has many geometrically meaningful compactifications, and that much of their birational geometry can be understood via the degenerations of the objects we are parametrizing. Describing this interplay is one of the leading questions in moduli theory nowadays. Within this context, we focus on a case that has been inaccessible until very recently: the moduli of $n$ unlabelled points in the plane (that is zero dimensional schemes). 

Our work begins with the most natural compactifications of the moduli space of $n$ unlabelled points in $\P^2$, which are the GIT quotients of the Hilbert scheme $\P^{2[n]}$ of $n$ points in $\P^2$ by the automorphism group $\SL_3$. Typically, these quotients are dependent on the choice of an $\SL_3$-linearized ample line bundle in $\mathbb{P}^{2[n]}$. However, we introduce a new approach by using linearizations that are not within the ample cone, but rather in the movable cone. This is a significant challenge as GIT quotients that use linearizations outside the ample cone result in technical difficulties, such as the uncertainty of whether the resulting quotient is projective or has a meaningful moduli interpretation.

To achieve our goal, we use recent developments on the birational geometry of $\P^{2[n]}$ via derived category techniques as discussed in \cite{ABCH13:hilbert_schemes_p2, CHW17:effective_cones_p2, LZ18:stability_p2}. Specifically, we take advantage of the fact that $\mathbb{P}^{2[n]}$ is a Mori dream space (see Definition \ref{defn:mds}). All movable divisors on $\P^{2[n]}$ can be written in the form $D_m \coloneqq mH - \tfrac{\Delta}{2}$ where $H$ is induced by the ample divisor on $\P^2$ and $\Delta$ is the divisor of non-reduced subschemes on $\P^2$. According to \cite{LZ18:stability_p2} for each of these divisors there is a Bridgeland stability condition parametrized by real parameters $\alpha_m > 0$ and $\beta_m \in \R$ such that the moduli space of Bridgeland semistable objects with Chern character $(1, 0, -n)$, denoted by $M_{\alpha_m, \beta_m}(1, 0, -n)$, is the birational model of $\mathbb{P}^{2[n]}$ on which $D_m$ is ample. We define 
\[
    \textbf{C}(m,n) \coloneqq 
    M_{\alpha_m, \beta_m}(1,0,-n) / \! \! /_{D_m} \SL_3.
\]
For $m \gg 0$ it turns out that the moduli space $M_{\alpha_m, \beta_m}(1,0,-n)$ is simply $\mathbb{P}^{2[n]}$, but for smaller values of $m$ the situation changes. 




Varying $D_m$ across the movable cone unlocks a wealth of geometry. Within GIT, there is a finite wall and chamber decomposition, see Definition \ref{defn:git_wall}.
Our first motivating question is: What is the structure 
of these \emph{GIT walls}? The next result is the first one in this direction.

\begin{restatable}{theorem}{MainThmWall}
\label{thm:MainThmWall}
Given $n \geq 5$ and for each $l$ satisfying $\frac{n}{2} \leq l < \frac{2n}{3}$, there is a GIT wall at $m_l = \frac{3(n - l)(n - l - 1)}{2(2n - 3l)}$  such that:
\begin{enumerate}
    \item the union of a length $(n - l)$ generic curvilinear subscheme supported at a single point and $l$ generic points becomes unstable for $m > m_l$ and is semistable for $m = m_l$, and 
    \item the union of $l$ general collinear points and $(n - l)$ general collinear points in a different line becomes unstable for $m < m_l$ and semistable at $m = m_l$.
    \item Both cases above degenerate to a common strictly semistable configuration at $m = m_l$. 
\end{enumerate}
\end{restatable}

We find more walls than these. Indeed, a similar statement for subschemes containing collinear points can be found in Section \ref{sec:WallsCollinear}. In Section \ref{sec:WallPtsConic} we establish a wall for subschemes contained in conics. The proof of Theorem \ref{thm:MainThmWall} is given in Section \ref{sec:ProofMainThmWall}. These GIT walls are increasing in $l$. While some of our walls are within the ample cone of $\P^{2[n]}$, many go well beyond that.

Our second result is motivated by the fact that our GIT quotients are Mori Dream Spaces by results of \cite{bak11:good_quotients}. Describing the largest GIT wall determines the nef cone for $\textbf{C}(m,n)$ with $m \gg 1$. 

\begin{restatable}{theorem}{largestWall}
\label{thm:largestWall}
If $3 \nmid n$, then the largest GIT wall is equal to $m_{l}$ for $l = \big\lfloor \frac{2n}{3} \big\rfloor$. This wall is in the interior of the ample cone of $\P^{2[n]}$ if and only if $n \geq 11$ or $n = 8$. In particular, for $m \gg 1$, the nef cone of $\textbf{C}(m,n)$ is generated by the divisors induced by $H$ and $D_{m_l}$. 
\end{restatable}

The proof of Theorem \ref{thm:largestWall} is given in Section \ref{sec:ProofThmLargestWall}. The restriction $3 \nmid n$ arises from the presence of strictly semistable points in the Chow quotient when $n$ is a multiple of $3$. For being able to describe the GIT stability, we rely on the interpretation of the birational models of $\P^{2[n]}$ as moduli spaces of Bridgeland semistable objects. Yet, it is remarkable that only ideals appear in the above statements. In Section \ref{subsec:bridgeland_stable_git_unstable} we show that Bridgeland-semistable objects that are not ideals are GIT-unstable in the first few birational models:

\begin{theorem}
\label{thm:git_unstable_bridgeland_stable}
For $n \geq 7$ any Bridgeland-semistable object that is not an ideal is GIT-unstable as long as $m > n - 4$. For $n = 6$ any Bridgeland-semistable object that is not an ideal is GIT-unstable as long as $m  > \tfrac{5}{2}$. Finally, for $n = 5$ any Bridgeland-semistable object that is not an ideal is GIT-unstable regardless of polarization.
\end{theorem}

We remark that Bridgeland walls (see Definition \ref{defn:bridgeland_wall}) and GIT walls are usually unrelated. Indeed, many walls from Theorem \ref{thm:MainThmWall} do not overlap with Bridgeland walls. On the other hand, in Proposition \ref{prop:wall_conic} we establish a wall where some strictly Bridgeland-semistable ideals are strictly GIT-semistable.

\subsection{Explicit description for five points}
\label{sec:Intro5pts}

In the case of five points, the largest GIT wall turns out to be the only GIT wall in the interior of the movable cone, see Figure \ref{fig:walls}.

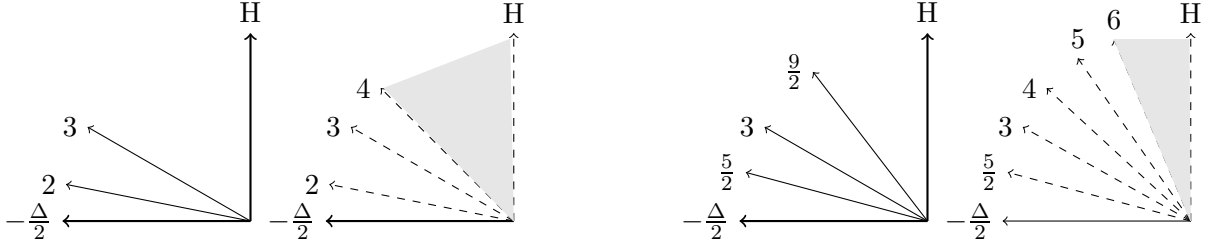
\begin{figure}[h!]
    \centering
\def\ra{1.5}    
\def\ca{0.5}    
\begin{tikzpicture}
\draw[thick,->] (0,0) -- (0,2.5) node[above] {H};
\draw[<-, thick] (-2.5,0) 
node[left] {$-\frac{\Delta}{2}$} -- (0,0); 
\draw[->] (0 ,0) -- 
( {0  - (2.5)*cos ((1/8)*pi/2 r)},
{2.5*sin( (1/8)*pi/2 r)}
) node[left] {$\tiny{2}$};
\draw[ ->] (0 ,0) -- 
( {0- (2.5)*cos ((2/6)*pi/2 r)},
{2.5*sin( (2/6)*pi/2 r)}
) node[left] {$\tiny{3}$};
\draw[dashed,->] (3 + \ca,0) -- (3+\ca,2.5) node[above] {H};
\draw[<-, thick] (0.5 +\ca,0) node[left] {$-\frac{\Delta}{2}$} -- (3+\ca,0);
\draw[dashed, ->] (3 + \ca ,0) -- 
( {3 + \ca - (2.5)*cos ((1/8)*pi/2 r)},
{2.5*sin( (1/8)*pi/2 r)}
) node[left] {$\tiny{2}$};
\draw[dashed, ->] (3 + \ca ,0) -- 
( {3 + \ca - (2.5)*cos ((2/6)*pi/2 r)},
{2.5*sin( (2/6)*pi/2 r)}
) node[left] {$\tiny{3}$};
\draw[dashed, ->] (3 +  \ca ,0) -- 
( {3 + \ca - (2.5)*cos ((3/6)*pi/2 r)},
{2.5*sin( (3/6)*pi/2 r)}
) node[left] {$\tiny{4}$};
\filldraw[gray!20] 
(3 +  \ca  ,0) -- 
( {3 + \ca +0.05 - (2.5)*cos ((3/6)*pi/2 r)},
{2.5*sin( (3/6)*pi/2 r)}) 
--
( {3 + 1.72 + \ca - (2.5)*cos ((3/6)*pi/2 r)},
{ 0.65 + 2.5*sin( (3/6)*pi/2 r)});
\draw[thick,->] (7.5 + \ra,0) -- (7.5 + \ra,2.5)node[above] {H};
\draw[thick, <-] (5+\ra,0) node[left] {$-\frac{\Delta}{2}$}
-- (7.5 + \ra,0);
\draw[ ->] (7.5 + \ra  ,0) -- 
( {7.5 + \ra - (2.5)*cos ((1/6)*pi/2 r)},
{2.5*sin( (1/6)*pi/2 r)}
)node[left] {$\tiny{\frac{5}{2}}$};
\draw[ ->] (7.5 + \ra ,0) -- 
( {7.5 + \ra - (2.5)*cos ((2/6)*pi/2 r)},
{2.5*sin( (2/6)*pi/2 r)}
) node[left] {$\tiny{3}$};
\draw[->] (7.5 + \ra ,0) -- 
( {7.5 + \ra - (2.5)*cos ((7/12)*pi/2 r)},
{2.5*sin( (7/12)*pi/2 r)}
) node[left] {$\tiny{\frac{9}{2}}$};
\draw[dashed, ->] (10.5 + \ra +\ca,0) -- (10.5+\ra + \ca,2.5)node[above] {H};
\draw[->] (10.5+\ra+\ca,0) -- 
(8+\ra+\ca,0) node[left] {$-\frac{\Delta}{2}$};
\draw[dashed, ->] (10.5 + \ra + \ca ,0) -- 
( {10.5 + \ra - (2.0)*cos ((1/6)*pi/2 r)},
{2.5*sin( (1/6)*pi/2 r)}
)node[left] {$\tiny{\frac{5}{2}}$};
\draw[dashed, ->] (10.5 + \ra + \ca ,0) -- 
( {10.5 + \ra - (2.0)*cos ((2/6)*pi/2 r)},
{2.5*sin( (2/6)*pi/2 r)}
) node[left] {$\tiny{3}$};
\draw[dashed, ->] (10.5 + \ra + \ca ,0) -- 
( {10.5 + \ra - (2.0)*cos ((3/6)*pi/2 r)},
{2.5*sin( (3/6)*pi/2 r)}
) node[left] {$\tiny{4}$};
\draw[dashed, ->] (10.5 + \ra + \ca ,0) -- 
( {10.5 + \ra - (2.0)*cos ((4/6)*pi/2 r)},
{2.5*sin( (4/6)*pi/2 r)}
) node[above] {$\tiny{5}$};
\draw[dashed, ->] (10.5 + \ra + \ca ,0) -- 
( {10.5 + \ra - (2.0)*cos ((5/6)*pi/2 r)},
{2.5*sin( (5/6)*pi/2 r)})
node[above] {$\tiny{6}$};
\filldraw[gray!20] 
(10.5 + \ra + \ca ,0)  -- 
( {10.5 + \ra - (2.0)*cos ((5/6)*pi/2 r)},
{2.5*sin( (5/6)*pi/2 r)})
--
( {10.5 +1 + \ra - (2.0)*cos ((5/6)*pi/2 r)},
{2.5*sin( (5/6)*pi/2 r)});
\end{tikzpicture}
    \caption{GIT walls (solid lines) and Bridgeland walls (dotted lines) within the movable cone of the Hilbert scheme for $n=5$ on the left, and for $n=7$ on the right. The numerical labels indicate their $m$ value.
    The gray area denotes the ample cone. 
    }
    \label{fig:walls}
\end{figure}

Next, we give a complete description of the GIT stability in this first non-trivial example. 

\begin{restatable}{theorem}{mainNFive}
\label{thm:main_n_5}
The VGIT decomposition of the movable cone for $\SL_3$ acting on $\P^{2[5]}$ and its birational models contains two chambers. All Bridgeland-stable objects that are not ideals are GIT-unstable independently of the polarization.
\begin{enumerate}
    \item If $3 < m \leq \infty$, then $Z \in \P^{2[5]}$ is GIT-stable if and only if $Z$ is reduced and no four points lie in a single line. Moreover, there are no strictly GIT-semistable points. In particular, the quotient of the symmetric product, and the quotient of the Hilbert scheme for $m \gg 0$ are isomorphic.
    \item If $2 < m < 3$, then $Z \in \P^{2[5]}$ is GIT-stable if and only if no subscheme of length $3$ lies in a single line and there is no subscheme of length $3$ supported at a single point. Moreover, there are no strictly GIT-semistable points.
\end{enumerate}
The quotients in both chambers and at $m = 3$ are isomorphic to the weighted projective plane $\P(1,2,3)$ but parametrize different orbits. For $m = 2$ the quotient is a single point. 
\end{restatable}

\subsection{Explicit description for seven points}
\label{sec:Intro7pts}

By Theorem \ref{thm:largestWall}, the first wall is at $m = \frac{9}{2}$, where the stability of configurations with a curvilinear triple point changes from unstable to semistable. This wall is outside the ample cone, see Figure \ref{fig:walls}. Moreover, other triple points with certain collinearity conditions also become semistable, and they stay semistable for lower $m$, see Example \ref{ex:SpecialTriplePoint7}. As a consequence, $\textbf{C}(m, 7)$ is a geometric quotient for $m > \frac{9}{2}$, but it is not for $m \in [\frac{9}{2}, \frac{5}{2}]$. Our next wall is at $m = 3$. Here, two types of configurations change from stable to unstable: schemes with four collinear points, or seven points supported on a conic. In contrast, schemes with a non-reduced triple point change from unstable to semistable, see  
Propositions \ref{prop:wall_conic} and \ref{prop:walls_collinear_points}.

We also describe a new remarkable compactification for $n=7$, see Section \ref{sec:FinalModel7}. We recall that the Chow quotient does not keep track of any scheme structures, and it is associated to the extremal ray $m = \infty$ of the movable cone of $\P^{2[7]}$. Therefore, we denote it as $\textbf{C}(\infty, 7)$. On the opposite end of the movable cone at $m = \frac{5}{2}$, we find a compactification $\textbf{C}(\frac{5}{2}, 7)$ that parametrizes objects with a rich scheme structure. If $D_{m_0}$ lies on the boundary of the movable cone, where $m_0 \neq \infty$, and $\dim(\textbf{C}(m_0,n)) = 2n - 8$, we refer to $\textbf{C}(m_0, n)$ as the \emph{final model}. If $\dim(\textbf{C}(m_0, n)) < 2n - 8$, then the \emph{final model} is $\textbf{C}(m_0 + \varepsilon, n)$ for small enough $\varepsilon$. As a consequence of \cite[Section 10.6]{ABCH13:hilbert_schemes_p2} the final model for $n = 7$ is a quotient of $\P(H^0(\Omega(4))) = \P^{14}$, where $\Omega$ is the cotangent bundle on $\P^2$. This is due to the fact that the Bridgeland-semistable objects are quotients of non-trivial morphisms $\OO(-5) \to \Omega(-1)$. 
Since $\Omega$ is $\SL_3$-equivariant, we obtain a natural $\SL_3$ action on this $\P^{14}$. There are three important subloci:

\begin{enumerate}
    \item The locus $X_1$ is the closure of ideal sheaves of zero-dimensional length seven subschemes that contain two general reduced points and a third point of length $5$ projectively equivalent to the one cut out by the ideal $(x^2, xy^2, y^3 + xyz + xz^2)$.
    \item The locus $X_2$ is the closure of ideal sheaves of zero-dimensional length seven subschemes that contain three general reduced points and a point of length $4$ projectively equivalent to the one cut out by the ideal $(x^2, y^2)$. 
    \item The locus $X_3$ is the closure of ideal sheaves of zero-dimensional length seven subschemes that contain a curvilinear triple point and four general reduced points such that there is a line that intersects the triple point in length two and one more of the reduced points.
\end{enumerate}

It turns out that $X_1 \cup X_2 \subset X_3$.

\begin{restatable}{theorem}{stabilityFinalModelNSeven}
\label{thm:stability_final_model_n_7}
GIT-stability for $\SL_3$ acting on $\P(H^0(\Omega(4))) = \P^{14}$ is given as follows. The locus of semistable points is $\P^{14} \backslash (X_1 \cup X_2)$ and the locus of stable points is $\P^{14} \backslash X_3$.
\end{restatable}

Note that this description appears completely geometric despite the fact that $\P^{14}$ contains points that do not correspond to subschemes. This does not mean that all semistable points are given by ideal sheaves and indeed, counterexamples occur among minimal orbits:

\begin{prop}[See Proposition \ref{prop:n_7_final_model_closed_orbits}]
The closed strictly semistable orbits of $\SL_3$ acting on $\P(H^0(\Omega(4)))$ are given by subschemes $Z_w$ cut out by the ideal
\[
I_w = (xy, x^2, (w + 1)y^2 + xz) \cap (yz, z^2, wy^2 + xz) \cap (x, z)
\]
for $w \in \C \backslash \{ 0, -1 \}$ plus two further orbits that are not representing zero-dimensional subschemes. In the quotient $\P(H^0(\Omega(4))) /\!\!/ \SL_3$ these extra orbits are represented by the limits of $\overline{Z_w}$ for $w \to 0$ and $w \to \infty$. We have $Z_{-w - 1} \in \SL_3 \cdot Z_w$, $\lim_{w \to 0} Z_w \in \SL_3 \cdot \lim_{w \to -1} Z_w$, and no other pair of points lies in the same orbit. In particular, the image of the strictly semistable locus in the quotient is isomorphic to $\P^1$.
\end{prop}

Finally, there is an unexplored relationship between $\textbf{C}(m,n)$ and the GIT of plane curves. A starting point is Proposition \ref{prop:plane_curves_vs_points}. 

\subsection{Ingredients in the proofs}

An important element of the proof of Theorem \ref{thm:MainThmWall} is the use of the Hilbert-Mumford numerical criterion on ideals with a positive dimensional stabilizer. This is inspired by the approach in \cite{AFS13:hilbert_stability_bicanonical} based on results in \cite{Kem78:instability}. In our case, stabilizers are not multiplicity free, and we have to do  more than just diagonal one-parameter subgroups. We also rely on the moduli interpretation for birational models of $\P^{2[n]}$ due to \cite{ABCH13:hilbert_schemes_p2, LZ18:stability_p2}. More precisely, we need to know that for a given wall $m = m_l$ the ideal sheaves of strictly-semistable points are Bridgeland-semistable for any stability condition that induces a multiple of the divisor $D_{m_l}$. This leads to the inequality $l \geq \tfrac{n}{2}$. The other condition $l < \tfrac{2n}{3}$ ensures that $D_{m_l}$ is in the movable cone.

The proof of Theorem \ref{thm:largestWall} is done by showing that the set of GIT-unstable objects is not changing between $m = m_{\lfloor 2n/3 \rfloor}$ and $m = \infty$. To handle the case of $\lceil 2n/3 \rceil$ points in a line the crucial non-trivial ingredient is understanding limiting behavior of generic configurations of points. We use fundamental results about this problem from \cite{OS99:cutting_corners}. A similar, but simpler and more concrete argument allows us to show that there are no further walls for $n = 5$.

The proof of Theorem \ref{thm:git_unstable_bridgeland_stable} uses different techniques. Here, we rely on the fact that 
moduli spaces of Bridgeland-semistable objects in $\P^2$ can be written as moduli spaces of quiver representations with relations, so we study those. 

Finally, the proof of Theorem \ref{thm:stability_final_model_n_7} combines all of our tools. By the use of Bridgeland stability, the final model for $n = 7$ is $\P^{14}$. We obtain an explicit description of the action of $\SL_3$ on this $\P^{14}$. From here  GIT-techniques are used to describe the semistable loci. A key difficulty is to translate the abstract description in $\P^{14}$ back to loci that can be understood in terms of the Hilbert scheme. We use computer calculations to deal with some of the technicalities. They are accessible at Zenodo, a general-purpose open repository operated by CERN, see \cite{code}.

\subsection{Related work}
Compactifying the moduli spaces of points in the plane (labeled and unlabelled) via invariant theory can be traced back to the work of Coble in the 1910's \cite{Cob15:point_sets} - therefore the name of our spaces, see also \cite{DO88:point_sets}. 
Recent compactifications of the moduli of labeled points in the plane \cite{ST21:projectivity} and \cite{GR17:wonderful} yield birational compactifications to our space after quotienting them by the permutation group. However, these approaches do not consider the scheme structure of the points - which is a central feature of our work. In fact, compactifications based on the Hilbert scheme are largely unexplored with the exception of early work due to N. Durgin on $\textbf{C}(m,6)$ with $m \gg 1$, \cite{Dur15:quotient_hilb_six}, and applications to the moduli of polarized del Pezzo surfaces \cite{Ish82:moduli_del_pezzo}.

\subsubsection*{Acknowledgements}
We thank Natalie Durgin and C\'esar Lozano Huerta for useful discussions and remarks on preliminary versions of this article.
We also thank the referees for many remarks that greatly improved the exposition of the paper.  
B.S. was supported by an AMS-Simons travel grant during part of this work. P. Gallardo was supported by the University of California, Riverside and Washington University at St Louis.

\tableofcontents

\section{Background}

\subsection{Hilbert schemes of points}

We start by recalling basic properties of Hilbert schemes of points in $\P^2$. 
By $\P^{2[n]}$ we denote the \emph{Hilbert scheme} parametrizing subschemes $Z \subset \P^2$ of dimension zero and length $n$. The symmetric product is denoted by $\P^{2(n)}$. By work of Fogarty \cite{Fog68:hilbert_schemeI} the Hilbert scheme is a desingularization of the symmetric product via the \emph{Hilbert--Chow morphism}
$\P^{2[n]} \to \P^{2(n)}$
that maps a subscheme to its underlying cycle. The exceptional divisor of this morphism is denoted by $\Delta$ and consists of all non-reduced subschemes. A second divisor can be defined as follows. The hyperplane $H \subset \P^2$ induces an $S_n$-equivariant divisor on $(\P^{2})^n$ which descends to a divisor $H^{(n)}$ on $\P^{2(n)}$. The pullback of this divisor via the Hilbert--Chow morphism is denoted by $H^{[n]}$. By abuse of notation we will still call this divisor $H$. Further work by Fogarty \cite{Fog73:hilbert_schemeII} shows
\[
\Pic(\P^{2[n]}) = \Z H \oplus \Z \frac{\Delta}{2}.
\]

Next we need to study the induced maps of various divisors in $\P^{2[n]}$. We follow the account in \cite[Section 3]{ABCH13:hilbert_schemes_p2}. These results originated in \cite{BS88:k_very_ample, CG90:d_very_ample, LQZ03:nef_cone_p2}. For any positive integer $m$ with
$
n < \binom{m + 2}{2}
$
and subscheme $Z \in \P^{2[n]}$ we have an inclusion $H^0(\II_Z(m)) \subset H^0(\OO(m))$. This induces a rational map to the Grassmannian. We compose it with the Pl\"ucker embedding to get
\[
\phi_m: \P^{2[n]} \dashrightarrow \Gr\left(\binom{m + 2}{2} - n, \binom{m + 2}{2} \right) \into \P^{N_n}
\]
for appropriate $N_n$. The image $Z_m \coloneqq \phi_m(Z)$ of a point $Z \in \P^{2[n]}$ is called the \emph{$m$-th Hilbert point} of $Z$. If $Z$ is in the indeterminacy loci of $\phi_m$, then we say that the $m$-th Hilbert point of $Z$ is not well-defined. Whenever, we talk about $Z_m$ we will suppose that it is well defined. One can define divisors $D_m$ via the pullback $\OO(D_m) = \phi_m^* \OO(1)$ which have the following two properties.

\begin{prop}[{\cite[Proposition 3.1]{ABCH13:hilbert_schemes_p2}}]
\label{prop:d_m}
The class of $D_m$ is given by
$
D_m = mH - \frac{\Delta}{2}.
$
\end{prop}

We will write $D_m$ for $mH - \frac{\Delta}{2}$ for all $m \in \R$. Additionally, we define $D_{\infty} \coloneqq H$. This is motivated by the fact that the ray $\R_{\geq 0} \cdot D_m$ converges to $\R_{\geq 0} \cdot H$ for $m \to \infty$.

\begin{prop}[\cite{LQZ03:nef_cone_p2}]
The divisor $D_{n-1}$ is globally generated, but not ample. In particular, the nef cone of $\P^{2[n]}$ is generated by $D_{\infty} = H$ and $D_{n-1}$.
\end{prop}

\subsection{Geometric invariant theory}

Geometric invariant theory (GIT) was introduced by Mumford. For full details we refer to \cite{Dol03:invariant_theory, MFK94:git}. Let $G$ be a reductive group, $X$ be an irreducible variety with a $G$-action that is either affine or projective, and $\LL \in \Pic^G(X)$ be an $G$-linearized line bundle. Note that this theory is often laid out for $X$ projective only, but we will require the case $X$ affine as well. Moreover, we will eventually also deal with the situation where $X$ is affine, but not irreducible. 

For any one-parameter subgroup $\lambda: \C^* \to G$ and $x \in X$, we can take the limit
$x_0 \coloneqq \lim_{t \to 0} \lambda(t) \cdot x$.
Note that $x_0$ is a fixed point by the action of $\lambda$. Therefore, $\lambda$ induces a $\C^*$-action on the fiber $\LL_{x_0} = \C$. If $\mu$ is the weight of this action, then the \emph{Hilbert Mumford index} is defined to be
$\mu^{\LL}(x, \lambda) \coloneqq -\mu$. The following proposition and lemma are immediate consequences of the definition.

\begin{prop}
The Hilbert Mumford index is linear in the line bundle, i.e., for two line bundles $\LL_1$ and $\LL_2$
\[
\mu^{\LL_1 \otimes \LL_2}(x, \lambda) = \mu^{\LL_1}(x, \lambda) + \mu^{\LL_2}(x, \lambda).
\]
\end{prop}

\begin{lem}
\label{lem:localHM}
Assume that $X$ and $Y$ are birational, $G$-equivariantly isomorphic in codimension two, and let $U$ be the maximal open subset on which $X$ and $Y$ are isomorphic. Let $L \in \Pic^G(X) = \Pic^G(Y)$ and $x \in U$. If $\lambda$ is a one-parameter subgroup in $G$ such that $\lim_{t \to 0} \lambda(t) \cdot x \in U$, then $\mu^{\LL}(x, \lambda)$ is independent of whether we compute it on $X$ or $Y$.
\end{lem}

If the limit in the above lemma lies outside of $U$, then the Hilbert-Mumford indices might differ. This can lead to issues regarding the linearity of the Hilbert-Mumford criterion in the above proposition. As a convention, if $\LL$ is ample on $X$, then we compute $\mu^{\LL}(x, \lambda)$ on $X$ and not on $Y$.

\begin{defn}[Hilbert-Mumford criterion]
\label{def:HMcrit}
Assume that $\LL \in \Pic^G(X)$ is ample.
\begin{enumerate}
\item A point $x \in X$ is $\LL$-stable if for all one-parameter subgroups $\lambda$ we have $\mu^{\LL}(x, \lambda) < 0$. The set of all stable points is denoted by $X^{s}(\LL)$.
\item A point $x \in X$ is called $\LL$-semistable if for all one-parameter subgroups $\lambda$ we have $\mu^{\LL}(x, \lambda) \leq 0$. The set of all semistable points is denoted by $X^{ss}(\LL)$.
\item The point $x$ is called  \emph{$\LL$-unstable} if it is not $\LL$-semistable.
\item Two points $x, y \in X^{ss}(\LL)$ are called \emph{S-equivalent} if their orbit closures in $X^{ss}(\LL)$ intersect.
\end{enumerate}
\end{defn}

Note that this is commonly not used as the definition, but proved. We refer to \cite[Theorem 9.1]{Dol03:invariant_theory} for the case where $X$ is projective and to \cite[Proposition 2.5]{Kin94:moduli_quiver_reps} for the affine case.

The sets $X^s(\LL)$ and $X^{ss}(\LL)$ are both open $G$-invariant subsets of $X$. If $X$ is projective or affine, then the sets $X^s$ are always affine, and we will never use any other set-up.

\begin{theorem}[{\cite[Section 8]{Dol03:invariant_theory}}]
There is a quasi-projective variety $X /\!\!/_{\LL} G$ together with a surjective morphism $\pi: X^{ss}(\LL) \onto X /\!\!/_{\LL} G$. This morphism is universal in the sense that all $G$-invariant morphisms $X^{ss}(\LL) \to Y$ factor through $\pi$. The fibers of $\pi$ are precisely the S-equivalence classes of semistable points. Moreover, if $X$ is projective, then so is $X /\!\!/_{\LL} G$.
\end{theorem}

\begin{rmk}
\label{rmk:minNumerical}
We need a way to compute the Hilbert-Mumford index in concrete situations (see \cite[Section 9.1]{Dol03:invariant_theory} for more details). Let $G$ act on a projective variety $X \subset \P^n$ via a representation $G \to \SL_{n+1}$. Let $\LL$ be the restriction of $\OO(1)$ to $X$. For a point $x$, we can find an affine lift $x^* = (x_0, \ldots, x_ n) \in \A^{n+1}$. If $\lambda$ is a one-parameter subgroup, then we can find a change of coordinates such that $\lambda$ embeds diagonally into $\SL_{n+1}$. For simplicity, assume that $\lambda$ is diagonal in the first place. Then there are integers $\lambda_0, \ldots, \lambda_n$ such that $\lambda(t) \cdot x^* = (t^{\lambda_0} x_0, \ldots, t^{\lambda_n} x_n)$ and thus, 
\begin{align}\label{eq:numericalCriterion}
\mu^{\LL}(x, \lambda) = \min \{\lambda_i : x_i \neq 0, i = 0, \ldots, n\}.    
\end{align}
The proof is based on the fact that the line bundle corresponding to the locally free sheaf $\OO_{\P^n}(-1)$ is given by
\[
\{ (x, L) \in \A^n \times \P^n : x \in L \}.
\]
\end{rmk}

The fact that the GIT quotient depends on the choice of a $G$-linearized line bundle is not a bug but a feature. Understanding how the quotient varies can lead to interesting insights to the birational geometry of the quotients. The study of this aspect was started in \cite{DH98:vgit} and \cite{Tha96:git_flips}.

\begin{theorem}[{\cite[Theorems 2.3 \& 2.4]{Tha96:git_flips}}]
\label{thm:wall_and_chamber_git}
There is a finite set of codimension one subspaces in $\Pic^G(X)$ such that the semistable set $X^{ss}(\LL)$, $X^{s}(\LL)$, and the quotient $X /\!\!/_{\LL} G$ remain fixed  on connected components of their complement. 
\end{theorem}

\begin{defn}
\label{defn:git_wall}
A set of subspaces as in Theorem \ref{thm:wall_and_chamber_git} is called \emph{minimal} if we cannot remove any of them without contradicting the conclusion of Theorem \ref{thm:wall_and_chamber_git}. 
The elements of such a minimal set are called \emph{GIT walls} (or sometimes just \emph{walls}) and the connected components of the complement are called \emph{chambers}.
\end{defn}

Note that if $G$ has no non-trivial characters, e.g., $G = \SL(n)$, then the quotient only depends on the ray $\R_{> 0} \cdot [\LL] \in N^1(X)_{\R}$ (see \cite[Section 2]{Tha96:git_flips}). An important question when we vary $\LL$ to non-ample line bundles is how GIT-stability behaves under $G$-equivariant morphisms. Let $\pi: Y \to X$ be a $G$-equivariant morphism of projective varieties. Let $\LL_X$, $\LL_Y$ be $G$-linearized ample divisors on $X$, $Y$, respectively.
 
\begin{theorem}[{\cite[Theorem 2.1]{Rei89:relative_git}}]
\label{thm:RelativeGIT}
Assume we have points $y \in Y$ and $x = \pi(y) \in X$.
\begin{enumerate}
    \item If $x$ is $\LL_X$-unstable, then $y$ is $(\LL_Y \otimes \pi^* \LL_X^d)$-unstable for $d \gg 0$.
    \item If $x$ is $\LL_X$-stable, then $y$ is $(\LL_Y \otimes \pi^* \LL_X^d)$-stable for $d \gg 0$.
\end{enumerate}
\end{theorem}

\subsection{Bridgeland stability}
\label{subsec:bridgeland}

The birational geometry of the Hilbert scheme of points $\P^{2[n]}$ has been extensively studied with Bridgeland stability. We need to recall some basics about stability before mentioning the main results about $\P^{2[n]}$.

\begin{defn}
The classical \emph{slope} for a coherent sheaf $E \in \Coh(\P^2)$ is defined as
\[
\mu(E) \coloneqq \frac{\ch_1(E)}{\ch_0(E)},
\]
where division by zero is interpreted as $+\infty$.
\end{defn}

\begin{defn}
A coherent sheaf $E$ is called \emph{slope-(semi)stable} if for any non-trivial proper subsheaf $F \into E$ the inequality $\mu(F) < (\leq) \mu(E/F)$ holds.
\end{defn}

All ideal sheaves of points are slope-stable, and have Chern character $(1, 0, -n)$. 
Therefore, it is not difficult to show that $\P^{2[n]}$ can be interpreted as the moduli space of slope-stable sheaves with Chern character $(1,0,-n)$. The notion of Bridgeland stability is required to describe other birational models of the Hilbert scheme. The key idea due to Bridgeland is to change the category of coherent sheaves for another heart of a bounded t-structure inside the bounded derived category. This was first done for K3 surface in \cite{Bri08:stability_k3}, and then later generalized to all surfaces in \cite{AB13:k_trivial}.

Let $\beta$ be an arbitrary real number. Then the twisted Chern character $\ch^{\beta}$ is defined as
\[
\ch^{\beta}_0 =  \ch_0, \ \ch^{\beta}_1 = \ch_1 - \beta \ch_0, \ \ch^{\beta}_2 = \ch_2 - \beta \ch_1 + \frac{\beta^2}{2} \ch_0.
\]
Note that for $\beta \in \Z$, we simply have $\ch^{\beta}(E) = \ch(E(-\beta))$.

\begin{defn}
\begin{enumerate}
\item A torsion pair is defined by
\begin{align*}
\TT_{\beta} &\coloneqq \{E \in \Coh(\P^2) : \text{any quotient $E \onto G$ satisfies $\mu(G) > \beta$} \}, \\
\FF_{\beta} &\coloneqq  \{E \in \Coh(\P^2) : \text{any subsheaf $0 \neq F \subset E$ satisfies $\mu(F) \leq \beta$} \}.
\end{align*}
The heart of a bounded t-structure is given as the extension closure $\Coh^{\beta}(\P^2) \coloneqq \langle \FF_{\beta}[1],\TT_{\beta} \rangle$. 
\item Let $\alpha > 0$ be a positive real number. The \emph{Bridgeland-slope} is defined as
\[
\nu_{\alpha, \beta} \coloneqq \frac{\ch^{\beta}_2 - \frac{\alpha^2}{2} \ch^{\beta}_0}{\ch^{\beta}_1}.
\]
\item An object $E \in \Coh^{\beta}(\P^2)$ is called \emph{Bridgeland-(semi)stable} (or \emph{$\nu_{\alpha,\beta}$-(semi)stable}) if for any non trivial proper subobject $F \subset E$ the inequality $\nu_{\alpha, \beta}(F) < (\leq) \nu_{\alpha, \beta}(E/F)$ holds.

\item We call the pair $(\Coh^{\beta}(\P^2), \nu_{\alpha, \beta})$ a \emph{Bridgeland stability condition}. 

\item The moduli space of $\nu_{\alpha, \beta}$-semistable objects of a fixed class $v \in K_0(\P^2)$ will be denoted by $M_{\alpha, \beta}(v)$. 
\end{enumerate}
\end{defn}

The following theorem is due to \cite{Bog78:inequality} in the case of sheaves. We also refer to \cite[Theorem 12.1.1]{HL10:moduli_sheaves} for a modern proof. For the case of Bridgeland stability we refer to \cite[Corollary 7.3.2]{BMT14:stability_threefolds}, but suspect that it was already implicitly contained in \cite{Bri08:stability_k3} and \cite{AB13:k_trivial}.

\begin{theorem}[Bogomolov's inequality]
If $E$ is slope-semistable or Bridgeland-semistable, then
\[
\Delta(E) \coloneqq \ch_1(E)^2 - 2\ch_0(E) \ch_2(E) \geq 0.
\]
\end{theorem}

For $v, w \in K_0(\P^2)$ we define
\[
W(v, w) \coloneqq \{(\alpha, \beta) \in \R_{> 0} \times \R : \nu_{\alpha, \beta}(v) = \nu_{\alpha, \beta}(w) \}.
\]
If $W(v, w)$ is neither $\emptyset$ nor all of $\R_{> 0} \times \R$, then we call it a \emph{numerical wall}. Moreover, if the set of $\nu_{\alpha, \beta}$-semistable objects for a given $v$ changes for points on different sides of the wall, then we call it an \emph{actual wall} for $v$.

The structure of walls is rather simple. It is usually called Bertram's Nested Wall Theorem and appeared in \cite{Mac14:nested_wall_theorem}. It roughly says that walls come in two sets of nested non-intersecting semicircles, potentially separated by a linear wall parallel to the $\alpha$-axis.

\begin{theorem}
\label{thm:Bertram}
Let $v \in K_0(\P^2)$ be a fixed class. 
All numerical walls in the following statements are with respect to $v$.
\begin{enumerate}
  \item Numerical walls in Bridgeland stability are either semicircles centered on the $\beta$-axis or rays parallel to the $\alpha$-axis. If $v_0 \neq 0$, there is a unique numerical vertical wall given by $\beta = v_1/v_0$.
  \item If $v_0 \neq 0$, then the curve $\nu_{\alpha, \beta}(v) = 0$ is given by a hyperbola, which may be degenerate. Moreover, this hyperbola intersects all semicircular walls at their top point.
  \item If two numerical walls given by classes $w,u \in K_0(\P^2)$ intersect, then $v$, $w$ and $u$ are linearly dependent. In particular, the two walls are completely identical.
  \item There is a largest semicircular wall.
\end{enumerate}
\end{theorem}

Short exact sequences that induce walls in Bridgeland stability satisfy further properties:

\begin{prop}[Structure of destabilizing sequences]
\label{prop:structure_sequences}
Let $0 \to F \to E \to G \to 0$ be a short exact sequence in $\Coh^{\beta}(\P^2)$ for $(\alpha, \beta) \in W(F, E)$ with $\ch_0(E) > 0$. All of the following statements hold with $F$ and $G$ exchanged.
\begin{enumerate}
    \item If $\ch_0(F) \geq \ch_0(E)$, then the radius $\rho(F, E)$ of $W(F, E)$ satisfies
    \[
    \rho(F, E)^2 \leq \frac{\Delta(E)}{4\ch_0(F)(\ch_0(F) - \ch_0(G))}
    \]
    
    \item If one of $\ch^{\beta}_1(F) = 0$, $\ch^{\beta}_1(F) = \ch^{\beta}_1(E)$, $\mu(F) = \mu(E)$, or $\mu(E) = \mu(G)$ holds, where $(\alpha, \beta) \in W(F, E)$, then we are dealing with a vertical wall. In the case of a semicircular wall, we must have $0 < \ch^{\beta}_1(F) < \ch^{\beta}_1(E)$ for any $(\alpha, \beta) \in W(F, E)$.
    
    \item Let $W(F, E)$ be a semicircular wall and $\ch_0(F) > 0$. Then $\nu_{\alpha, \beta}(F) > \nu_{\alpha, \beta}(E)$ below $W(F, E)$ if and only if $\mu(F) < \mu(E)$.
    
    \item Let $\ch_0(F) > 0$ and $\mu(F) < \mu(E)$. Then the center of $W(F, E)$ is decreasing in $\ch_2(F)$. In particular, the radii of these walls are increasing as long as they are to the left of the vertical wall.
    
    \item Assume that the wall is semicircular. Then we have either $\ch_0(F) > 0$ and $\mu(F) < \mu(E)$, or $\ch_0(G) > 0$ and $\mu(G) < \mu(E)$.
    
    \item Assume that $\ch_0(F) > 0$ and that $F$ and $G$ are Bridgeland-stable along the wall. Then there is a neighborhood of the wall in which any non-trivial extension
    \[
    0 \to F \to E \to G \to 0
    \]
    is $\nu_{\alpha, \beta}$-stable if and only if $\nu_{\alpha, \beta}(F) < \nu_{\alpha, \beta}(E)$.
\end{enumerate}
\end{prop}

\begin{proof}
Part (i) is proved in \cite[Lemma 2.4]{MS20:space_curves} inspired by a similar statement in \cite[Proposition 8.3]{CH16:ample_cone_plane}. The definition of $\Coh^{\beta}(\P^2)$ implies that $\ch^{\beta}_1(F) \geq 0$, and $\ch^{\beta}_1(G) \geq 0$. With this in mind, part (ii), (iii), and (iv) are straightforward computations.

A proof of part (v) can be found in \cite[Proposition 2.7]{Sch21:low_rank_p3}. Note that we assume that the wall is semicircular which excludes $\mu(F) \neq \mu(E)$ and $\mu(G) \neq \mu(E)$. For a proof of (vi) we refer to \cite[Lemma 3.11]{Sch20:stability_threefolds}.
\end{proof}

The next proposition allows to relate Bridgeland stability to the Hilbert scheme. It appeared in the case of K3 surfaces in \cite[Proposition 14.2]{Bri08:stability_k3}, but the proof works on other surfaces as well.

\begin{lem}
\label{lem:large_volume_limit}
If $E \in \Coh^{\beta}(\P^2)$ is $\nu_{\alpha, \beta}$-semistable
for all $\alpha \gg 0$, then it satisfies one of the following conditions:
\begin{enumerate}
\item $\HH^{-1}(E)=0$ and $\HH^0(E)$ is a torsion-free slope semistable sheaf,
\item $\HH^{-1}(E)=0$ and $\HH^0(E)$ is a torsion sheaf, or
\item $\HH^{-1}(E)$ is a torsion-free slope semistable sheaf, and $\HH^0(E)$ is either $0$ or a torsion sheaf supported in dimension zero.
\end{enumerate}
Conversely, assume that $E \in \Coh(\P^2)$ is a torsion-free slope stable sheaf and $\beta < \mu(E)$. Then $E \in \Coh^{\beta}(\P^2)$ is $\nu_{\alpha,\beta}$-stable for $\alpha \gg 0$.
\end{lem}

\subsection{Mori dream spaces and Hilbert schemes}
\label{subsec:mds_hilbert_scheme}

The idea in \cite{ABCH13:hilbert_schemes_p2} is that crossing walls in Bridgeland stability for $v = (1,0,-n)$ should provide different birational models for $\P^{2[n]}$.
By \cite[Cor 1.3.2]{BCHM10:mmp},  $\P^{2[n]}$ is a \emph{Mori dream space}. 
We introduce a few definitions before explaining this statement in more detail. Let $X$ be a normal projective variety. By $\Mov(X)$ we denote the closed \emph{cone of movable divisors (or movable cone)} in $N^1(X)_{\Q}$.

\begin{defn}[{\cite[Definition 1.10]{HK00:mds}}]
\label{defn:mds}
The variety $X$ is called a Mori dream space if the following three things hold.
\begin{enumerate}
\item \label{item:mds1} The variety $X$ is $\Q$-factorial, and $\Pic(X)_{\Q} = N^1(X)_{\Q}$.
\item \label{item:mds2} The nef cone of $X$ is spanned by finitely many semi-ample divisors.
\item There is a finite collection of birational maps $f_i: X \dashrightarrow X_i$ that are isomorphisms in codimension two such that the $X_i$ also satisfy (\ref{item:mds1}), (\ref{item:mds2}), and
\[
    \Mov(X) = \bigcup_{i} f_i^* \Nef(X_i).
\]
\end{enumerate}
The maps $f_i$ are called \emph{small $Q$-factorial modifications (SQM)}.
\end{defn}

\begin{theorem}{\cite{LZ18:stability_p2}}
\label{thm:moduli_spaces_tilt_stability}
For any $\alpha > 0$, $\beta < 0$ not lying on a wall, the moduli space of Bridgeland-semistable objects $M_{\alpha, \beta}(1,0,-n)$ is either empty or a smooth projective variety birational to $\P^{2[n]}$. Moreover, all the SQM's in the Mori dream space structure of $\P^{2[n]}$ are given by precisely these birational models except possibly one at the end that is not isomorphic in codimension two.
\end{theorem}

To understand the divisors better they use a tool developed in \cite{BM14:projectivity} called the \emph{Positivity Lemma}. In this case, it says that for any $(\alpha, \beta)$ there is a nef divisor class on the moduli space of $\nu_{\alpha, \beta}$-stable objects with Chern character $(1,0,-n)$. Moreover, unless $(\alpha, \beta)$ lies on a wall the divisor is ample. One can express the divisor in terms of our chosen basis of $\P^{2[n]}$. If $(\alpha, \beta)$ lies on a semicircular wall with center $s$, then up to scale
\[
    D_{\alpha, \beta} = \left(-s - \frac{3}{2}\right)H - \frac{\Delta}{2}.
\]
A key ingredient in the proof of Theorem \ref{thm:moduli_spaces_tilt_stability} is that this construction maps the wall and chamber decomposition in Bridgeland stability to the Mori dream space decomposition
\[
    \Mov(\P^{2[n]}) = \bigcup_{i} f_i^* \Nef(X_i).
\]

\begin{defn}
\label{defn:bridgeland_wall}
We say that a ray $\R_{> 0} \cdot D_m$ is a \emph{Bridgeland wall} for $\P^{2[n]}$ if there is a wall in Bridgeland stability for $(1, 0, -n)$ such that the divisor induced by the Positivity Lemma for any stability conditions along this wall lies on this ray as well.
\end{defn}

\subsection{Moduli of quiver representations}

Next, we discuss GIT for affine varieties as studied in \cite{Kin94:moduli_quiver_reps}. If $X$ is affine, then the trivial line bundle is ample. In contrast to the projective case, choosing the linearization has consequences. It is the same as the choice of a character of $G$. Let $\theta$ be such a character. To be consistent with our previous sign convention (which differs from \cite{Kin94:moduli_quiver_reps}), we linearize $\OO_X$ via the action of $-\theta$.

Let $Q$ be a quiver with vertices $Q_0$ and arrows $Q_1$. The maps $s,t: Q_1 \to Q_0$ determine the source and target of each arrow. Let $\RR(Q, d)$ be the space of all representations of the quiver with a fixed dimension vector $d = (d_i)_{i \in Q_0}$. This space has a natural group action of $\GL(d)$ by change of basis. More explicitly:
\[
    \RR(Q, d) \coloneqq \bigoplus_{a \in Q_1} \Hom(\C^{d_{s(a)}}, \C^{d_{t(a)}}), \ \GL(d) \coloneqq \prod_{i \in Q_0} \GL(d_i).
\]
Note that the one-dimensional long diagonal $\Delta \subset \GL(d)$ acts trivially on all of $\RR(Q, d)$. Therefore, we take quotients with respect to $\widetilde{\GL}(d) \coloneqq \GL(d)/\Delta$. A character $\theta$ of $\GL(d)$ is given as a product of powers of determinants. We will generally write $\theta = (\theta_i)_{i \in Q_0}$ to mean the character that maps an element $g \in \GL(d_i)$ to $\det(g)^{\theta_i}$. 

For any $e = (e_i)_{i \in Q_0} \in \mathbb{Z}^{Q_0}$ we define $\theta(e) \coloneqq \sum_{i \in Q_0} \theta_i e_i$ and obtain a pairing that is linear in both $\theta$ and $e$. 

\begin{rmk}
A character of $\widetilde{\GL}(d)$ is a character of $\GL(d)$ such that the product of its determinants is trivial. Therefore, we will always assume $\theta(d) = 0$. 
\end{rmk}

\begin{defn}
Let $V$ be a representation of $Q$ with dimension vector $d_V$, and let $\theta$ be a character of $\widetilde{\GL}(d)$ with $\theta(d_V)=0$. We say that $V$ is \emph{$\theta$-(semi)stable} if for all subrepresentations $W \subset V$ with dimension vector $d_W$, the inequality $\theta(d_W) < (\leq) 0$ holds. 
\end{defn}

The key is that GIT-stability matches with this notion.

\begin{prop}[{\cite[Proposition 3.1]{Kin94:moduli_quiver_reps}}]
Let $\OO_{\RR(Q, d)}$ be linearized with the character $-\theta$. Then a representation $V \in \RR(Q, d)$ is $\theta$-(semi)stable if and only if it is GIT-(semi)stable.
\end{prop}

We will mostly deal with quivers with relations. The relations define a closed subvariety of $\RR(Q, d)$ whose quotient we can take instead. The previous proposition holds without any differences in this slightly modified set-up.

\subsection{Relation between quivers and Bridgeland stability}
\label{subsec:quiver_to_stable_complexes}

We finish this section by explaining the relation between moduli spaces of quiver representations and Bridgeland stability on $\P^2$. Let $E = \OO(-1) \oplus \OO \oplus \OO(1)$. Then we have a functor 
\[
R\Hom(E, \cdot): \Db(\P^2) \to \Db(\mod-\End(E)).
\]
A combination of \cite{Bei78:exceptional_collection_pn} and \cite{Bon89:derived_equivalence_quivers} shows that this functor is an equivalence of triangulated categories. This is all based on the theory of full strong exceptional collections. We also refer to \cite{Huy06:fm_transforms} for the unfamiliar reader. Moreover, $\mod-\End(E)$ is equivalent to the category of representations of the quiver $Q$

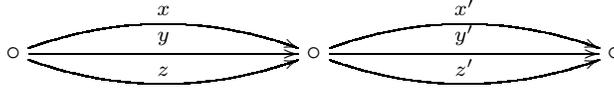
\centerline{
\xygraph{
!{<0cm,0cm>;<1cm,0cm>:<0cm,1cm>::}
!{(0,0) }*+{\circ}="1"
!{(4,0) }*+{\circ}="2"
!{(8,0) }*+{\circ}="3"
"1"-@{>}@/_0.4cm/^z"2"
"1"-@{>}^y"2"
"1"-@{>}@/^0.4cm/^x"2"
"2"-@{>}@/_0.4cm/^{z'}"3"
"2"-@{>}^{y'}"3"
"2"-@{>}@/^0.4cm/^{x'}"3"
}} \ \\

together with the three commutation relations $y'x = x'y$, $z'x = x'z$, and $z'y = y'z$. If $F \in \Db(\P^2)$, then the corresponding element of $\Db(\Rep-Q)$ under these equivalences can be computed as follows. As complexes of vector spaces we get
\[
    \RHom(E, F) = \RHom(\OO(1), F) \oplus \RHom(\OO, F) \oplus \RHom(\OO(-1), F)).
\]
The algebra $\End(E)$ is given by
\[
    \C^{\oplus 3} \oplus \Hom(\OO(-1), \OO) \oplus \Hom(\OO, \OO(1)) \oplus \Hom(\OO(-1), \OO(1).
\]
The action of $\End(E)$ on $\RHom(E, F)$ is given as follows. The three copies of $\C$ each act via scalar multiples of the identity. The morphisms in $\Hom(\OO(-1), \OO)$ induce maps $\RHom(\OO, F) \to \RHom(\OO(-1), F)$ and morphisms in $\Hom(\OO, \OO(1))$ induce maps $\RHom(\OO(1), F) \to \RHom(\OO, F)$. Finally, all morphisms in $\Hom(\OO(-1), \OO(1))$ can be written as a combination of compositions of morphism from $\Hom(\OO(-1), \OO)$ and $\Hom(\OO, \OO(1))$. Therefore, the action is completely determined by understanding the action of $\Hom(\OO(-1), \OO) \oplus \Hom(\OO, \OO(1))$, i.e., the action of $x$, $y$, and $z$.

Multiplication by $x$, $y$, and $z$ induces morphisms $x, y, z: \RHom(\OO(1), F) \to \RHom(\OO, F)$ and $x', y', z': \RHom(\OO, F) \to \RHom(\OO(-1), F)$. This induces a complex in $\Db(\Rep-Q)$. Note that if $\Ext^i(\OO(-1) \oplus \OO \oplus \OO(1), F) = 0$ for $i \neq 0$, then we simply obtain a quiver representation.

In order to avoid writing down six matrices for each representation, we will use the following notation. Assume we have a representation given by

\centerline{
\xygraph{
!{<0cm,0cm>;<1cm,0cm>:<0cm,1cm>::}
!{(0,0) }*+{
\C^a 
}="1"
!{(4,0) }*+{\C^b}="2"
!{(8,0) }*+{\C^c}="3"
"1"-@{>}@/_0.4cm/^{A_x}"2"
"1"-@{>}^{A_y}"2"
"1"-@{>}@/^0.4cm/^{A_z}"2"
"2"-@{>}@/_0.4cm/^{B_{z'}}"3"
"2"-@{>}^{B_{y'}}"3"
"2"-@{>}@/^0.4cm/^{B_{x'}}"3"
}} \ \\
where $\C^a = \RHom(\mathcal{O}(1), F)$, $\C^b = \RHom(\mathcal{O}, F)$, and $\C^c = \RHom(\mathcal{O}(-1), F)$.
Then the representation is determined by the matrices $A(F) = xA_x + yA_y + zA_z$ and $B(F) = x'B_{x'} + y'B_{y'} + z'B_{z'}$ with linear entries in $\C[x,y,z]$, respectively $\C[x',y',z']$. It should be understood that these matrices are only well defined up to base change.

\begin{ex}
\label{ex:matrices}
Let $\II_Z = (x, y^2)$ and $F = \II_Z(-1)[1]$. A direct calculation shows that 
\begin{align*}
    \RHom(\OO(-1), \II_Z(-1)[1]) &= \C, \\
    \RHom(\OO, \II_Z(-1)[1]) &= H^0(\OO_Z) = \C^2, \\
    \RHom(\OO(1), \II_Z(-1)[1]) &= H^0(\OO_Z) = \C^2.
\end{align*}
Therefore, we have a representation with dimension vector $(2,2,1)$.  Let $\{1, \bar{y}\}$ be a basis of $H^0(\OO_Z)$. Then, multiplication by $z$ induces the identity map, multiplication by $y$ induces the map $1 \mapsto \bar{y}$ and $\bar{y} \mapsto 0$, and multiplication by $x$ is trivial. We obtain 
\[
    A(F) = \begin{pmatrix}
    z & 0 \\ 
    y & z
    \end{pmatrix}.
\]
To calculate the matrix $B(F)$, we observe that $\Hom(\OO(-1), \II_Z(-1)[1])$ is the cokernel of the evaluation morphism $H^0(\OO) \to H^0(\OO_Z)$. This means that $\Hom(\OO(-1), \II_Z(-1)[1])$ is the quotient of $H^0(\OO_Z)$ by $\C \cdot \bar{1}$ and the matrix is
\[
    B(F) = \begin{pmatrix}
        y & z
    \end{pmatrix}.
\]
\end{ex}

To finish this section, we explain the identification of moduli spaces of Bridgeland stable objects on $\P^2$ with moduli spaces of certain quiver representation, see \cite[Proposition 7.5]{ABCH13:hilbert_schemes_p2}. We will use that the equivalence $\RHom(E, \cdot)$ identifies
\[
\Rep-Q = \langle \OO(-2)[2], \Omega[1], \OO(-1) \rangle,
\]
where the brackets denote the smallest additive full subcategory containing these three objects and all extensions (but not shifts). This is based on the fact that $\OO(-2)$, $\Omega$, $\OO(-1)$ is the so-called dual exceptional collection to $\OO(-1)$, $\OO$, $\OO(1)$. We refer to \cite[Section 3]{Bri05:t-structures_local_cy} for more details.

Note that we are not using the same exceptional collection as in \cite{ABCH13:hilbert_schemes_p2}, but a slightly different one from \cite{LZ18:stability_p2}.

\begin{prop}
\label{prop:quiver_moduli_iso}
Choose $\alpha$, $\beta$ such that 
$
\alpha^2 + \left(\beta - \frac{3}{2}\right)^2 < \frac{1}{4}.
$
We can define a torsion pair
\begin{align*}
    \TT_{\gamma} &\coloneqq \{E \in \Coh^{\beta}(\P^2) : \text{any quotient $E \onto G$ satisfies $\nu_{\alpha, \beta}(G) > \gamma$} \}, \\
    \FF_{\gamma} &\coloneqq  \{E \in \Coh^{\beta}(\P^2) : \text{any subobject $0 \neq F \into E$ satisfies $\nu_{\alpha, \beta}(F) \leq \gamma$} \}.
\end{align*}
There is a choice of $\gamma \in \R$ such that
\[
    \langle \TT_{\gamma}, \FF_{\gamma}[1] \rangle = \langle \OO(-2)[2], \Omega[1], \OO(-1) \rangle.
\]
In particular, moduli spaces of Bridgeland stable objects for these choices of stability conditions are the same as moduli spaces of representations of finite-dimensional algebras as defined in \cite{Kin94:moduli_quiver_reps}. This means they have projective good moduli spaces.
\end{prop}

Later, we will mostly be interested in the case of ideal sheaves of points or other Bridgeland-semistable objects with the same numerical invariants. In that case, the relevant dimension vector is $(n, n - 1, n - 3)$ as will be shown in Section \ref{subsec:bridgeland_stable_git_unstable}. This will be the only point in which we actually use the dual exceptional collection $\OO(-2)$, $\Omega$, $\OO(-1)$. Otherwise, the matrices will suffice.

\section{Set up and general stability results.}

\subsection{The setup}

The group $G = \SL_3$ acts on the Hilbert scheme $\P^{2[n]}$ via its action on $\P^2$. The usual action of $\SL_3$ on $\C[x, y, z]$ is given by $(A \cdot f) (\mathbf{x}) = f(A^{-1}\cdot \mathbf{x})$ with $\mathbf{x} = (x,y,z)$. In order to avoid keeping inverses everywhere, we act via the transpose instead
\[
(A \cdot f) (\mathbf{x}) \coloneqq f(A^t \cdot \mathbf{x}).
\]
For consistency, we act on points via $(A^{-1})^t$. Note that the transpose does not really matter because the vast majority of computations in this article are with diagonal one-parameter subgroups. Since $A \mapsto (A^{-1})^t$ is an automorphism of $\SL_3$, this does not affect the quotients either.

If $D_m$ is an ample divisor on $\P^{2[n]}$, then we define
\[
    \textbf{C}(m, n) \coloneqq \P^{2[n]} /\!\!/_{D_m}\SL_3.
\]
Our goal is to obtain a GIT quotient for any divisor in the movable cone. Usually, GIT is done with ample divisors only. By Theorem \ref{thm:moduli_spaces_tilt_stability} we know that $\P^{2[n]}$ is a Mori dream space. This means that for any divisor $D_m$ in the interior of the movable cone there is a birational model on which $D_m$ is actually ample. Essentially, the same picture holds on the boundary of the movable cone except that this model might not be birational, but of lower dimension. The appropriate definition of $\textbf{C}(m,n)$ is to quotient this model instead of the Hilbert scheme itself.

\begin{defn}
\label{def:spaces}
For any movable divisor $D_m$, let $(\Coh^{\beta_m}(\P^2), \nu_{\alpha_m, \beta_m})$ be the Bridgeland stability condition that induces $D_m$ via the Positivity Lemma. Then we define
\[
    \textbf{C}(m,n) \coloneqq M_{\alpha_m, \beta_m}(1, 0, -n) /\!\!/_{D_m} \SL_3.
\]
\end{defn}

This definition works out since the action of $\SL_3$ on $\P^2$ induces a natural action on these moduli spaces of Bridgeland-semistable objects. Moreover, Theorem \ref{thm:RelativeGIT} shows that these quotients behave as expected when birational models change at flips or the boundary of the movable cone.

A well-known special case is given by the quotient of the symmetric product
\[
    \textbf{C}(\infty,n) = \mathbb{P}^{2(n)}/ \!\!/ \SL_3.
\]

\begin{theorem}[{\cite[Proposition 7.27]{Muk03:introduction_invariants_moduli}}]
\label{thm:stability_chow}
A point $\Lambda \in \P^{2(n)}$ is GIT-semistable if and only if at most $\tfrac{2n}{3}$ points are collinear and any point occurs with multiplicity at most $\tfrac{n}{3}$. It is GIT-stable if and only if these inequalities are strict.
\end{theorem}

In particular, note that the existence of strictly GIT-semistable points implies $3 \mid n$.  By $\lambda(a, b) \coloneqq (a, b, -a - b)$ we denote the diagonal one-parameter subgroup
\[
    \lambda(t) = \begin{pmatrix}
    t^{a} & 0 & 0 \\
    0 & t^{b} & 0 \\
    0 & 0 & t^{-a - b}
    \end{pmatrix}.
\]
Since the numerical criterion is independent of taking positive multiples of $\lambda$, we will include the cases $a, b \in \Q$, even though that is not a one-parameter subgroup in a strict sense. For $r \in \Q$, we write $\lambda_r \coloneqq  (1, r, -1-r)$, and if $r \in [-\tfrac{1}{2}, 1]$, we call these one-parameter subgroups \emph{normalized}.

\begin{lem}
Up to change of coordinates and up to positive multiples any non-trivial one-parameter subgroup of $\SL(3)$ is of the form $\lambda_r$ for $r \in [-\tfrac{1}{2}, 1]$. 
\end{lem}

\begin{proof}
Let $\lambda = (q, r, s)$ be a diagonal one-parameter subgroup. Since we are in $\SL(3)$, we have $q + r + s = 0$. Up to a change of coordinates, we can assume $q \geq r \geq s$. Since $\lambda$ is non-trivial, this implies $q > 0$, and we can divide by $q$ to reduce to $q = 1$. Then $s = -1 - r$. The fact $r \in [-\tfrac{1}{2}, 1]$ follows from $q = 1 \geq r \geq s =-1 - r$.
\end{proof}

Let $Z \subset \P^2$ be a subscheme of dimension zero and $\lambda$ be a one-parameter subgroup in $\SL_3$. To simplify notation, we will from now on write $\mu_m(Z, \lambda) \coloneqq  \mu^{D_m}(Z, \lambda)$. We will also denote $D_m$-(semi)stable subschemes $Z$ as $m$-(semi)stable. 

\begin{lem}
\label{lemma:m0}
Let $d$ be any fixed rational number. If $\mu_d(Z, \lambda) \neq \mu_{d+1}(Z, \lambda)$, then
\[
    m_0(Z, \lambda) = d + \frac{\mu_d(Z, \lambda)}{\mu_d(Z, \lambda) - \mu_{d+1}(Z, \lambda)}
\]
is the unique $m$ such that $\mu_m(Z, \lambda) = 0$. If $\mu_d(Z, \lambda) = \mu_{d+1}(Z, \lambda) \neq 0$, then $m_0(Z, \lambda) = \infty$.
\end{lem}

\begin{proof}
If $\mu_d(Z, \lambda) \neq \mu_{d+1}(Z, \lambda)$, then
\[
    D_{m_0} = m_0 H - \frac{\Delta}{2} = c_1D_d + c_2 D_{d+1} = (c_1 + c_2)d H + c_2 H - (c_1 + c_2) \frac{\Delta}{2},
\]
and $c_1$ and $c_2$ must satisfy $c_1 + c_2 = 1$. On the other hand,
\[
    0 = \mu_{m_0}(Z, \lambda) = c_1\mu_{d}(Z, \lambda) + c_2 \mu_{d+1}(Z, \lambda) = \mu_{d}(Z, \lambda) - c_2\left( \mu_{d}(Z, \lambda) - \mu_{d+1}(Z, \lambda) \right)
\]
implies 
\[
    c_2 = \frac{\mu_{d}(Z, \lambda)}{\mu_{d}(Z, \lambda) - \mu_{d+1}(Z, \lambda)}.
\]
We obtain
\[ 
    D_{m_0} = c_1 D_{d} + c_2 D_{d+1} = \left(d + \frac{\mu_{d}(Z, \lambda)}{ \mu_{d}(Z, \lambda) - \mu_{d+1}(Z, \lambda)} \right) H -\frac{\Delta}{2}.
\]
If $\mu_d(Z, \lambda) = \mu_{d+1}(Z, \lambda) \neq 0$, then the result follows from $D_{d+1} - D_d = H$.
\end{proof}

Checking diagonal one-parameter subgroups in the numerical criterion will usually be simple in our setting. Of course, every one-parameter subgroup can be normalized to a diagonal one-parameter subgroup $\lambda_r$. However, when working with a single ideal, the change of basis will also change the ideal. The following proposition will allow us to reduce the amount of tori from which one-parameter subgroups have to be checked. We were inspired by \cite[Proposition 2.4]{AFS13:hilbert_stability_bicanonical}

\begin{prop}
\label{prop:tori_to_check}
Let $G$ be acting on $X$ and for $x \in X$ let $T'$ be a torus of $\Stab_G(x)$. Let $B \subset G$ be a Borel subgroup containing $T'$. If $x$ is $\LL$-unstable, then there is a maximal torus $T \subset B$ containing $T'$ and a one-parameter subgroup $\lambda$ in $T$ such that $\mu^\LL(x, \lambda) > 0$.
\end{prop}

\begin{proof}
By Theorem 3.4 in \cite{Kem78:instability} there is a parabolic subgroup $P \subset G$ such that any maximal torus in $P$ contains a destabilizing one-parameter subgroup. Moreover, \cite[Corollary 3.5]{Kem78:instability} says $\Stab_G(x) \subset P$. By \cite[Corollary 28.3]{Hum75:linear_algebraic_groups} the intersection of any two Borel subgroups of $G$ contains a maximal torus of $G$. In particular, the intersection of $P$ and $B$ contains a maximal torus of $G$. Since $T' \subset B \cap P$ we can extend it to a maximal torus $T \subset B \cap P$ of $G$.
\end{proof}

We will use Proposition \ref{prop:tori_to_check} in the following special cases. 

\begin{rmk}
\label{rmk:tori_to_check}
Let $Z \in \P^{2[n]}$ be a subscheme of dimension zero, and let $T$ be the diagonal torus.
\begin{enumerate}
    \item If $\Stab_{\SL_3}(Z)$ contains a one-parameter subgroup $\lambda = (a, b, c)$ for pairwise different $a, b, c$, then GIT-semistability of $Z$ can be checked only with one-parameter subgroups in $T$. Indeed, the only torus containing $\lambda$ is the diagonal one.
    
    \item If $\Stab_{\SL_3}(Z)$ contains the one-parameter subgroup $(1, 1, -2)$, then GIT-semistability of $Z$ can be checked only with one-parameter subgroups in tori of the form $g^{-1} T g$ for
    \[
    g = \begin{pmatrix}
        1 & a & 0 \\
        0 & 1 & 0 \\
        0 & 0 & 1
    \end{pmatrix}
    \]
    for some $a \in \C$. Indeed, we can choose $B \subset \SL_3$ to be the Borel subgroup of upper-triangular matrices, and $T'$ the image of $\lambda_1$. Then any maximal torus of $B$ containing $T'$ is of the above form.
    
    \item If $\Stab_{\SL_3}(Z)$ contains the one-parameter subgroup $(2, -1, -1)$, then GIT-semistability of $Z$ can be checked only with one-parameter subgroups in tori of the form $g^{-1} T g$ for
    \[
    g = \begin{pmatrix}
        1 & 0 & 0 \\
        0 & 1 & a \\
        0 & 0 & 1
    \end{pmatrix}.
    \]
\end{enumerate}
\end{rmk}

In order to compute instability of certain ideals $\II$, we will use the definition of the \emph{state polytope} $\State_m(\II)$. We refer to 
\cite[Section 9.4]{Dol03:invariant_theory} where it is called the weight polytope. Our definition is equivalent, but serves our purposes better.

\begin{defn}
Let $\II \subset \C[x,y,z]$ be a homogeneous ideal, let $m \in \Z$, and let $T \subset \SL_3(\C)$ be the maximal torus that is diagonal with respect to the coordinates $x, y, z$. Then $H^0(\II(m)) \subset H^0(\OO(m))$ is a subvectorspace, and the Pl\"ucker embedding induces a point $x_{\II} \in \P^N$ for appropriate $N$. The torus $T$ acts on $\P^N$ and this action lifts to $\C^{N+1}$. The \emph{state polytope} of $\II$, denoted by $\State_m(\II)$, is the convex hull of the characters of the action of $T$ on a lift of $x_{\II}$ to $\C^{N+1}$.
\end{defn}

The following lemma is an immediate consequence of the definition of state polytopes.

\begin{lem}
\label{lem:numericalFunction}
Let $Z \in \P^{2[n]}$ and $m \in \Z$ such that the $m$-th Hilbert point of $Z$ is well-defined. Then for any one-parameter subgroup $\lambda = (a, b, c)$ we have
\[
\mu_m(Z, \lambda) = \min \{ ai + bj + ck : (i, j, k) \in \State_m(\II_Z) \}.
\]
\end{lem}

\subsection{Bridgeland-semistable objects}
\label{subsec:bridgeland_stable_git_unstable}

The goal of this section is to prove Theorem \ref{thm:git_unstable_bridgeland_stable} that says that certain Bridgeland semistable complexes will never be GIT-semistable for any $m$. Our strategy is to study them via quiver representations. A description of the relevant Bridgeland-semistable objects will be given in Proposition \ref{prop:wall_between_-1_-2}. Their GIT-unstability will then be established in Propositions \ref{prop:quiver_locus1}, \ref{prop:quiver_locus2}, and \ref{prop:quiver_locus3}.

The strategy to show that a Bridgeland-semistable object $E$ is GIT-unstable repeats as follows:
\begin{enumerate}
    \item Determine the matrices $A(E[1])$ and $B(E[1])$ that describe $E[1]$ as a quiver representation in $\Rep-Q$ where $Q$ is the quiver described in Section \ref{subsec:quiver_to_stable_complexes} (see Example \ref{ex:matrices}). When $E[1]$ is an extension between two simpler objects, we will understand the quiver representations of these simpler objects first and then determine all possible extensions that still satisfy the relations of $Q$.
    \item Use these matrices to determine the Hilbert-Mumford index for some well chosen one-parameter subgroup $\lambda$, the point corresponding to $E$ in the moduli space, and two different movable divisors.
    \item Use linearity of the Hilbert-Mumford index to show that $\lambda$ destabilizes $E$ where claimed.
\end{enumerate}

Bridgeland stability on $\P^2$ was studied in various works such as \cite{ABCH13:hilbert_schemes_p2, CHW17:effective_cones_p2, LZ18:stability_p2}. We need some specific examples that are not explicitly spelled out in those articles. For the convenience of the readers, we will prove the following statement.

\begin{prop}
\label{prop:wall_between_-1_-2}
Let $n \in \Z$ with $n \geq 5$ and let $E \in \Db(\mathbb{P}^2)$ with $\ch(E) = (1, 0, -n)$.
\begin{enumerate}
\item The walls for objects with Chern character $(1, 0, -n)$ that are strictly larger than the wall $W(E, \OO(-2))$ are given by $W_l \coloneqq W(E, \II_Y(-1))$ where $Y \subset \P^2$ is a zero-dimensional subscheme of length $n - l$ for $\tfrac{n + 1}{2} < l \leq n$. The divisor induced along $W_l$ is given by $D_{l - 1}$. If $E$ is strictly semistable along $W_l$, then it is an extension between $\II_Y(-1)$ and $\OO_L(-l)$ for a line $L \subset \P^2$. Such an extension $E$ is stable above the wall if it is a non-trivial extension
\[
0 \to \II_Y(-1) \to E \to \OO_L(-l) \to 0.
\]
Such an extension $E$ is stable below the wall if it is a non-trivial extension
\[
0 \to \OO_L(-l) \to E \to \II_Y(-1) \to 0.
\]

\item Let $W_{(n + 1)/2} \coloneqq W(E, \OO(-2))$. In the special case where $n$ is odd, we have $W_{(n + 1)/2} = W(E, \II_{\widetilde{Y}}(-1))$ where $\widetilde{Y} \subset \P^2$ is a zero-dimensional subscheme of length $\tfrac{n - 1}{2} = n - \tfrac{n + 1}{2}$. The subschemes $Z \subset \P^2$ such that $E = \II_Z$ is strictly semistable along this wall are precisely those where either
\begin{enumerate}
    \item $Z$ is completely contained in a conic $C \subset \P^2$ (possibly singular) and no more than $\tfrac{n + 1}{2}$ points lie on a single line, or
    \item there is a line $L \subset \P^2$ containing exactly $\tfrac{n + 1}{2}$ points of $Z$ (only if $n$ is odd).
\end{enumerate}
The divisor corresponding to this wall is $D_{\tfrac{n - 1}{2}}$.
\end{enumerate}
\end{prop}

\begin{proof}
The wall $W(E, \OO(-2))$ has center $s = -\tfrac{n}{2} - 1$ and radius $\rho^2 = \tfrac{n^2}{4} - n + 1$. Since $n \geq 5$,
\[
\frac{n^2}{4} - n + 1 > \frac{n}{4} = \frac{\Delta(E)}{8}
\]
and by Proposition \ref{prop:structure_sequences} we only have to deal with destabilizing subobjects or quotients of rank one. Assume that $E$ is destabilized via a short exact sequence
$0 \to F \to E \to G \to 0$
with $\ch_0(F) = 1$. The case where $G$ has rank one is completely analogous and we will not spell it out.

To prove (i) assume that $W(E, F)$ is larger than $W(E, \OO(-2))$. The closure of $W(E, \OO(-2))$ contains the two points $(0, -2)$ and $(0, -n)$. Therefore, the short exact sequence is in $\Coh^{\beta}(\P^2)$ for both $\beta = -2$ and $\beta = -n$. We get $\ch_1^{-2}(F) = \ch_1(F) + 2 > 0$ and $\ch_1^{-n}(F) = \ch_1(F) + n < \ch_1^{-n}(E) = n$, i.e., $\ch_1(F)  = -1$. If $l \coloneqq n + \ch_2(F \otimes \OO(1))$, then $\ch(F) = (1, 0, -n + l) \otimes \ch(\OO(-1))$. This means $F = \II_Y(-1)$ for a zero-dimensional subscheme 
$Y \subset \P^2$ of length $n - l$. The fact that $W(E, F)$ is strictly larger than $W(E, \OO(-2))$ is equivalent to $l > \tfrac{n + 1}{2}$. The quotient satisfies $\ch(G) = (0, 1, -l - \tfrac{1}{2})$ and this means $G = \OO_L(-l)$ for a line $L \subset \P^2$.

The center of $W(E, F)$ is $s = -l - \tfrac{1}{2}$ and this implies the description of the corresponding divisor. The description of which objects are stable above and below the wall follows from Proposition \ref{prop:structure_sequences}.

The fact that the walls are identical is a straightforward computation and the description of the corresponding divisor follows from the same calculation. Assume that $E = \II_Z$ for a subscheme $Z \subset \P^2$ is strictly semistable along this wall. If $\ch_1(F) = -1$, then the same argument as in (i) leads to $F = \II_{\widetilde{Y}}(-1)$ and $G = \OO_L(-\tfrac{n + 1}{2})$. However, the short exact sequence
\[
0 \to \II_{\widetilde{Y}}(-1) \to \II_Z \to \OO_L\left(-\frac{n + 1}{2}\right) \to 0
\]
implies that $Z$ must contain $\tfrac{n + 1}{2}$ points in $L$.

Since the closure of $W(E, \OO(-2))$ contains the points  $(0, -2)$ and $(0, -n)$, we get $\ch_1^{-2}(F) = \ch_1(F) + 2 \geq 0$, and $\ch_1^{-n}(F) = \ch_1(F) + n \leq \ch_1^{-n}(E) = n$. We can rule out $\ch_1(F) = 0$ since that would lead to the vertical wall. We already dealt with $\ch_1(F) = -1$ and thus, are only left with $\ch_1(F) = -2$. Bogomolov's inequality says $\ch_2(F) \leq 2$, but if $\ch_2(F) < 2$, then we are obtaining a smaller wall, i.e., $\ch_2(F) = 2$, and we must have $F = \OO(-2)$. The existence of a non-trivial morphism $\OO(-2) \into \II_Z$ means that $Z$ is contained in a conic. The only way that $\II_Z$ would not be strictly-semistable along this wall is if it was already destabilized at one of the walls $W_l$ we described in (i). But that is equivalent to more than $\tfrac{n + 1}{2}$ points lying on a line.
\end{proof}

Next, we translate what the contents of this proposition imply about ideal sheaves:
\begin{cor}
\label{cor:ideals_between_-1_-2}
Let $Z$ be a zero-dimensional subscheme $Z \subset \P^2$ of length $n$ with ideal sheaf $\II_Z$,  
and let $(\Coh^{\beta_m}(\P^2), \nu_{\alpha_m, \beta_m})$ be a stability condition whose induced divisor is a positive multiple of $D_m$.
\begin{enumerate}
    \item Assume that $n - 1 < m < \infty$. Then any $\II_Z$ is $\nu_{\alpha_m, \beta_m}$-stable.
    \item Assume that $l - 2 < m < l - 1$ for an integer $l$ with $\tfrac{n + 3}{2} \leq l \leq n$. Then $\II_Z$ is $\nu_{\alpha_m, \beta_m}$-stable if and only if $Z$ does not contain a collinear subscheme of length greater than or equal to $l$.
    \item Assume that $n$ is even and $\tfrac{n - 1}{2} < m < \tfrac{n}{2}$. Then $\II_Z$ is $\nu_{\alpha_m, \beta_m}$-stable if and only if $Z$ does not contain a collinear subscheme of length greater than or equal to $\tfrac{n}{2} + 1$.
    \item 
    Assume that $\tfrac{n - 1}{2} \leq m < \tfrac{n - 1}{2}$. Then, the ideal $\II_Z$ is $\nu_{\alpha_m, \beta_m}$-stable if and only if both of the following conditions hold: $Z$ is not contained in a conic, and $Z$ does not contain a collinear subscheme of length greater than or equal to $\tfrac{n + 1}{2}$.
\end{enumerate}
\end{cor}

We need the following crude estimate for an extremal ray of the effective cone of $\P^{2[n]}$:

\begin{lem}
\label{lem:estimate_effective_cone}
The divisor $D_2$ is outside the effective cone for $n \geq 7$. If $n = 5$ or $n = 6$, then $D_2$ is on the boundary of the effective cone.
\end{lem}

\begin{proof}
At the end of Section \ref{subsec:mds_hilbert_scheme} we recalled that the induced divisor along a wall with radius $s$ is given by $D_{-s-3/2}$. If $D_2$ is in the effective cone, then there must be a numerical wall in Bridgeland stability for objects with Chern character $(1, 0, -n)$ with center $s = -\tfrac{7}{2}$ along which there are semistable objects. The equation $\nu_{0, \beta}(1, 0, -n) = 0$ is equivalent to $\beta = \pm \sqrt{2n}$. This means all walls for objects with Chern character $(1, 0, -n)$ have center $s < -\sqrt{2n}$. This immediately rules out $s = -\tfrac{7}{2}$ for $n \geq 2$. The cases $n = 5$ and $n = 6$ were treated as examples in \cite{ABCH13:hilbert_schemes_p2}.
\end{proof}

Most walls described in Proposition \ref{prop:wall_between_-1_-2} are associated to configurations of points supported on a line $L \subset \mathbb{P}^2$. Therefore, we start by describing related matrices for our chosen quiver between $\beta = -2$ and $\beta = -1$. Recall Example \ref{ex:matrices} for how to compute these matrices. 
\begin{lem}
\label{lem:oo_L_quiver}
Let $L$ be the line cut out by $z = 0$, and fix an integer $n \geq 2$. Then $\OO_L(-n)[1]$ is in $\Rep-Q$ with dimension vector $(n, n - 1, n - 2)$ and the matrices are given by
\[
    A(\OO_L(-n)[1]) = \begin{pmatrix}
    x      & y      & 0      & 0      & \dots  &  0      \\ 
    0      & x      & y      & 0      & \dots  &  0      \\
    0      & 0      & x      & y      & \dots  &  0      \\
    \vdots & \vdots & \vdots & \ddots & \ddots &  \vdots \\
    0      & 0      & \dots  & 0      & x      &  y      \\
    \end{pmatrix}, \
    B(\OO_L(-n)[1]) = \begin{pmatrix}
    x'      & y'      & 0       & \dots   &  0      \\ 
    0       & x'      & y'      & \dots   &  0      \\
    \vdots  & \vdots  & \ddots  & \ddots  &  \vdots \\
    0       & 0       & \dots   & x'      &  y'
    \end{pmatrix},
\]
\end{lem}

\begin{proof}
Using Serre duality on $L \cong \P^1$, we get 
\begin{align*}
    \RHom(\OO(-1), \OO_L(-n)[1]) &= H^0(\OO_L(n-3))^{\vee}, \\
    \RHom(\OO, \OO_L(-n)[1]) &= H^0(\OO_L(n-2))^{\vee}, \\
    \RHom(\OO(1), \OO_L(-n)[1]) &= H^0(\OO_L(n-1))^{\vee}.
\end{align*}
We choose the basis $x^k, x^{k-1}y, \ldots, y^k$ on $H^0(\OO_L(k))$ for any positive integer $k$. Then the multiplication morphisms $x, y: H^0(\OO_L(k)) \to H^0(\OO_L(k + 1)$ are given by the matrices
\[
    M_{x, k} = \begin{pmatrix}
    1      & 0      & \dots  & 0      \\ 
    0      & 1      & \dots  & 0      \\
    \vdots & \vdots & \ddots & \vdots \\
    0      & 0      & \dots  & 1      \\
    0      & 0      & \dots  & 0
    \end{pmatrix}, \
    M_{y, k} = \begin{pmatrix}
    0      & 0      & \dots  &  0     \\ 
    1      & 0      & \dots  &  0     \\
    0      & 1      & \dots  &  0     \\
    \vdots & \vdots & \ddots & \vdots \\
    0      & 0      & \dots  &  1      \\
    \end{pmatrix}.
\]
Since $L$ is cut out by $z = 0$, the matrix $M_{z, k}$ corresponding to multiplication by $z$ is trivial. On the vector spaces $H^0(\OO_L(k))^{\vee}$ we choose the dual basis. Then 
\begin{align*}
    A(\OO_L(-n)[1]) &= x M_{x, n - 2}^{t} + y M_{y, n - 2}^{t} + z M_{y, n - 2}^{t}, \\
    B(\OO_L(-n)[1]) &= x M_{x, n - 3}^{t} + y M_{y, n - 3}^{t} + z M_{y, n - 3}^{t}
\end{align*}
as claimed.
\end{proof}

The next step is to understand the characters $\theta$ corresponding to the walls $W(\OO(-1), (1,0,-n))$ and $W(\OO(-2), (1,0,-n))$. The divisors corresponding to these two walls will be those for which we will compute the Hilbert-Mumford index. Let $n \geq 4$. Since
\[
-(1, 0, -n) = n \ch(\OO(-2)[2]) + (n - 1) \ch(\Omega[1]) + (n-3) \ch(\OO(-1)),
\]
the corresponding dimension vector of our ideal sheaves in $\Rep-Q$ is $(n, n - 1, n - 3)$. This also means that it is not possible for $\II_Z$ to be in $\Rep-Q$, but rather $\II_Z[1]$.

\begin{enumerate}
    \item In King's quiver stability in $\Rep-Q$ the wall $W(\OO(-1), (1,0,-n))$ is given by $\theta(0,0,1) = 0$. Due to $\theta(n, n - 1, n - 3) = 0$, we must have either $\theta = (n - 1, -n, 0)$ or $\theta = (-n + 1, n, 0)$. For $\theta = (-n + 1, n, 0)$ any quotient $\II_Z[1] \to \OO(-2)[2]$ destabilizes $\II_Z[1]$, a contradiction. Thus, $\theta_1 = (n - 1, -n, 0)$ is the corresponding character.
    \item The wall $W(\OO(-2), (1,0,-n))$ corresponds to $\theta(1,0,0) = 0$. Due to $\theta(n, n - 1, n - 3) = 0$ and the fact that positive scaling does not play a role, we must have either $\theta = (0, -n + 3, n - 1)$ or $\theta = (0, n - 3, -n + 1)$. For $\theta = (0, -n + 3, n - 1)$ any subobject $\OO(-1) \into \II_Z[1]$ destabilizes $\II_Z[1]$, a contradiction. Thus, $\theta_2 = (0, n - 3, -n + 1)$ is the corresponding character.
\end{enumerate}

\begin{prop}
\label{prop:quiver_locus1}
Any Bridgeland-stable extension
\[
0 \to \OO_L(-n) \to E \to \OO(-1) \to 0
\]
where $L \subset \P^2$ is a line is GIT-unstable for $n \geq 4$ and any polarization.
\end{prop}

\begin{proof}
By Proposition \ref{prop:quiver_moduli_iso} all relevant moduli space between $W(\OO(-1), E)$ and $W(\OO(-2), E)$ are moduli space of representations in $\Rep-Q$. If we can show that our objects $E$ are only Bridgeland-stable in that region, we can forget about the derived category and only deal with quiver representations.

These objects $E$ can only be Bridgeland-stable below the wall $W(\OO(-1), E)$. We will prove that they are unstable in Bridgeland stability below $W(\OO(-2), E)$. This follows if $\Hom(\OO(-2), E) \neq 0$. We have an exact sequence
\[
0 \to \Hom(\OO(-2), E) \to \left(\Hom(\OO(-2), \OO(-1)) = \langle x, y, z \rangle \right) \to \Hom(\OO(-2), \OO_L(-n)[1]). 
\]
Since by assumption $L$ is cut out by $z$, the last morphism is never injective, since at least $z$ maps to $0$. Therefore, $\Hom(\OO(-2), E) \neq 0$.

Using the action of $\SL_3$ allows us to reduce to the case where $L$ is cut out by $z = 0$. In the category $\Rep-Q$ we have a short exact sequence
\[
0 \to \OO(-1) \to \OO_L(-n)[1] \to E[1] \to 0.
\]
Therefore, $A(E[1]) = A(\OO_L(-n)[1])$ and $B(E[1]) = T \cdot B(\OO_L(-n)[1])$, where $T: \C^{n-2} \to \C^{n-3}$ is some surjective linear map. The key observation here is that both $A(E[1])$ and $B(E[1])$ are matrices with entries in $\C[x, y]$, but no $z$ occurs in either of them.

We will show that $E$ is destabilized by the diagonal one-parameter subgroup $\lambda = (-1,-1,2)$. We have
\[
\Ext^1(\OO(-1), \OO_L(-n)) = H^0(\OO_L(n - 3))^{\vee} = \langle x^{n-3}, x^{n-4}y, \ldots, y^{n-3} \rangle^{\vee}.
\]
Therefore, $\lambda$ fixes $E$. The action of $\lambda$ does not fix the matrices $A(E[1])$ and $B(E[1])$, but the one-parameter subgroup 
\[
\lambda' = ((-1, \ldots, -1), (0, \ldots, 0), (1, \ldots, 1),(-1,-1,2)) \in \GL_n \times \GL_{n-1} \times \GL_{n-3} \times \SL_3
\]
does.

For any character $\theta$ in between $\theta_1$ and $\theta_2$ we denote the induced line bundle on the moduli space of $\theta$-semistable representations by $L^{\theta}$. To understand $\mu^{L^{\theta}}([E], \lambda)$ we need to understand the action of $\lambda$ on the fiber $L^{\theta}_{[E]}$.  By our sign convention, $L^{\theta}$ is induced by the linearization given by $-\theta$ on the trivial line bundle $\OO_{\RR(Q, (n,n-1,n-3))}$. The action on the fiber can be determined by pairing $-\theta$ with $\lambda'$, and overall, we get $\mu^{L^{\theta}}([E], \lambda) = - \langle \theta, \lambda' \rangle$. Indeed, we have $-\langle \theta_1, \lambda' \rangle = n(n-1) > 0$ for $n \geq 2$ and $-\langle \theta_2, \lambda' \rangle = (n-3)(n-1) > 0$ for $n \geq 4$.
\end{proof}

The next locus to investigate are non-trivial extensions
$0 \to \OO_L(-n + 1) \to E \to \II_P(-1) \to 0$, where $P \in \P^2$ is a point and $L \subset \P^2$ is a line. Using our standard quiver in between $\beta = -1$ and $\beta = -2$, we get the following description of the matrices.

\begin{lem}
\label{lem:quiver_IP-1}
Let $P$ be the point cut out by $x = y = 0$. Then $\II_P(-1)[1] \in \Rep-Q$ with dimension vector $(1,1,0)$, and the matrices are given by $A(\II_P(-1)[1]) = z$, $B(\II_P(-1)[1]) = 0$.
\end{lem}

\begin{proof}
We have
\[
    \RHom(\OO(-1), \II_P(-1)[1]) = 0,
    \RHom(\OO, \II_P(-1)[1]) = \RHom(\OO(1), \II_P(-1)[1]) = H^0(\OO_P).
\]
Clearly, multiplication by $x$ and $y$ is trivial on $H^0(\OO_P)$ and $z$ acts as the identity.
\end{proof}

\begin{lem}
\label{lem:matrices_locus2}
Let $P$ be the point cut out by $x = y = 0$ and $L$ the line cut out by $z = 0$. Then the object $E[1] \in \Rep-Q$ has dimension vector $(n, n - 1, n - 3)$. There is $a \in \C^{n-3} \backslash \{ 0 \}$ such that the matrices have the block form
\[
    A(E[1]) = \begin{pmatrix}
    A(\OO_L(-n + 1)[1]) & 0 \\
    0                   & z
    \end{pmatrix}, \
    B(E[1]) = \begin{pmatrix}
    B(\OO_L(-n + 1)[1]) & az \\
    \end{pmatrix}.
\]
\end{lem}

\begin{proof}
Since $E[1]$ is an extensions of $\OO_L(-n + 1)[1]$ and $\II_P[1]$ whose matrices we know, there are $a, b, c \in \C^{n-2}$ and $d, e, f \in \C^{n-3}$ such that
\[
    A(E[1]) = \begin{pmatrix}
    A(\OO_L(-n + 1)[1]) & ax + by + cz \\
    0                   & z
    \end{pmatrix}, \
    B(E[1]) = \begin{pmatrix}
    B(\OO_L(-n + 1)[1]) & dx' + ey' + fz' \\
    \end{pmatrix}.
\]
Next we make a change of basis. Let $I_m \in \GL_m$ be the identity matrix for $m \in \N$. We define matrices $T_1 \in \GL_n$ and $T_2 \in \GL_{n-1}$ as
\[
    T_1 \coloneqq \begin{pmatrix}
    I_{n-2} & 0 & a \\
    0       & 1 & b_{n-2} \\
    0       & 0 & 1
    \end{pmatrix}, \ 
    T_2 \coloneqq \begin{pmatrix}
    I_{n-2} & -c \\
    0       & 1
    \end{pmatrix}.
\]
Replacing $A(E[1])$ by $T_2 A(E[1]) T_1^{-1}$ and $B(E[1])$ by $B(E[1]) T_2^{-1}$ lets us reduce to the case $a = 0$, $b_{n-2} = 0$, and $c = 0$. That changes $d$ and $e$ but by abuse of notation we still call them $d$ and $e$.

Next we need to use the relations of the quiver. The product $B(E[1]) A(E[1])$ is
\[
    \begin{pmatrix}
    B(\OO_L(-n + 1)[1])A(\OO_L(-n + 1)[1]) & B(\OO_L(-n + 1)[1])by + dx'z + ey'z + fz'z
    \end{pmatrix}.
\]
The relation $y'x = x'y$ implies $b = 0$. The relation $x'z = z'x$ implies $d = 0$. Finally, the relation $y'z = z'y$ implies $e = 0$. We have $f \neq 0$, since otherwise the extension would be trivial.
\end{proof}

\begin{prop}
\label{prop:quiver_locus2}
Any Bridgeland stable extension
\[
0 \to \OO_L(-n + 1) \to E \to \II_P(-1) \to 0
\]
where $L \subset \P^2$ is a line and $P \in \P^2$ is GIT-unstable for $n \geq 4$.
\end{prop}

\begin{proof}
It is enough to show that a general element is GIT-unstable. Therefore, we may assume $P \notin L$. Using the action of $\SL(3)$ we can reduce to the case $P = (0:0:1)$ and $L$ is cut out by $z = 0$. In particular the matrices of Lemma \ref{lem:matrices_locus2} apply.

The one-parameter subgroup $\lambda = (-1, -1, 2)$ does not fix $A(E[1])$, $B(E[1])$, but
\[
\lambda' = ((-1, \ldots, -1, 5), (0, \ldots, 0, 3), (1, \ldots, 1), (-1, -1, 2)) \in \GL_n \times \GL_{n - 1} \times \GL_{n - 3} \times \SL_{3}
\]
does. We can compute $-\langle \theta_1, \lambda' \rangle = n^2 - 4n + 6 > 0$ and $-\langle \theta_2, \lambda' \rangle = (n - 3)(n - 4) > 0$. Therefore, $E$ is GIT-unstable between $\beta = -1$ and $\beta = -2$.

In order to show that $E$ is GIT-unstable everywhere, we will use linearity of the Hilbert-Mumford index. The divisor along the wall whose closure includes $(\alpha, \beta) = (0, -1)$ is $D_{n-1}$ and the divisor along the wall whose closure includes $(\alpha, \beta) = (0, -2)$ is $D_{(n-1)/2}$.

Let $D_m = aD_{n-1} + bD_{(n-1)/2}$ such that the Hilbert-Mumford index vanishes. Then $a + b = 1$ and $a(n^2 - 4n + 6) + b(n^2 - 7n + 12) = 0$. A straightforward calculation shows for any $n \geq 4$
\[
m = -\frac{(n^2 - 10n + 18)(n - 1)}{6(n-2)} < 2.
\]
However, in this case $D_m$ is not even in the effective cone because of Lemma \ref{lem:estimate_effective_cone}.
\end{proof}

The next locus to study is that of non-trivial extensions $0 \to \OO_L(-n + 2) \to E \to \II_W(-1) \to 0$,
where $W \subset \P^2$ is a zero-dimensional subscheme of length two and $L \subset \P^2$ is a line. Any subscheme $Z \subset \P^2$ with length $n = 5$ is contained in a conic and therefore, Proposition \ref{prop:wall_between_-1_-2} implies that $W(\OO(-2), E) = W(\OO_L(-3), E)$ is the final wall in this case. Moreover, we have $\Hom(\OO(-2), \II_W(-1)) = \C$ and $\Ext^1(\OO(-2), \OO_L(-3)) = 0$, i.e., $\Hom(\OO(-2), E) = \C$. This means such an extension $E$ is never stable and these objects are only relevant for $n \geq 6$.

We will mostly deal with $\II_W = (x, y^2)$ and $L$ being cut out by $z$ and generalize from this special case. Using our standard quiver in between $\beta = -1$ and $\beta = -2$, we get the following description of the matrices.

\begin{lem}
Let $\II_W = (x, y^2)$. We have $\II_W(-1)[1] \in \Rep-Q$ with dimension vector $(2,2,1)$ and with an appropriate choice of basis the matrices are given as
\[
    A(E[1]) = \begin{pmatrix}
    z & 0 \\ 
    y & z \\
    \end{pmatrix}, \
    B(E[1]) = \begin{pmatrix}
    y & z
    \end{pmatrix}.
\]
\end{lem}
\begin{proof}
See Example  \ref{ex:matrices}.
\end{proof}

\begin{lem}
\label{lem:matrices_locus3}
Let $\II_W = (x, y^2)$ and $L$ be the line cut out by $z = 0$. Then the object $E[1] \in \Rep-Q$ has dimension vector $(n, n - 1, n - 3)$. There is a basis and $a \in \C^{n-4} \backslash \{ 0 \}$ such that the matrices have the block form
\[
    A(E[1]) = \begin{pmatrix}
    A(\OO_L(-n + 2)[1]) & 0 & 0 \\
    0                   & z & 0 \\
    0                   & y & z
    \end{pmatrix}, \
    B(E[1]) = \begin{pmatrix}
    B(\OO_L(-n + 2)[1]) & az' & 0 \\
    0                   & y'  & z'
    \end{pmatrix}.
\]
\end{lem}

\begin{proof}
Since $E[1]$ is an extensions of $\OO_L(-n + 2)[1]$ and $\II_W(-1)[1]$ whose matrices we know, there are $a, b, c, d, e, f \in \C^{n-3}$ and $a', b', c', d', e', f' \in \C^{n-4}$ such that
\begin{align*}
    A(E[1]) &= \begin{pmatrix}
    A(\OO_L(-n + 2)[1]) & ax + by + cz & dx + ey + fz \\
    0                   & z            & 0            \\
    0                   & y            & z
    \end{pmatrix}, \\
    B(E[1]) &= \begin{pmatrix}
    B(\OO_L(-n + 2)[1]) & a'x' + b'y' + c'z' & d'x' + e'y' + f'z' \\
    0                   & y'                 & z'
    \end{pmatrix}.
\end{align*}
Next we make change of basis. Let $I_m \in \GL_m$ be the identity matrix for $m \in \N$. We define matrices $T_1 \in \GL_n$, $T_2 \in \GL_{n - 1}$, and $T_3 \in \GL_{n - 3}$ via
\[
    T_1 = \begin{pmatrix}
    I_{n-3} & 0 & a                   & d         \\
    0       & 1 & b_{n - 3} - f_{n-3} & e_{n - 3} \\
    0       & 0 & 1                   & 0         \\
    0       & 0 & 0                   & 1         \\
    \end{pmatrix}, \
    T_2 = \begin{pmatrix}
    I_{n-3} & -c & -f \\
    0       & 1  & 0  \\
    0       & 0  & 1
    \end{pmatrix},
\]
\[
    T_3 = \begin{pmatrix}
    I_{n-4} & -b' - \tilde{c} \\
    0       & 1       \\
    \end{pmatrix},
    \tilde{c} = \begin{pmatrix}
    c_2 \\
    \vdots \\
    c_{n-3}
    \end{pmatrix}.
\]
Replacing $A(E[1])$ by $T_2 A(E[1]) T_1^{-1}$ and $B(E[1])$ by $T_3 B(E[1]) T_2^{-1}$ lets us reduce to the case
\begin{align*}
    A(E[1]) &= \begin{pmatrix}
    A(\OO_L(-n + 2)[1]) & by & ey \\
    0                   & z  & 0  \\
    0                   & y  & z
    \end{pmatrix}, \\
    B(E[1]) &= \begin{pmatrix}
    B(\OO_L(-n + 2)[1]) & a'x' + c'z' & d'x' + e'y' + f'z' \\
    0                   & y'                 & z'
    \end{pmatrix},
\end{align*}
where $b_{n-3} = 0$ and $e_{n-3} = 0$. By abuse of notation we keep writing $b$, $e$, $a'$, $c'$, $d'$, $e'$, $f'$ even though their values have technically changed from above.

Next we need to use the relations of the quiver. To shorten notation, let $A = A(\OO_L(-n + 1)[1])$ and $B = B(\OO_L(-n + 1)[1])$. We can compute that $B(E[1]) A(E[1])$ is given by
\[
\begin{pmatrix}
BA & Bby + a'x'z + c'z'z + d'x'y + e'y'y + f'z'y & Bey + d'x'z + e'y'z + f'z'z \\
0  & y'z + z'y                                   & z'z
\end{pmatrix}.
\]
The relation $x'y = y'x$ implies $e = 0$. The relation $x'z = z'x$ implies $a' = d' = 0$. With this last vanishing, we can use $x'y = y'x$ again to get $b = 0$. Using $y'z = z'y$ we get $e' = f' = 0$. Finally, we have $c' \neq 0$, since the extension is non-trivial. 
\end{proof}

\begin{prop}
\label{prop:quiver_locus3}
Any Bridgeland-stable extension $0 \to \OO_L(-n + 2) \to E \to \II_W(-1) \to 0$
where $W \subset \P^2$ is zero-dimensional of length two and $L \subset \P^2$ is a line is GIT-unstable for $n \geq 6$.
\end{prop}

\begin{proof}
It is enough to show that a general element is GIT-unstable. Therefore, we may assume $W \cap L = \emptyset$ and $W$ is reduced. Using the action of $\PGL(3)$ we can reduce to the case $W = \{(0:0:1), (0:1:1)\}$ and $L$ is cut out by $z = 0$. Let $\lambda$ be the one-parameter subgroup $(-1, -1, 2)$. We have $\lim_{t \to 0} \lambda \cdot \II_W = (x, y^2)$ and thus, we may simply assume that $\II_W = (x, y^2)$ to begin with. In particular the matrices of Lemma \ref{lem:matrices_locus3} apply, i.e.,
\[
    A(E[1]) = \begin{pmatrix}
    A(\OO_L(-n + 2)[1]) & 0 & 0     \\
    0                   & z & 0    \\
    0                   & y & z
    \end{pmatrix}, \
    B(E[1]) = \begin{pmatrix}
    B(\OO_L(-n + 2)[1]) & az' & 0 \\
    0                   & y'  & z'
    \end{pmatrix}.
\]
for some $a \in \C^{n-4} \backslash \{ 0 \}$. We define a one-parameter subgroup
\[
\lambda' = ((-1, \ldots, -1, 5, 8), (0, \ldots, 0, 3, 6), (1, \ldots, 1, 4), (-1, -1, 2)) \in \GL_n \times \GL_{n - 1} \times \GL_{n - 3} \times \SL_3.
\]
Then $\lambda'$ fixes the matrices and in particular, $\lambda = (-1, -1, 2)$ fixes $E[1]$. We can compute
\[
    - \langle \theta_1, \lambda' \rangle = n^2 - 7n + 15 > 0, \
    - \langle \theta_2, \lambda' \rangle = n^2 - 10n + 27 > 0.
\]
Let $D_m = aD_{n-1} + bD_{(n-1)/2}$ such that the Hilbert-Mumford index vanishes. Then $a + b = 1$ and $a(n^2 - 7n + 15) + b(n^2 - 10n + 27) = 0$. A straightforward calculation shows for any $n \geq 6$
\[
m = -\frac{(n^2 - 13n + 39)(n - 1)}{6(n-4)} < 2.
\]
Due to Lemma \ref{lem:estimate_effective_cone}, $D_m$ for $m < 2$ is outside the effective cone, and thus, the Hilbert-Mumford criterion is strictly positive for any effective divisor.
\end{proof}


\subsection{Proofs of Theorems \ref{thm:MainThmWall}, \ref{thm:largestWall}, and related results}

To focus on the core arguments, we will postpone various technical computations in the upcoming proofs until Section \ref{sec:git_calculations}.

\MainThmWall*

\begin{proof}
The existence of the wall is an immediate consequence of Lemma \ref{lem:wall_for_fat_plus_line}.

\begin{enumerate}
    \item Let $Z$ be the union of a generic curvilinear point of length $(n - l)$ and $l$ generic points in the plane. Our strategy is to take the flat limit with respect to $\lambda_1$. This limit will yield the ideal from Lemma \ref{lem:wall_for_fat_plus_line}. The semistability of the limit at $m_l$ implies the semistability of $Z$. In Lemma \ref{lem:wall_for_fat_plus_line} we established that $\lambda_1$ destabilizes this limit for $m > m_l$. This shows that our $Z$ has the same property. Next, we provide details for this argument.

    The ideal of a general curvilinear point of length $(n - l)$ in $\A^2$ that is supported at the origin can be written as $\left( (x, y)^{n - l} + (p_{n - l - 1}(x,y)) \right)$ where $p_{n - l - 1}(x, y)$ is a polynomial not necessarily homogeneous of degree $(n - l - 1)$ which is smooth at the origin. Therefore, we can suppose that up to a change of coordinates, the ideal associated to $Z$ in $\P^2$ is cut out by the ideal
    \[
        \widetilde{\II}_Z = 
        \left( (x, y)^{n - l} + (p_{n - l - 1}(x, y, z)) \right)
        \cap 
        \bigcap_{i = 1}^{i = l} (y - a_i x, z - b_i x)
    \]
    where $p_{n - l - 1}(x, y, z) = z^{n - l - 2} q_{1}(x, y) + z^{n - l - 3} q_{2}(x, y) + \cdots + q_{n - l - 1}(x, y)$ for homogeneous polynomials $q_i(x, y)$ of degree $i$ with $q_1(x, y) \neq 0$. This means $p_{n - l - 1}$ defines a curve going through $[0:0:1]$ which is smooth at this point. Note that this ideal is not necessarily saturated and we will only take the saturation at the end. We want to take the flat limit of $\widetilde{\II}_Z$ with respect to $\lambda_1$. For this, we recall that 
    \[
        \lim_{t \to 0} \lambda_1(t) \cdot [1: a_i: b_i]
        =
        \lim_{t \to 0} [t^{-1}: t^{-1}a_i: t^{2}b_i]
        =
        \lim_{t \to 0} [1: a_i: t^{3}b_i]
        =
        [1:a_i:0].
    \]
    Therefore,
    \[
        \lim_{t \to 0} \lambda_1(t) \cdot \bigcap_{i = 1}^l (y - a_i x, z - b_i x) = (z, r_l(x, y))
    \]
    where $r_l(x,y)$ is a homogeneous polynomial that vanishes at the points $[1:a_1:0]$. On the other hand, 
    \[
        \lim_{t \to 0} \lambda_1(t) \cdot \left( (x, y)^{n - l} + (p_{n - l - 1}(x, y, z)) \right) = \left( (x, y)^{n - l} + (z^{n - l - 2} q_1(x,y)) \right)
    \]
    The non-reduced point is disjoint from the collinear points, so their flat limits can be computed independently.  Moreover, by using the group action, we can suppose that $q_1(x, y)=x$. Therefore, we obtain 
    \begin{align*}
        \widetilde{\II}_{Z_0} \coloneqq
        \lim_{t \to 0} \lambda_1(t) \cdot \widetilde{\II}_Z
        &= \left( x^{n - l}, x^{n - l - 1}y, \ldots, xy^{n - l - 1}, y^{n - l}, z^{n - l - 2}x \right)
        \bigcap
        (z, r_{l}(x,y))
        \\
        &= \left( zx^{n - l}, zx^{n - l - 1}y, \cdots, zxy^{n - l - 1},  zy^{n - l}, z^{n - l - 2}x, r_{l}(x,y) \right)
    \end{align*}
    Let $\II_{Z_0}$ be the saturation of $\widetilde{\II}_{Z_0}$. Since $xz \in \II_{Z_0}$, we get
    \[
        \II_{Z_0} = \left(xz, zy^{n-l}, r_{l}(x,y) \right)
    \]
    By Lemma \ref{lem:wall_for_fat_plus_line} $Z_0$ is $m_l$-semistable, and thus, $Z$ is $m_l$-semistable as well.
    
    \item Now, let $Z$ be the union of $l$ general collinear points in one line and $n - l$ general collinear points in another line. Up to a change of coordinates, we can assume that
    \[
        \II_Z = \left(x, q_{n - l}(y, z)\right) \cap \left(z, p_l(x, y)\right)
    \]
    for general homogeneous polynomials $q_{n - l}(y, z)$ and $p_l(x, y)$ of degree $n - l$ and $l$ respectively. The one-parameter subgroup $\lambda_1^{-1}$ fixes points on the line cut out by $z$. If $P$ is a general point on the line cut out by $x$, then $\lim_{t \to 0} \lambda_1^{-1}(t) \cdot P = (0:0:1)$. Overall, we get that $\lim_{t \to 0} \lambda_1^{-1}(t) \cdot \II_Z$ leads to the subscheme described in Lemma \ref{lem:wall_for_fat_plus_line}.
    
    \item The last statement follows directly from the proofs of (i) and (ii). \qedhere
\end{enumerate}
\end{proof}

\largestWall*

\begin{proof}
Let $m = m_0$ be the largest wall. We are trying to show $m_0 = m_{\lfloor 2n/3 \rfloor}$. At this wall, there must be a strictly-semistable subscheme $Z$ that is unstable at $m = \infty$. By Theorem \ref{thm:stability_chow} such a $Z$ must be in the closure of the loci of subschemes with either a non-reduced subscheme supported at a single point of length $\lceil \frac{n}{3} \rceil$ or containing a collinear subscheme of length $\lceil \frac{2n}{3} \rceil$.

If $Z$ contains a non-reduced subscheme supported at a single point of length $\lceil \frac{n}{3} \rceil$, then Theorem \ref{thm:MainThmWall} shows that $Z$ is unstable for $m > m_{\lfloor 2n/3 \rfloor}$.
From now on assume that $Z$ contains a collinear subscheme of length $l \coloneqq \lceil \frac{2n}{3} \rceil$. We will show that $Z$ is $m$-unstable for $m_0 \geq l$. This is enough since $m_{\lfloor 2n/3 \rfloor} \geq l$. By using the numerical 
criterion described in the proof of 
\cite[Theorem 11.1]{Dol03:invariant_theory}, 
we conclude that $\lambda_{-1/2}$ destabilizes the cycle associated to $Z$ at $m = \infty$, i.e., in the symmetric product. By using the Hilbert-Chow morphism, we conclude that $\lambda_{-1/2}$ destabilizes $Z$ for $m \gg 1$. Therefore, it will be enough to show that $\lambda_{-1/2}$ destabilizes $Z$ for $m = m_l$. Taking the limit of $Z$ with respect $\lambda_{-1/2}$ turns out to be complicated. Since $\mu_l(Z, \lambda_{-1/2 - \varepsilon})$ depends continuously on $\varepsilon$, we will use $\lambda_{-1/2 - \varepsilon}$ for small $\varepsilon > 0$ instead.

We may assume that this $Z$ is general under the assumption that it contains a collinear subscheme of length $l$. Up to the action of $\SL_3$ this means
\[
    \II_Z = \bigcap_{i = 1}^{n-l} (y - a_i x, z - b_i x) \cap (x, p_l(y,z))
\]
for a general homogeneous polynomial $p_l(y, z) \in \C[y, z]$ of length $l$, and general $a_i, b_i \in \C$ for $i = 1, \ldots, n - l$. There are unique integers $s \geq 1$, $0 \leq k \leq s$ such that $n - l = \Delta(s) + k$. By Lemma \ref{lem:limit_with_cutting_corners}, we know for any $\varepsilon > 0$
\[
    \II_{Z_0} \coloneqq \lim_{t \to 0} \lambda_{-\tfrac{1}{2} - \varepsilon}(t) \cdot \II_Z = (m_0, \ldots, m_s) \cap (x, p_l(y,z)). 
\]
By definition of the Hilbert-Mumford index, we have $\mu_l(Z, \lambda_{-1/2 - \varepsilon}) = \mu_l(Z_0, \lambda_{-1/2 - \varepsilon})$ and this depends continuously on $\varepsilon$. Therefore, it is enough to show $\mu_l(Z_0, \lambda_{-1/2}) > 0$. Lemma \ref{lem:numerical_criterion_wall_from_collinear} implies
\[
    \mu_l \left(Z_0, \lambda_{-\tfrac{1}{2}}\right) = \frac{s^3 - ls^2 + 3ks - ls - s + l^2 - 2kl}{2}.
\]
From here, the argument splits into two cases.

\begin{enumerate}
    \item If $n \equiv 1 (\mod 3)$, then $l = 2 \Delta(s) + 2k + 1$. In this scenario we get
    \[
        \mu_l \left(Z_0, \lambda_{-\tfrac{1}{2}}\right) = \frac{s^3 + 3ks + s^2 + 2k + 1}{2} > 0.
    \]
    
    \item If $n \equiv 2 (\mod 3)$, then $l = 2 \Delta(s) + 2k + 2$ and
    \[
        \mu_l \left(Z_0, \lambda_{-\tfrac{1}{2}}\right) = \frac{s^3 + 3ks + 2s^2 + 4k + s + 4}{2} > 0.
    \]
\end{enumerate}
The ample cone is spanned by $D_{n - 1}$ and $D_{\infty} = H$. A straightforward computation shows that the wall is in the interior of the ample cone if and only if $n \geq 11$ or $n = 8$.
\end{proof}

The following result is limited to the case where $n - l$ is a triangular number. It is conceivable that a very similar statement holds in general, but the situation will become quite a bit more technical.

\begin{prop}
\label{prop:walls_collinear_points}
Let $l, n$ be integers such that $n - l$ is the triangular number $\Delta(s) = 1 + \ldots + s$. Assume further that $s^2 \leq l < 2\Delta(s)$. Then there exists a GIT wall at $m_{l, s} \coloneqq \frac{s(s + 1)(s - 1)}{s^2 + s - l}$ such that
\begin{enumerate}
    \item the union of $l$ generically collinear points and $\Delta(s)$ generic points in the plane which is unstable for $m < m_{l, s}$ becomes strictly semistable for $m = m_{l, s}$, and 
    \item the union of a subscheme projectively equivalent to the one cut out by $(y^s, y^{s-1}z, \ldots, z^s)$ and $l$ generic points is unstable for $m > m_{l, s}$ and strictly semistable at $m = m_{l, s}$.
    \item Both cases above degenerate to a common strictly semistable configuration at $m = m_{l, s}$. 
\end{enumerate}
\end{prop}

\begin{proof}
The fact that the wall exists is immediate from Lemma \ref{lem:wall_for_collinear_points}.
\begin{enumerate}
    \item Let $Z$ be the union of $l$ generically collinear points and $\Delta(s)$ generic points. Up to a change of coordinates we can assume that
    \[
        \II_Z = \bigcap_{i = 1}^{\Delta(s)} (y - a_i x, z - b_i x) \cap (x, p_l(y, z))
    \]
    where $p_l(y, z)$ is general of degree $l$. Note that for a general point $P \in \P^2$, we have $\lim_{t \to 0} \lambda_{-\frac{1}{2}}(t) \cdot P = (1:0:0)$. By Lemma \ref{lem:limit_with_cutting_corners}
    \[
    \lim_{t \to 0} \lambda_{-\frac{1}{2}}(t) \cdot \bigcap_{i = 1}^{\Delta(s)} (y - a_i x, z - b_i x) = (y^s, y^{s - 1}z, \ldots, z^s).
    \]
    Overall, this means that $\lim_{t \to 0} \lambda_{-1/2}(t) \cdot Z$ is an ideal as described in Lemma \ref{lem:wall_for_collinear_points}. This shows the claimed semistability at $m = m_{l, s}$. Moreover, $\lambda_{-1/2}$ destabilizes $Z$ for $m < m_{l, s}$.
    
    \item Let
    \[
    \II_Z = (y^s, y^{s-1}z, \ldots, z^s) \cap \bigcap_{i = 1}^l (y - a_i x, z - b_i x)
    \]
    for general $a_1, \ldots, a_l, b_1, \ldots, b_l \in \C$. Then taking the limit $\lim_{t \to 0} \lambda_{-1/2}^{-1}(t) \cdot \II_Z$ leads to the subscheme from Lemma \ref{lem:wall_for_collinear_points}.
    
    \item The last statement is implicit in the proofs of (i) and (ii). \qedhere
\end{enumerate}
\end{proof}

\begin{prop}
\label{prop:wall_conic}
There is a GIT wall at $m_0 = \tfrac{n - 1}{2}$ such that subschemes contained in a conic become strictly-semistable at $m = m_0$.
\end{prop}

\begin{proof}
Again by Corollary \ref{cor:ideals_between_-1_-2} the ideal $\II_Z$ is destabilized in Bridgeland stability precisely at the wall that induces the divisor $D_{m_0}$ and therefore, the statement makes sense.

Since the locus of unstable points is closed, we may assume that $Z$ consists of $n$ general points in a conic. Using the action of $\SL_3$, we can assume that the conic is given by $y^2 + xz$. This conic is stabilized by the one-parameter subgroup $\lambda_0$. Moreover, for a general point $P$ on this conic, we have $\lim_{t \to 0} \lambda_0(t) \cdot P = (1:0:0)$. Therefore, taking the limit $\lim_{t \to 0} \lambda_0(t) \cdot Z$ leads to Lemma \ref{lem:wall_conic_even} or Lemma \ref{lem:wall_conic_odd}.
\end{proof}


\subsection{Further remarks} Next, we establish a few smaller results.

\begin{prop}
\label{prop:plane_curves_vs_points}
Let $Z \subset \P^2$ be a zero-dimensional subscheme and let $m$ be an integer such that the $m$-th Hilbert point of $Z$ is well-defined. If there exists a basis $f_1, \ldots f_r \in H^0\left( \II_Z(m) \right)$ such that the plane curve cut out by $\prod_i^r f_i(x,y,z)$ is GIT unstable, then $Z$ is $m$-unstable. In particular, if the curve cut out by $\prod_i^r f_i$ has a singular point of multiplicy larger than $\frac{2mr}{3}$, then $Z$ is $m$-unstable.
\end{prop}

\begin{proof}
The Pl\"ucker coordinates associated to the $m$-th Hilbert point of $\II_Z$ are defined by the $(r \times r)$-minors of a matrix whose columns represent the polynomials $f_i(x,y,z)$ with respect to the basis of degree $m$ monomials. Therefore, the non-zero Pl\"ucker coordinates are of the form $\wedge_{i=1}^r x_{K(i)}$ where $x_{K(i)}$ is a monomial in the polynomial $f_i(x,y,z)$.

Let $\lambda$ be a one-parameter subgroup such that $\mu_m(\prod_i^r f_i, \lambda) > 0$. Without loss of generality we may assume that $\lambda$ is diagonal. If not, we simply do a base-change. Observe that $\lambda$ acts on both $\bigwedge_{i=1}^r x_{K_i}$ and the symmetric product $\prod_{i=1}^r x_{K(i)}$. Moreover, if the wedge product is non-zero, then their $\lambda$-eigenvalues are the same. As a consequence, $\mu_m(Z, \lambda) \geq \mu(\prod_i^r f_i, \lambda) > 0$ by Remark \ref{rmk:minNumerical}.

The second statement follows from \cite[Exercise 10.12]{Dol03:invariant_theory}.
\end{proof}

We finish this section with an interesting example for $n = 7$. We will only sketch the argument and leave more details to the reader.

\begin{ex}
\label{ex:SpecialTriplePoint7}
Let $Z \subset \P^2$ be the subscheme cut out by the ideal
\[
    \II_Z = (y^2 + xz, xy, x^2) \cap (x,z) \cap \bigcap_{i=1}^{3} (y - a_i x, z - b_i x)
\]
for generic $a_1, a_2, a_3, b_1, b_2, b_3 \in \C$. In Theorem \ref{thm:MainThmWall} we have shown that the union of a curvilinear triple point with four general points is strictly-semistable at $m = \tfrac{9}{2}$. In our case, there is a unique line that intersects the point cut out by $(y^2 + xz, xy, x^2)$ in length two and this line is cut out by $x$. The fact that one of the four points is on that line means that $Z$ is not general enough to fall under the precise statement of Theorem \ref{thm:MainThmWall}. A similar argument shows that $\lim_{t \to 0} \lambda_0(t) Z$ is strictly-semistable at $m = \tfrac{9}{2}$ and $\mu_{9/2}(Z, \lambda_0) = 0$.

In Theorem \ref{thm:stability_final_model_n_7} we will show that $\II_Z$ is also strictly-semistable in the final model at $m = \tfrac{5}{2}$. This means $\II_Z$ has to be semistable for all $\tfrac{5}{2} \leq m \leq \tfrac{9}{2}$. A close inspection of the argument shows that the one-parameter subgroup $\lambda_0$ is responsible for $Z$ being strictly-semistable at $m = \tfrac{5}{2}$ as well. Therefore, $\II_Z$ is going to be strictly-semistable for all $\tfrac{5}{2} \leq m \leq \tfrac{9}{2}$.

\end{ex}

\section{Five points in detail}

We start by showing in Section \ref{subsec:complexes_five_points} that all Bridgeland-semistable objects that are not ideals are GIT-unstable in every birational model for $n = 5$. Therefore, the proof relies on understanding stability of ideals, which is mostly done in Section \ref{subsec:unstable_ideals_five_points}. The final arguments for the proof of Theorem \ref{thm:main_n_5} about the Variation of GIT are given in Section \ref{subsec:main_thm_five_points}.

\subsection{Bridgeland-semistable objects}
\label{subsec:complexes_five_points}

\begin{prop}[\cite{ABCH13:hilbert_schemes_p2}]
\label{prop:bridgeland_walls_5}
There are three walls in Bridgeland stability for objects $E$ with Chern character $(1,0,-5)$ when $\beta < 0$. From largest to smallest wall they are the following.
\begin{enumerate}
\item Semistable objects are given by extensions between $\OO(-1)$ and $\OO_L(-5)$ for a line $L \subset \P^2$. Therefore, this wall destabilizes ideal sheaves of points completely contained in $L$, and replaces them by non-trivial extensions
\[
0 \to \OO_L(-5) \to E \to \OO(-1) \to 0.
\]
\item Semistable objects are given by extensions between $\II_P(-1) $ and $\OO_L(-4)$ for a line $L \subset \P^2$. Therefore, this wall destabilizes ideal sheaves of points where four points lie on a line $L$, and replaces them by sheaves with torsion.
\[
0 \to \OO_L(-4) \to E \to \II_P(-1) \to 0.
\]
\item Finally, all semistable objects have a map from $\OO(-2)$ and after this wall the moduli space is empty. In the open locus that still parametrizes ideal sheaves, this is simply saying that five points always lie on a conic.
\end{enumerate}
The effective cone of $\P^{2[n]}$ is spanned by the exceptional divisor $\Delta$ and $D_2$. The movable cone of $\P^{2[5]}$ is spanned by $H$ and $D_2$.
\end{prop}

By Proposition \ref{prop:quiver_locus1} and Proposition \ref{prop:quiver_locus2} all Bridgeland-stable objects destabilizing along the first two walls that are not ideal sheaves are GIT-unstable. Any non-ideal sheaf destabilizes at either of those two walls. 
Therefore, the only GIT-semistable objects are ideals.

Finally, we will translate what the above result means for the Bridgeland stability of an ideal $\II_Z$ where $Z$ is zero-dimensional of length five. 

\begin{cor}
Assume that $\II_Z$ is a zero-dimensional subscheme $Z \subset \P^2$ of length five and that $\sigma_m$ is a stability condition whose induced divisor is a positive multiple of $D_m$.
\begin{enumerate}
    \item Assume that $4 \leq m \leq \infty$. Then any $\II_Z$ is $\sigma_m$-semistable.
    \item Assume that $3 \leq m < 4$. Then $\II_Z$ is $\sigma_m$-semistable if and only if $Z$ is not collinear.
    \item Assume that $2 \leq m < 3$. Then $\II_Z$ is $\sigma_m$-semistable if and only if $Z$ does not contain a collinear subscheme of length four.
\end{enumerate}
\end{cor}


\subsection{GIT-unstable points}
\label{subsec:unstable_ideals_five_points}

In this section we will find unstable $Z$ for $n = 5$ depending on $m$. However, the fact there are no further unstable $Z$ will only be shown in Section \ref{subsec:main_thm_five_points}.

\begin{lem}
\label{lem:n5_triple_point_unstable}
Let $Z \in \P^{2[5]}$ be a subscheme containing a point of length at least three. Then $Z$ is $m$-unstable for all $m \geq 2$. 
\end{lem}

\begin{proof}
Since the locus of GIT-unstable points is closed, it is enough to show that a subscheme $Z$ containing a curvilinear triple point lying on a smooth conic is unstable. 

We can use the action of $\SL_3$ to assume that $Z$ is given by the ideal
\[
\II_Z = \left( xz + y^2, x^2, xy \right) \cap \left( xz + y^2, x+z \right) = 
\left( xz + y^2, x^2y + xyz, x^3 + x^2z \right).
\]
Clearly, the 3rd and 4th Hilbert points are well defined for $Z$, and we can compute
\begin{align*}
\State_3(\II_Z) &= \Conv \{ (9,2,4), (5,6,4), (6,8,1), (4,8,3) \}, \\
\State_4(\II_Z) &= \Conv \{ (20,10,10), (19,10,11), (13,16,11), (16,18,6), (12, 18, 10) \}.
\end{align*}
In particular,$\mu_{3}(Z, \lambda_0) = 1$ and $\mu_{4}(Z, \lambda_0)= 2$.
Lemma \ref{lemma:m0} implies that $\mu_m(Z, \lambda_0) > 0$ for all $m > m_0(Z, \lambda_0) = 2$.
\end{proof}

\begin{lem}
\label{lem:n5_four_points_line_unstable}
Let $Z \in \P^{2[5]}$ be a subscheme containing a collinear length four subscheme. Then $Z$ is $m$-unstable for all $m > 2$. 
\end{lem}

\begin{proof}
Since the locus of GIT-unstable points is closed, it is enough to show that a general $Z$ as in the statement is GIT-unstable. This means we can assume $Z$ is reduced and exactly four points lie in $L$. We can use the action of $\SL_3$ to assume that $L$ is cut out by $x = 0$ and that the fifth point outside of $L$ is cut out by $y = z = 0$. Therefore,
\[
\II_Z = (x, p_4(y, z)) \cap (y, z) = (xy, xz, p_4(y, z))
\]
for some degree four polynomial $p_4 \in \C[y,z]$. We can use that $Z$ is general again to assume that $p_4(y, z) = ay^4 + by^3z + cy^2z^2 + dyz^3 + ez^4$ for $a, b, c, d, e \in \C$ with $a, e \neq 0$.

Clearly, the 4th and 5th Hilbert points are well defined for $Z$, and we can compute
\[
\State_{4}(\II_Z) = \Conv \{ (16, 14, 10), (16, 10, 14) \}, \
\State_{5}(\II_Z) = \Conv \{ (30, 29, 21), (30, 21, 29) \}.
\]
In particular,
\begin{align*}
\mu_{4}(Z, \lambda_r) &= \min \{ 16 + 14r - 10(1 + r), 16 + 10r - 14(1 + r) \} = 2 - 4r, \\
\mu_{5}(Z, \lambda_r) &= \min \{ 30 + 29r - 21(1 + r), 30 + 21r - 29(1 + r) \} = 1 - 8r.
\end{align*}
Lemma \ref{lemma:m0} implies that $\mu_{m}(Z, \lambda_r) \geq 0$ for all
\begin{align*}
m > m_0(Z, \lambda_r) = 4 + \frac{2-4r}{1 + 4r} = \frac{6 + 12r}{1 + 4r}.
\end{align*}
Therefore, $Z$ is destabilized by  $\lambda_{-3/4}$ for all $m > \tfrac{3}{2}$, and in particular for $m \geq 2$.
\end{proof}

\begin{lem}
\label{lem:n5_three_points_line_unstable}
Let $Z \in \P^{2[5]}$ be a subscheme containing a collinear length three subscheme. Then $Z$ is $m$-unstable for all $2 \leq m < 3$. 
\end{lem}

\begin{proof}
Since the locus of GIT-unstable points is closed, it is enough to consider the case where $Z$ is reduced and no subset of four points lies on a line. In that case, we can use the action of $\SL_3$ to make the five points $(0:0:1)$, $(0:1:0)$, $(1:0:0)$, $(1:0:-1)$, and $(0:-a:1)$ for some $a \neq 0$. Then the ideal is given by
\[
\II_Z = (x, yz(y + az)) \cap (y, z(x+z)) = (xy, yz(y + az), xz(x+z)).
\]
A straightforward computation shows
\begin{align*}
    \State_3(\II_Z) &= \Conv \{ (5,5,5), (5,6,4), (6,5,4), (6,6,3) \}, \\
    \State_4(\II_Z) &= \Conv \{ (13,13,14), (15,15,10), (13,15,12), (15,13,12) \}.
\end{align*}
In particular,
\begin{align*}
        \mu_{3}(Z, \lambda_{-1/2}) = \min \left\{ 0, 0, \frac{3}{2}, \frac{3}{2} \right\} = 0, 
        &&
    \mu_{4}(Z, \lambda_{-1/2}) = \min \left\{ -\frac{1}{2}, \frac{5}{2}, -\frac{1}{2}, \frac{5}{2} \right\} = -\frac{1}{2}.
\end{align*}
Therefore, $\mu_m(Z, \lambda_{-1/2}) > 0$ for all $m < 3$.
\end{proof}

\begin{lem}
\label{lem:n5_non_reduced_unstable}
Let $Z \in \P^{2[5]}$ be a non-reduced subscheme. Then $Z$ is $m$-unstable for all $m > 3$. 
\end{lem}

\begin{proof}
This is a direct consequence of Theorem \ref{thm:MainThmWall} by setting $n = 5$ and $l = 3$.
\end{proof}

Next we describe slightly more special configurations than in the last lemma that are unstable even for $m \leq 3$. 

\begin{lem}
\label{lem:n5_non_reduced_special_config_unstable_I}
Let $L_1, L_2, L_3, L_4$ be lines such that $L_2, L_3, L_4$ intersect in precisely one point, and $L_1$ intersects the other three lines in three separate points. Then we can define the following zero-dimensional subscheme $Z$:
\begin{enumerate}
    \item $Z$ has reduced points at $L_1 \cap L_3$, $L_1 \cap L_4$, and $L_2 \cap L_3 \cap L_4$;
    \item $Z$ has a length two point supported at $L_1 \cap L_2$ that is scheme-theoretically contained in $L_2$.
\end{enumerate}
This subscheme $Z$ is $m$-unstable for all $m \geq 2$.
\end{lem}

\begin{proof}
Using the action of $\SL_3$ we can assume that
\[
\II_Z = (x, (y + z)z) \cap (x^2, y) \cap (y, z) = (xy, x^2z, y^2z + yz^2).
\]
Clearly, the 3rd and 4th Hilbert points are well defined for $Z$, and we can compute
\[
    \State_3(\II_Z) = \Conv \{ (6, 6, 3), (6, 5, 4) \}, \
    \State_4(\II_Z) = \Conv \{ (15, 15, 10), (15, 13, 12) \}.
\]
Thus, $\mu_{3}(Z, \lambda_{-1/2}) = \frac{3}{2}$, $\mu_{4}(Z, \lambda_{-1/2}) = \frac{5}{2}$. By Lemma \ref{lemma:m0} $\mu_m(Z, \lambda_{-1/2}) > 0$ for all $m > \tfrac{3}{2}$.
\end{proof}

\begin{lem}
\label{lem:n5_non_reduced_special_config_unstable_II}
Let $L_1, L_2, L_3, L_4$ be lines such that $L_1, L_2, L_3$ intersect in precisely one point, and $L_4$ intersects the other three lines in three separate points. Then we can define the following zero-dimensional subscheme $Z$:
\begin{enumerate}
    \item $Z$ has a reduced point at $L_3 \cap L_4$;
    \item $Z$ has a length two point supported at $L_1 \cap L_4$ that is scheme-theoretically contained in $L_1$;
    \item $Z$ has a length two point at $L_1 \cap L_2 \cap L_3$ that is scheme-theoretically contained in $L_2$.
\end{enumerate}
This subscheme $Z$ is $m$-unstable for all $m \geq 2$.
\end{lem}

\begin{proof}
Using the action of $\SL_3$ we can assume that
\[
    \II_Z = (x, z^2) \cap (x^2, y) \cap (x + y, z) = (x^2 + xy, yz^2, xyz).
\]
Clearly, the 3rd and 4th Hilbert points are well defined for $Z$, and we can compute
\[
    \State_3(\II_Z) = \Conv \{ (8, 3, 4), (6, 5, 4) \}, \
    \State_4(\II_Z) = \Conv \{ (18, 10, 12), (15, 13, 12) \}.
\]
In particular, $\mu_{3}(Z, \lambda_1) = 3$ and $\mu_{4}(Z, \lambda_1) = 4$.
By Lemma \ref{lemma:m0} $\mu_m(Z, \lambda_1) > 0$ for all $m > 0$.
\end{proof}

\begin{lem}
\label{lem:n5_non_reduced_special_config_unstable_III}
Let $L_1, L_2, L_3, L_4$ be lines such that $L_2, L_3, L_4$ intersect in precisely one point, and $L_1$ intersects the other three lines in three separate points. Then we can define the following zero-dimensional subscheme $Z$:
\begin{enumerate}
    \item $Z$ has a reduced point at $L_1 \cap L_4$;
    \item $Z$ has a length two point supported at $L_1 \cap L_2$ that is scheme-theoretically contained in $L_2$;
    \item $Z$ has a length two point supported at $L_1 \cap L_3$ that is scheme-theoretically contained in $L_3$.
\end{enumerate}
This subscheme $Z$ is $m$-unstable for all $m \geq 2$.
\end{lem}

\begin{proof}
Using the action of $\SL_3$ we can assume that
\[
    \II_Z = (x^2, y) \cap (x^2, z) \cap (x, y + z) = (x^2, y^2z + yz^2, xyz).
\]
Clearly, the 3rd and 4th Hilbert points are well defined for $Z$, and we can compute
\[
    \State_3(\II_Z) = \Conv \{ (8, 4, 3), (8, 3, 4) \}, \
    \State_4(\II_Z) = \Conv \{ (18, 12, 10), (18, 10, 12) \}.
\]
In particular, $\mu_{3}(Z, \lambda_0) = 4$ and $\mu_{4}(Z, \lambda_0) = 6$. By Lemma \ref{lemma:m0} $\mu_m(Z, \lambda_1) > 0$ for all $m > 1$.
\end{proof}


\subsection{VGIT}
\label{subsec:main_thm_five_points}
The goal of this section is to prove Theorem \ref{thm:main_n_5}. 
Recall that we always quotient an open subset of $\P^{2[5]}$, so we can mostly ignore the derived category in the proof.

\begin{lem}
\label{lem:n5_stable_m>3}
If $m > 3$, then $Z \in \P^{2[5]}$ is GIT-stable if and only if $Z$ is reduced and no four points lie in a single line. Moreover, there are no strictly GIT-semistable points. Therefore, the quotient for $3 < m < \infty $ is the same as for $m = \infty$ given as the quotient of $\P^{2(5)}$ by $\SL_3$.
\end{lem}

\begin{proof}
A cycle in $\P^{2(5)}$ GIT-stable if and only if $Z$ is reduced and no four points lie in a single line (see Section \ref{subsec:bridgeland}).  Moreover, there are no strictly GIT-semistable points. Therefore, we can conclude for $m \gg 3$ by Theorem \ref{thm:RelativeGIT}.

If there was a wall for some $m_0 \in (3, \infty)$, then a GIT-unstable point for $m > m_0$ must become strictly $m_0$-semistable. Such a point must be either non-reduced or contain four points in a single line. However, those points are unstable for all $m > 3$ by Lemma \ref{lem:n5_four_points_line_unstable} and Lemma \ref{lem:n5_non_reduced_unstable}.
\end{proof}

For any $(a : b) \in \P^1$, we define a subscheme $Z^0_{(a:b)}$ through the ideal
\[
\II_{Z^0_{(a:b)}} = (x, (y + z)(y - z)(ay + bz)) \cap (z, y^2) = (xz, xy^2, (y + z)(y - z)(ay + bz)).
\]
The base change via the matrix
$\text{diag}(1,i, -i)$ shows that $Z^0_{(a:b)}$ and $Z^0_{(a:-b)}$ are projectively equivalent. A straightforward computation shows that $Z^0_{(a:b)}$ is projectively equivalent to $Z^0_{(c:d)}$ if and only if $(c:d) \in \{(a:b), (a:-b)\}$. It will turn out that for $(a: b) \notin \{ (0:1), (1:1), (1:-1)\}$ these subschemes represent closed minimal orbits at $m = 3$. However, the remaining cases of $(a, b)$ are more problematic. Note that since $Z^0_{(1:1)}$ and $Z^0_{(1:-1)}$ are in the same orbit, we will only have to deal with $(1:1)$ and $(0:1)$. 

For $(a:b) = (1:1)$ let $\lambda$ be the one-parameter subgroup that acts by $t$ on $x$, by $1$ on $y + z$, and by $t^{-1}$ on $y - z$. Then
\[
\lim_{t \to 0} \lambda(t) (x, (y + z)^2(y - z)) \cap (z, y^2) = (x, (y + z)^2(y - z)) \cap (y - z, y^2).
\]
For $(a:b) = (0:1)$ let $\lambda$ be the one-parameter subgroup that acts by $t$ on $x$, by $t^{-1}$ on $y + z$, and by $1$ on $z$. Then
\[
\lim_{t \to 0} \lambda(t) (x, (y + z)(y - z)z) \cap (z, y^2) = (x, (y + z)^2z) \cap (z, y^2)
\]
is projectively equivalent to the previous limit. We define $Z^0$ through the ideal
\[
\II_{Z^0} = (x, y^2) \cap (x, z) \cap (z, y^2) = (xz, xy^2, y^2z).
\]

\begin{lem}
\label{lem:closed_orbits_at_m_3}
There is a GIT-wall at $m = 3$. Every strictly $3$-semistable orbit with positive dimensional stabilizer contains one of the subschemes $Z^0_{(a:b)}$ for $(a : b) \in \P^1$ 
or the subscheme $Z^0$. Moreover, each of these subschemes is strictly $3$-semistable.
\end{lem}

\begin{proof}
We start by showing that these $Z$ are indeed strictly-semistable. The stabilizer of $Z^0$ is given by the diagonal maximal torus. By Remark \ref{rmk:tori_to_check} this means stability can be checked with just diagonal one-parameter subgroups. A straightforward computation shows $\State_3(\II_{Z^0}) = \{(5, 5, 5)\}$ and hence, $Z^0$ is $3$-semistable.

The stabilizer of $Z^0_{(a:b)}$ contains a unique one-parameter subgroup given by $\lambda = (2, -1, -1)$, i.e., $Z^0_{(a:b)}$ cannot be stable. Assume that $Z^0_{(a:b)}$ is $3$-unstable. By Remark \ref{rmk:tori_to_check} there is a base change $\overline{x} = x$, $\overline{y} = y + sz$, $\overline{z} = z$ for some $s \in \C$ such that
\[
\II_{\overline{Z}^0_{(a:b)}} = (\overline{x}\overline{z}, \overline{x}\overline{y}^2, (\overline{y} + \overline{z})(\overline{y} - \overline{z})(a\overline{y} + b\overline{z})) = (xz, xy^2, (y + sz + z)(y + sz - z)(ay + asz + bz))
\]
is destabilized by a diagonal one-parameter subgroup $\lambda = (r_1, r_2, r_3)$.
We can compute
\[
(y + sz + z)(y + sz - z)(ay + asz + bz) = ay^3 + (3as + b)y^2z + (3as^2 + 2bs - a)yz^2 + (as^3 + bs^2 - as - b)z^3.
\]
Since $(a, b) \neq (0, 0)$, a straightforward computation shows $(5, 5, 5) \in \State_3(\II_{\overline{Z}^0_{(a:b)}})$ regardless of $a, b, s$, and therefore, $\mu_3(\overline{Z}^0_{(a:b)}, \lambda) \leq 0$, a contradiction.

Next, let $Z$ be an arbitrary strictly $3$-semistable point with positive dimensional stabilizer. In particular, the set-theoretic support must have non-trivial stabilizer as well. By Lemma \ref{lem:n5_triple_point_unstable} there can be no subscheme of length three set-theoretically supported at a single point. That means $Z$ has to be set-theoretically supported at either three points, four points such that at least three of them lie on a line, or five points such that at least four of them lie on a line. By Lemma \ref{lem:n5_four_points_line_unstable} no subscheme of length four can lie on a line. There are few configurations satisfying these conditions.
\begin{enumerate}
    \item Assume the set-theoretic support of $Z$ consists of four points $P_1, P_2, P_3, P_4$ such that $P_1, P_2, P_3 \in L$ and $P_4 \notin L$. We may assume that $Z$ has length two at either $P_1$ or $P_4$. If the double point is supported at $P_4$, then we are dealing with one of the cases $Z^0_{(a:b)}$. 
    
    Assume that $P_1$ is the double point. Then it cannot be scheme-theoretically contained in $L$, and it is contained in a unique different line $L'$. If $P_4 \notin L'$, then $Z$ has finite stabilizer. Therefore, we know that $P_4 \in L'$. This subscheme is $3$-unstable by Lemma \ref{lem:n5_non_reduced_special_config_unstable_I}.
    \item Assume the set-theoretic support of $Z$ consists of three points $P_1, P_2, P_3$ all lying on a line $L$. Without loss of generality, we may assume that $P_1$ and $P_2$ support double points of $Z$. Both of these cannot be scheme-theoretically contained in $L$. Hence, they are contained in separate lines $L_1, L_2$. However, these subschemes are $3$-unstable by Lemma \ref{lem:n5_non_reduced_special_config_unstable_III}.
    \item Lastly, assume that the set-theoretic support of $Z$ consists of three general points $P_1, P_2, P_3$ and that $P_1, P_2$ support double points in $Z$. Then $P_1 \in L_1$ and $P_2 \in L_2$ for unique lines $L_1, L_2$. Since no line can contain a length four subscheme of $Z$, we must have $L_1 \neq L_2$.
    
    Assume $P_2 \in L_1$. Then $P_3 \notin L_1$. If furthermore $P_3 \notin L_2$, then $Z$ is $3$-unstable by Lemma \ref{lem:n5_non_reduced_special_config_unstable_II}. If $P_3 \in L_2$, we are dealing with a further degeneration and the fact that the locus of unstable points is closed finishes the argument. The case $P_1 \in L_2$ is ruled out by a symmetric argument.
    
    Assume that $P_1 \notin L_2$ and $P_2 \notin L_1$. If further $P_3 \in L_1$ or $P_3 \in L_2$, then we are dealing with either $Z^0$ or $Z^0_{(1:1)}$. Assume that $P_3 \notin L_1 \cup L_2$. Let $P_4 = L_1 \cap L_2$. Then the scheme $Z' = Z \cup \{P_4\}$ has the same stabilizer as $Z$. But the set-theoretic support of $Z'$ consists of four general points, and hence, the stabilizer is finite. \qedhere
\end{enumerate}
\end{proof}

We define subschemes $Z^+_{(a:b)}$ and $Z^-_{(a:b)}$ for $(a:b) \in \P^1$ through their ideals
\begin{align*}
\II_{Z^+_{(a:b)}} &= (x, (y+z)(y-z)(ay + bz)) \cap (z, y(x-y)), \\
\II_{Z^-_{(a:b)}} &= (x, (y+z)(y-z)) \cap (x + z, ay + bz) \cap (z, y^2), \\
\II_{Z^-} &= (x, y^2) \cap (z, y^2) \cap (x + z, y + z).
\end{align*}

It will turn out that the subschemes $Z^+_{(a:b)}$ are the ones stable for $m > 3$ and destabilized at $m = 3$. The subschemes $Z^-_{(a:b)}$ for $(a:b) \notin \{ (0:1), (1:1), (1:-1) \}$ and $Z^-$ turn out to be the ones stable for $2 < m < 3$ that destabilize at $m = 3$. Note that $Z^-_{(0:1)} \in \SL_3 \cdot Z^0_{(0:1)}$. Moreover, it will turn out that $Z^-_{(1:1)}$ is $3$-unstable. This can be directly computed using the one-parameter subgroup $\lambda = (-2, 1, 1)$, but it is also a consequence of the following two lemmas.

\begin{lem}
\label{lem:n_5_destabilized}
Let $Z \in \P^{5[n]}$ be a $3$-semistable subscheme with $Z^0_{(a:b)} \in \overline{\SL_3 \cdot Z}$ for $(a : b) \in \P^1$ and $(a:b) \notin \{ (0:1), (1:1), (1:-1)\}$. Then $Z$ is projectively equivalent to $Z^+_{(a:b)}$, $Z^-_{(a:b)}$, or $Z^0_{(a:b)}$.
\end{lem}

\begin{proof}
There is a one-parameter subgroup $\lambda$ such that $\lim_{t \to 0} \lambda(t) \cdot Z \in \SL_3 \cdot Z^0_{(a:b)}$. Up to change of coordinates we can assume that $\lim_{t \to 0} \lambda(t) \cdot Z = Z^0_{(a:b)}$. Then $\lambda(t)$ maps to the stabilizer of $Z^0_{(a:b)}$ and hence, we can assume $\lambda = (2, -1, -1)$ or $\lambda = (-2, 1, 1)$. Since $Z^0_{(a:b)}$ has one double point and three reduced points, we know that either the same holds for $Z$ or $Z$ is reduced.
\begin{enumerate}
    \item Assume that $\lambda = (-2, 1, 1)$. For a point $(x:y:z) \in \P^2$, we can compute
    \[
    \lim_{t \to 0} \lambda(t) \cdot (x:y:z) = \lim_{t \to 0} (t^2 x: t^{-1} y: t^{-1} z) =
    \begin{cases}
    (0: y: z) & \text{, if $y \neq 0$ or $z \neq 0$,} \\
    (1:0:0) & \text{, if $y = z = 0$.}
    \end{cases}
    \]
    This means the only point converging to $(1:0:0)$ is the point itself. Moreover, $\lambda$ fixes any line through $(1:0:0)$. Therefore, $\lim_{t \to 0} \lambda(t) \cdot Z = Z^0_{(a:b)}$ implies that $Z$ must contain the subscheme cut out by $(z, y^2)$. The three reduced points have to be given by $(x_1:1:1)$, $(x_2:1:-1)$, $(x_3:-b:a)$ for some $x_1, x_2, x_3 \in \C$. Base changing via the matrix
    \[
    \begin{pmatrix}
    1 & \frac{-x_1 - x_2}{2} & \frac{x_2 - x_1}{2} \\
    0 & 1 & 0 \\
    0 & 0 & 1
    \end{pmatrix}^{-1}
    \] 
    reduces to the case $x_1 = x_2 = 0$. If additionally $x_3 = 0$, then $Z = Z^0_{(a:b)}$. Assume that $x_3 \neq 0$. Then $Z$ is projectively equivalent to $Z^-_{(a:b)}$ which can be seen by acting with $\lambda(t)$ for appropriate $t$.
    \item Assume that $\lambda = (2, -1, -1)$.  For a point $(x:y:z) \in \P^2$, we can compute
        \[
        \lim_{t \to 0} \lambda(t) \cdot (x:y:z) = \lim_{t \to 0} (t^{-2} x: ty: tz) =
        \begin{cases}
        (1: 0: 0) & \text{, if } x \neq 0, \\
        (0:y:z) & \text{, if } x = 0.
        \end{cases}
        \]
    \begin{enumerate}
        \item Assume further that $Z$ is reduced and its five points are given by $P_i = (x_i:y_i:z_i)$ for $i = 1, \ldots, 5$. Then up to permuting the $P_i$ we must have $P_1 = (0:1:1)$, $P_2 = (0:1:-1)$, $P_3 = (0:-b:a)$, $P_4 = (1:y_4:z_4)$ and $P_5 = (1:y_5:z_5)$. The unique line through $P_4$ and $P_5$ is defined by
        \[
        (y_5z_4 - y_4z_5)x + (z_5 - z_4)y + (y_4 - y_5)z = 0.
        \]
        Therefore, the fact $\lim_{t \to 0} \lambda(t) \cdot Z = Z^0_{(a:b)}$ implies $z_4 = z_5$ and $y_4 \neq y_5$. We get $Z \in \SL_3 \cdot Z^+_{(a:b)}$.
        
        \item Assume that the set-theoretic support of $Z$ is given by four points $P_i = (x_i:y_i,z_i)$ for $i = 1, \ldots, 4$, and that $Z$ has a double point at $P_4$ scheme-theoretically contained in a line $L$. Again $\lim_{t \to 0} \lambda(t) \cdot Z = Z^0_{(a:b)}$ implies that up to permuting the points, we have $P_1 = (0:1:1)$, $P_2 = (0:1:-1)$, $P_3 = (0:-b:a)$, and $P_4 = (1:y_4:z_4)$. Next, $\lim_{t \to 0} \lambda(t) \cdot L = \{ z = 0 \}$ and $(1:y_4:z_4) \in L$ are equivalent to $L$ being cut out by $-z_4x + z = 0$. From here we get $Z \in \SL_3 \cdot Z^0_{(a:b)}$. \qedhere
    \end{enumerate}
\end{enumerate}
\end{proof}

\begin{lem}
\label{lem:n_5_destabilized_to_Z^0}
Let $Z \in \P^{5[n]}$ be a $3$-semistable subscheme such that $Z^0 \in \overline{\SL_3 \cdot Z}$. Then $Z$ is projectively equivalent to either $Z^0$, $Z^0_{(1:1)}$, $Z^0_{(0:1)}$, $Z^+_{(1:1)}$, $Z^+_{(0:1)}$, or $Z^-$.
\end{lem}

\begin{proof}
There is a non-trivial one-parameter subgroup $\lambda$ such that $\lim_{t \to 0} \lambda(t) \cdot Z \in \SL_3 \cdot Z^0$. Up to change of coordinates we can assume that $\lim_{t \to 0} \lambda(t) \cdot Z = Z^0$. Then $\lambda(t)$ maps to the stabilizer of $Z^0$ and hence, we can assume that it is a diagonal one-parameter subgroup $\lambda = (r_1, r_2, r_3)$. Since $Z^0$ is invariant under permuting $x$ and $z$, we may assume that $r_1 \geq r_3$. Note that the underlying cycle to $Z_0$ is given by
$2(0:0:1) + (0:1:0) + 2(1:0:0)$. The argument splits into seven cases.
\begin{enumerate}
    \item Assume that $r_1 > r_2 > r_3$. Then
    \[
    \lim_{t \to 0} \lambda(t) (x:y:z) = \lim_{t \to 0} (x:t^{r_1 - r_2}y:t^{r_1 - r_3}z) = \begin{cases}
    (1:0:0) &\text{, if $x \neq 0$,} \\
    (0:1:0) &\text{, if $x = 0$ and $y \neq 0$,} \\
    (0:0:1) &\text{, if $x = y = 0$.}
    \end{cases}
    \]
    Therefore, the underlying cycle of $Z$ has to be of the form
    \[
    2(0:0:1) + (0:1:z_1) + (1:y_2:z_2) + (1:y_3:z_3)
    \]
    for $z_1, y_2, z_2, y_3, z_3 \in \C$. Let $L_1$ be the line that $2(0:0:1)$ lies on scheme-theoretically. Since $\lim_{t \to 0} L_1 = \{ x = 0 \}$ and $r_1 > r_2$, we must have $L_1 = \{ x = 0 \}$ in the first place. There are two subcases.
    \begin{enumerate}
        \item Let $(1:y_2:z_2) \neq (1:y_3:z_3)$. Then the line $L_2$ through these two points is given by the equation
        \[
        (y_3 z_2 - y_2 z_3) x + (z_3 - z_2)y + (y_2 - y_3)z = 0.
        \]
        Since $\lim_{t \to 0} \lambda(t) \cdot L_2 = \{ z = 0 \}$ and $r_1 > r_2 > r_3$, we must have $y_2 \neq y_3$. Using the group action of $\SL_3$, we can reduce to the case $z_1 = y_2 = z_2 = 0$, $y_3 = 1$, and $z_3 \in \{ 0, 1\}$. If $z_3 = 0$, we get $Z \in \SL_3 \cdot Z^0_{(0:1)}$. If $z_3 = 1$, then we get $Z \in \SL_3 \cdot Z^+_{(1:1)}$.
        \item Assume that $(1:y_2:z_2) = (1:y_3:z_3)$, and let $L_2$ be the line through this point that determines its length two scheme-structure. Then $\lim_{t \to 0} \lambda(t) \cdot L_2 = \{ z = 0 \}$ is equivalent to $(0:0:1) \notin L_2$. Using the group action of $\SL_3$, we can reduce to the case $z_1 = y_2 = z_2 = y_3 = z_3 = 0$ without changing the condition $(0:0:1) \notin L_2$. If $(0:1:0) \in L_2$, then $Z = Z^0$. If $(0:1:0) \notin L_2$, then $Z \in \SL_3 \cdot Z^0_{(1:1)}$. 
    \end{enumerate}
    
    \item Assume that $r_1 = r_2 > r_3$, i.e., up to positive scaling $\lambda = (1, 1, -2)$. Then
    \[
    \lim_{t \to 0} \lambda(t) (x:y:z) = \lim_{t \to 0} (x:y:t^3z) = \begin{cases}
    (x:y:0) &\text{, if $x \neq 0$ or $y \neq 0$,} \\
    (0:0:1) &\text{, if $x = y = 0$.}
    \end{cases}
    \]
    Therefore, the underlying cycle of $Z$ has to be of the form
    \[
    2(0:0:1) + (0:1:z_1) + (1:0:z_2) + (1:0:z_3)
    \]
    for $z_1, z_2, z_3 \in \C$. Let $L_1$ be the line that $2(0:0:1)$ lies on scheme-theoretically. Any such line is fixed by $\lambda$, and we get $L_1 = \{ x = 0 \}$. If $z_2 \neq z_3$, then the points $(1:0:z_2)$ and $(1:0:z_3)$ lie uniquely on the line cut out by $y = 0$ which is fixed by $\lambda$, and therefore, cannot converge to the line cut out by $z = 0$. Thus, $z_2 = z_3$. Let $L_2$ be the line that $2(1:0:z_2)$ is scheme theoretically contained in. The fact $\lim_{t \to 0} \lambda(t) \cdot L_2 = \{ z = 0 \}$ is equivalent to $(0:0:1) \notin L_2$. If $(0:1:z_1) \in L_2$, then $Z \in \SL_3 \cdot Z^0$. If $(0:1:z_1) \notin L_2$, then $Z \in \SL_3 \cdot Z^0_{(1:1)}$.
    
    \item Assume that $r_2 > r_1 > r_3$. Then
    \[
    \lim_{t \to 0} \lambda(t) (x:y:z) = \lim_{t \to 0} (t^{r_2 - r_1}x:y:t^{r_2 - r_3}z) = \begin{cases}
    (0:1:0) &\text{, if $y \neq 0$,} \\
    (1:0:0) &\text{, if $y = 0$ and $x \neq 0$,} \\
    (0:0:1) &\text{, if $x = y = 0$.}
    \end{cases}
    \]
    Therefore, the underlying cycle of $Z$ has to be of the form
    \[
    2(0:0:1) + (x_1:1:z_1) + (1:0:z_2) + (1:0:z_3)
    \]
    for $x_1, z_1, z_2, z_3 \in \C$. Let $L_1$ be the line that $2(0:0:1)$ is scheme-theoretically contained in. Then $\lim_{t \to 0} \lambda(t) \cdot L_1 = \{ x = 0 \}$ is equivalent to $(1:0:z_2) \notin L_2$. If $z_2 \neq z_3$, then the unique line through $(1:0:z_2)$ and $(1:0:z_3)$ is cut out by $y = 0$ and this line is invariant under the action of $\lambda$, i.e., $L_2$ cannot converge to $\{ z = 0 \}$. Thus, $z_2 = z_3$ and let $L_2$ be the line in which $2(1:0:z_2)$ is scheme-theoretically contained in. The fact $\lim_{t \to 0} \lambda(t) \cdot L_2 = \{ z = 0 \}$ is equivalent to $(0:0:1) \notin L_2$. If $(x_1:1:z_1) \in L_1 \cap L_2$, then $Z \in \SL_3 \cdot Z^0$. If $(x_1:1:z_1) \in L_1 \backslash L_2 \cup L_2 \backslash L_1$, then $Z \in \SL_3 \cdot Z^0_{(1:1)}$. If $(x_1:1:z_1) \notin L_1 \cup L_2$, then $Z \in \SL_3 \cdot Z^-$.
    
    \item Assume that $r_1 > r_2 = r_3$, i.e., up to positive scaling $\lambda = (2, -1, -1)$. Then
    \[
    \lim_{t \to 0} \lambda(t) (x:y:z) = \lim_{t \to 0} (x:t^3y:t^3z) = \begin{cases}
    (1:0:0) &\text{, if $x \neq 0$,} \\
    (0:y:z) &\text{, if $x = 0$.}
    \end{cases}
    \]
    Therefore, the underlying cycle of $Z$ has to be of the form
    \[
    2(0:0:1) + (0:1:0) + (1:y_1:z_1) + (1:y_2:z_2)
    \]
    for $y_1, z_1, y_2, z_2 \in \C$. Let $L_1$ be the line in which $2(0:0:1)$ is scheme-theoretically contained in. The fact $\lim_{t \to 0} \lambda(t) \cdot L_1 = \{ x = 0\}$ implies that $L_1$ itself is already cut out by $x = 0$. We will deal with two subcases individually.
    \begin{enumerate}
        \item Assume that $(1:y_1:z_1) \neq (1:y_2:z_2)$. Then the unique line $L_2$ through these two points is cut out by
        \[
        (y_1z_2 - y_2z_1)x + (z_1 - z_2)y + (y_2 - y_1) z = 0.
        \]
        The fact $\lim_{t \to 0} \lambda(t) \cdot L_2 = \{ x = 0\}$ is equivalent to $y_1 \neq y_2$ and $z_1 = z_2$. Then $Z \in \SL_3 \cdot Z^0_{(0:1)}$.
        \item Assume that $y_1 = y_2$ and $z_1 = z_2$, and let $L_2$ be the line in which $2(1:y_1:z_1)$ is scheme-theoretically contained in. Then $\lim_{t \to 0} \lambda(t) \cdot L_2 = \{ z = 0 \}$ is equivalent to $(0:1:0) \in L_2$ and $(0:0:1) \notin L_2$. Therefore, $Z \in \SL_3 \cdot Z^0$.
    \end{enumerate}
    
    \item Assume that $r_1 > r_3 > r_2$. Then 
    \[
    \lim_{t \to 0} \lambda(t) (x:y:z) = \lim_{t \to 0} (x:t^{r_1 - r_2}y:t^{r_1 - r_3}z) = \begin{cases}
    (1:0:0) &\text{, if $x \neq 0$,} \\
    (0:0:1) &\text{, if $x = 0$ and $z \neq 0$,} \\
    (0:1:0) &\text{, if $x = z = 0$.}
    \end{cases}
    \]
    Therefore, the underlying cycle of $Z$ has to be of the form
    \[
    (0:y_1:1) + (0:y_2:1) + (0:1:0) + (1:y_3:z_3) + (1:y_4:z_4)
    \]
    for $y_1, y_2, y_3, z_3, y_4, z_4 \in \C$. If $y_1 \neq y_2$, then the unique line $L_1$ through $(0:y_1:1)$ and $(0:y_2:1)$ is cut out by $x = 0$ which is fixed by $\lambda$. If $y_1 = y_2$ and $L_1$ is the line in which $2(0:y_1:1)$ is scheme-theoretically contained in, then $L_1$ has to be cut out by $x = 0$. Indeed, no other other lines converges via $\lambda(t)$ to the line cut out by $x = 0$. We deal with the remainder of this case in two subcases.
    \begin{enumerate}
        \item Assume that $(1:y_3:z_3) \neq (1:y_4:z_4)$. Then the unique line $L_2$ through these two points is defined by
        \[
        (y_3z_4 - y_4z_3)x + (z_3 - z_4)y - (y_3 - y_4)z = 0.
        \]
        The fact $\lim_{t \to 0} \lambda(t) \cdot L_2 = \{ z = 0\}$ is equivalent to $z_3 = z_4$. In particular, this means $y_3 \neq y_4$ in this case. If $y_1 \neq y_2$, then we get $Z \in \SL_3 \cdot Z^+_{(0:1)}$. If $y_1 = y_2$, then $Z \in \SL_3 \cdot Z^0_{(0:1)}$.
        \item Assume that $y_3 = y_4$ and $z_3 = z_4$. Let $L_2$ be the line in which $2(1:y_3:z_3)$ is scheme-theoretically contained in. Then $\lim_{t \to 0} \lambda(t) \cdot L_2 = \{ z = 0\}$ is equivalent to $(0:1:0) \in L_2$, but $(0:y_1:1), (0:y_2:1) \notin L_2$. If $y_1 \neq y_2$, then $Z \in \SL_3 \cdot Z^0_{(1:1)}$. If $y_1 = y_2$, then $Z \in \SL_3 \cdot Z^0$.
    \end{enumerate}
    
    \item Assume that $r_2 > r_1 = r_3$, i.e., up to positive scale $\lambda = (-1:2:-1)$. Then
    \[
    \lim_{t \to 0} \lambda(t) (x:y:z) = \lim_{t \to 0} (t^3x:y:t^3z) = \begin{cases}
    (0:1:0) &\text{, if $y \neq 0$,} \\
    (x:0:z) &\text{, if $y = 0$.}
    \end{cases}
    \]
    Therefore, the underlying cycle of $Z$ has to be of the form
    \[
    2(0:0:1) + (x_1:1:z_1) + 2(1:0:0)
    \]
    for $x_1, z_1 \in \C$. Let $L_1$ be the line on which $2(0:0:1)$ lies scheme-theoretically. Then $\lim_{t \to 0} \lambda(t) \cdot L_1 = \{ x = 0\}$ is equivalent to $(1:0:0) \notin L_1$. Let $L_2$ be the line on which $2(1:0:0)$ lies scheme-theoretically. Then $\lim_{t \to 0} \lambda(t) \cdot L_1 = \{ z = 0\}$ is equivalent to $(0:0:1) \notin L_2$. If $(0:1:0) \in L_1 \cap L_2$, then $Z \in \SL_3 \cdot Z^0$. If $(0:1:0) \in (L_1 \cup L_2) \backslash (L_1 \cap L_2)$, then $Z \in \SL_3 \cdot Z^0_{(1:1)}$. If $(0:1:0) \notin L_1 \cup L_2$, then $Z \in \SL_3 \cdot Z^-$.
    \item Lastly, assume that $r_1 = r_3 > r_2$, i.e., up to positive scale $\lambda = (1:-2:1)$. Then
    \[
    \lim_{t \to 0} \lambda(t) (x:y:z) = \lim_{t \to 0} (x:t^3y:z) = \begin{cases}
    (x:0:z) &\text{, if $x \neq 0$ or $z \neq 0$,} \\
    (0:1:0) &\text{, if $x = z = 0$.}
    \end{cases}
    \]
    Therefore, the underlying cycle of $Z$ has to be of the form
    \[
    (0:y_1:1) + (0:y_2:1) + (0:1:0) + (1:y_3:0) + (1:y_4:0)
    \]
    for $y_1, y_2, y_3, y_4 \in \C$. If $y_1 \neq y_2$, then the unique line through $(0:y_1:1)$ and $(0:y_2:1)$ is already cut out by $x = 0$. If $y_1 = y_2$, then the line $L_1$ in which $2(0:y_1:1)$ is scheme-theoretically contained in is also cut out by $x = 0$ already. Indeed, no other lines converges correctly via $\lambda$. The same argument exchanging the roles of $x$ and $z$ shows that $(1:y_3:0) + (1:y_4:0)$ is contained in the line cut out by $z = 0$ even if $y_3 = y_4$.
    
    If $y_1 = y_2$ and $y_3 = y_4$, then $Z \in \SL_3 \cdot Z^0$. If $y_1 \neq y_2$ and $y_3 = y_4$, then $Z \in \SL_3 \cdot Z^0_{(0:1)}$. Similarly, if $y_1 = y_2$ and $y_3 \neq y_4$, then $Z \in \SL_3 \cdot Z^0_{(0:1)}$ as well. If $y_1 \neq y_2$ and $y_3 \neq y_4$, then $Z \in \SL_3 \cdot Z^+_{(0:1)}$. \qedhere
\end{enumerate}
\end{proof}

For $0 < \varepsilon \ll 1$, we define $m_{-} = 3 - \varepsilon$, $m_{+} = 3 + \varepsilon$, and $m_0 = 3$. We have morphisms $f^+: X /\!\!/_{m_+} \SL(3) \to X /\!\!/_{m_0} \SL(3)$ and $f^-: X /\!\!/_{m_-} \SL(3) \to X /\!\!/_{m_0} \SL(3)$.

\begin{lem}
\label{lem:stability_3_minus_epsilon}
The morphism $f^+$ and $f^-$ are isomorphisms. Moreover, $Z \in \P^{2[5]}$ is $m_{-}$-stable if and only if no subscheme of length $3$ lies in a single line and there is no subscheme of length $3$ supported at a single point. Lastly, any $m_{-}$-semistable subscheme $Z$ is $m_{-}$-stable. 
\end{lem}

\begin{proof}
Note that the GIT quotient of a normal variety is normal. Therefore, showing that the morphisms are bijective is enough to obtain isomorphisms.

If a subscheme $Z$ contains no double points and no three points on a line, then Lemma \ref{lem:n_5_destabilized} and Lemma \ref{lem:n_5_destabilized_to_Z^0} show that $Z$ is $m_0$-stable, and the morphism $f^+$ and $f^-$ are certainly both isomorphisms on this locus.

Assume that $Z$ is $m_{+}$-stable, and contained in the fiber $(f^+)^{-1}(Z^0)$. This means $Z$ contains no double points, and Lemma \ref{lem:n_5_destabilized_to_Z^0} implies that $Z$ is projectively equivalent to $Z^+_{(0:1)}$. Indeed, the fiber is a single point.

Assume that $Z$ is $m_{+}$-stable, and contained in the fiber $(f^+)^{-1}(Z^0_{(a:b)})$ for $(a:b) \notin \{ (1:1), (1:-1), (0:1)\}$. Again, $Z$ contains no double points, and by Lemma \ref{lem:n_5_destabilized} it is projectively equivalent to $Z^+_{(a:b)}$. Indeed, the fiber is a single point. By Lemma \ref{lem:closed_orbits_at_m_3} we covered all fibers, i.e., $f^+$ is an isomorphism.

Next, we analyze $f^-$. Since these GIT quotients are all proper, its fibers have to contain at least one point. Assume that $Z$ is $m_{-}$-stable, and contained in the fiber $(f^-)^{-1}(Z^0)$. By Lemma \ref{lem:n5_three_points_line_unstable} we know that $Z$ contains no subscheme of length $3$ in a line. Due to Lemma \ref{lem:n_5_destabilized_to_Z^0} the only orbit left is the one containing $Z^-$. In particular, $Z^-$ has to be $m_{-}$-semistable.

Assume that $Z$ is $m_{-}$-stable, and contained in the fiber $(f^+)^{-1}(Z^0_{(a:b)})$ for $(a:b) \notin \{ (1:1), (1:-1), (0:1)\}$. By Lemma \ref{lem:n5_three_points_line_unstable} we know that $Z$ contains no subscheme of length $3$ in a line and by Lemma \ref{lem:n_5_destabilized} this is only possible if $Z$ is projectively equivalent to $Z^{-}_{(a:b)}$. In particular, $Z^{-}_{(a:b)}$ has to be $m_{-}$-semistable.

The description of $m_{-}$-semistable points follows from the fact that if $Z$ is a subscheme with no triple points, at least one double point, and no subscheme of length three supported on a line, then it is projectively equivalent to either $Z^-$ or $Z^-_{(a:b)}$ for some $(a:b) \notin \{ (0:1), (1:1), (1:-1) \}$. If a subscheme $Z$ is strictly $m_{-}$-semistable, then there must be a strictly $m_{-}$-semistable point with positive dimensional stabilizer. However, all of the $m_{-}$-semistable $Z$ have zero-dimensional stabilizer.
\end{proof}

\begin{lem}
There is no VGIT wall for $2 < m < 3$.
\end{lem}

\begin{proof}
If not, let $2 < m_0 < 3$ be maximal such that it is a VGIT wall. Then there is a GIT-unstable $Z \in  \P^{2[5]}$ for $m_0 < m < 3$ that is $m_0$-semistable. By Lemma \ref{lem:stability_3_minus_epsilon} this means $Z$ contains either a triple point or three points on a line. However, by Lemmas \ref{lem:n5_triple_point_unstable} and  \ref{lem:n5_three_points_line_unstable} they all stay unstable.
\end{proof}

\mainNFive*

\begin{proof}
By the previous lemmas we understand the GIT stability conditions for $2 < m \leq \infty$. By \cite[Section 11.2]{Dol03:invariant_theory} there is an isomorphism
\[
    (\P^{1})^5 / \! \!/_{\mathcal{O}(1^5)} \SL_2 \cong (\P^{2})^5 / \! \!/_{\mathcal{O}(1^5)} \SL_3
\]
that is equivariant under the action of the symmetric group $S_5$. Together with \cite[Section 1]{dPW18:moduli_binary_quintics} we get
\[
    \P^{2(5)} / \! \!/ \SL_3 \cong \P^{1(5)} /\!\!/ \SL_2 \cong  \P(1, 2, 3).
\]
At $m = 2$, the moduli space of S-equivalence classes of Bridgeland-semistable objects is the moduli space of conics. Its quotient by $\SL_3$ is a single point parametrizing the orbit of smooth conics.
\end{proof}


\section{The final model for seven points}
\label{sec:FinalModel7}

\subsection{The setup}
\label{subsec:n_7_final_model_setup}

According to \cite[Section 10.6]{ABCH13:hilbert_schemes_p2} the last wall in Bridgeland stability for $n = 7$ destabilizes objects $E$ that fit into short exact sequences
\[
    0 \to \Omega(-1) \to E \to \OO(-5)[1] \to 0
\]
in the category $\Coh^{\beta}(\P^2)$ for any $\beta$ such that
there is a point $(\alpha, \beta)$ on the numerical wall $W(\Omega(-1), \OO(-5)[1])$.
Therefore, right above this wall the moduli space of Bridgeland-semistable objects is given as the space those extensions that are non-trivial. Since $\Ext^1(\OO(-5)[1], \Omega(-1)) = H^0(\Omega(4))$, this means the moduli space in the chamber above the wall is given by $\P(H^0(\Omega(4)))$. Note that \cite{ABCH13:hilbert_schemes_p2} also shows that the movable cone of $\P^{2[7]}$ is generated by $H$ and $D_{5/2}$, but the effective cone is larger and only ends at $D_{12/5}$. There are no further walls between these $D_{5/2}$ and $D_{12/5}$. Therefore, the variety $\P(H^0(\Omega(4)))$ is the image of a divisorial contraction occurring at $D_{5/2}$. 

The Euler sequence is given by
$0 \to \Omega(4) \to \OO(3)^{\oplus 3} \to \OO(4) \to 0$, where the second map is defined through $(f, g, h) \mapsto xf + yg + zh$. Therefore, the global sections $H^0(\Omega(4))$ can be identified with triples of cubics $(f, g, h)$ such that $xf + yg + zh = 0$. From this description we obtain $H^0(\Omega(4)) = \C^{15}$, i.e., $\P(H^0(\Omega(4))) = \P^{14}$. For a point $s \in H^0(\Omega(4)) \backslash \{0\}$ we denote the corresponding point in $\P(H^0(\Omega(4)))$ by $[s]$.

\begin{lem}
Let $(f, g, h) \in H^0(\Omega(4))$ be general. More precisely, assume that $(f, g, h) \neq 0$ and that the cokernel of the induced morphism $\OO(-5) \to \Omega(-1)$ is a torsion-free sheaf. Then this cokernel is the ideal sheaf corresponding to the homogeneous ideal $(f, g, h) \subset \C[x,y,z]$.
\end{lem}

\begin{proof}
If $E$ is this cokernel, then $\ch(E) = \ch(\Omega(-1)) - \ch(\OO(-5)) = (1, 0, -7)$. 
Since $E$ is assumed to be torsion-free, its Chern character implies that it has to be an ideal sheaf $\II_Z$ of a zero-dimensional subscheme $Z \subset \P^2$ of length $7$.  
We are left to show that it is cut out by $(f, g, h)$.

We have two inclusions $\OO \into \Omega(4) \into \OO(3)^{\oplus 3}$ induced by the choice of $(f, g, h)$. We get the two quotients $\Omega(4)/ \OO = \II_Z(5)$ and $\OO(3)^{\oplus 3}/\Omega(4) = \OO(4)$. By calling $F$ the remaining quotient $\OO(3)^{\oplus 3} / \OO$, we get an induced short exact sequence
$0 \to \II_Z(5) \to F \to \OO(4) \to 0$. The long exact sequence from dualizing this short exact sequence is given by $0 \to \OO(-4) \to F^{\vee} \to \OO(-5) \to 0 \to \lExt^1(F, \OO) \to \OO_Z \to 0$. In particular, $\lExt^1(F, \OO) = \OO_Z$. Since $F$ is the cokernel of $\OO \to \OO(3)^{\oplus 3}$, we can dualize this as well to get that $\lExt^1(F, \OO)$ is the cokernel of the morphism $\OO(-3)^{\oplus 3} \to \OO$ given by $(f, g, h)$. Indeed, the image of this morphism is the ideal generated by $(f, g, h)$.
\end{proof}

We define a basis of $H^0(\Omega(4))$ as follows:
\begin{alignat*}{7}
e_0 &\coloneqq (x^2y, -x^3, 0), \ & e_4 &\coloneqq (xz^2, 0, -x^2z), \ & e_8 &\coloneqq (y^2z, -xyz, 0), \ & e_{12} &\coloneqq (0, y^2z, -y^3), \\
e_1 &\coloneqq (xy^2, -x^2y, 0), \ & e_5 &\coloneqq (z^3, 0, -xz^2), \ & e_9 &\coloneqq (y^2z, 0, -xy^2), \ & e_{13} &\coloneqq (0, yz^2, -y^2z), \\
e_2 &\coloneqq (y^3, -xy^2, 0), \ & e_6 &\coloneqq (xyz, -x^2z, 0), \ & e_{10} &\coloneqq (yz^2, -xz^2, 0), \ & e_{14} &\coloneqq (0, z^3, -yz^2). \\
e_3 &\coloneqq (x^2z, 0, -x^3), \ & e_7 &\coloneqq (xyz, 0, -x^2y), \ & e_{11} &\coloneqq (yz^2, 0, -xyz), & &
\end{alignat*}
Since $\Omega(4)$ is an $\SL_3$-equivariant vector bundle, we get an induced action of $\SL_3$ on $H^0(\Omega(4))$. The next step is to understand this action. Note that the morphisms in the Euler-sequence are not $\SL_3$ equivariant, and we cannot simply take the kernel of a map of $\SL_3$-representations. The key is that when the Euler-sequence changes after acting with $A \in \SL_3$, then $\Omega(4)$ changes as a sub-vector bundle of $\OO(3)^{\oplus 3}$. To get the correct action, we need to move it back. More precisely, let $(f, g, h) \in H^0(\OO(3)^{\oplus 3})$ sastisfying $xf + yg + zh = 0$, and let $A \in \SL_3$. Then $A \cdot f$, $A \cdot g$, and $A \cdot h$ satisfy the linear relation
\[
    \begin{pmatrix} x & y & z \end{pmatrix} A \begin{pmatrix} A \cdot f \\ A \cdot g \\ A \cdot h\end{pmatrix} = \begin{pmatrix} A \cdot x & A \cdot y & A \cdot z \end{pmatrix} \begin{pmatrix} A \cdot f \\ A \cdot g \\ A \cdot h\end{pmatrix} = 0.
\]
This means the action of $\SL_3$ on $H^0(\Omega(4))$ is given by
\[
    A \cdot (f, g, h) \coloneqq A \begin{pmatrix}
    A \cdot f \\
    A \cdot g \\
    A \cdot h
    \end{pmatrix}.
\]

With this in mind, we can determine the action of one-parameter subgroups on our chosen basis.

\begin{lem}
For any $r \in [-\tfrac{1}{2}, 1]$ and $t \in \C$ we have
\begin{alignat*}{7}
\lambda_r(t) e_0 &= t^{r + 3} e_0, \ & \lambda_r(t) e_4 &= t^{-2r} e_4, \ & \lambda_r(t) e_8 &= t^r e_8, \ & \lambda_r(t) e_{12} &= t^{2r - 1} e_{12}, \\
\lambda_r(t) e_1 &= t^{2r + 2} e_1, \ & \lambda_r(t) e_5 &= t^{-3r - 2} e_5, \ & \lambda_r(t) e_9 &= t^r e_9, \ & \lambda_r(t) e_{13} &= t^{-2} e_{13}, \\
\lambda_r(t) e_2 &= t^{3r + 1} e_2, \ & \lambda_r(t) e_6 &= t e_6, \ & \lambda_r(t) e_{10} &= t^{-r - 1} e_{10}, \ & \lambda_r(t) e_{14} &= t^{-2r - 3} e_{14}. \\
\lambda_r(t) e_3 &= t^{-r + 2} e_3, \ & \lambda_r(t) e_7 &= t e_7, \ & \lambda_r(t) e_{11} &= t^{-r - 1} e_{11}, & & 
\end{alignat*}
\end{lem} 

\begin{proof}
Compute according to the group action.
\end{proof}


\subsection{Unstable and non-stable loci}

The goal of this section is to prove Theorem \ref{thm:stability_final_model_n_7}. Recall from the introduction that we have three subloci $X_1$, $X_2$, and $X_3$ of $\P^{14}$, where $X_1 \cup X_2 \subset X_3$.

\begin{lem}
\label{lem:stability_final_n7}
Let $s \in H^0(\Omega(4)) \backslash \{ 0 \}$. Then $[s]$ is unstable if and only if there is $a_0 e_0 + \ldots + a_{14} e_{14} \in \SL_3 \cdot s$ such that $a_5 = a_{10} = a_{11} = a_{13} = a_{14} = 0$ and additionally one the following two holds:
\begin{enumerate}
    \item $a_8 = a_9 = a_{12} = 0$, or
    \item $a_4 = 0$.
\end{enumerate}
Moreover, $[s]$ is not stable if and only if it is either unstable or there is a point $a_0 e_0 + \ldots + a_{14} e_{14} \in \SL_3 \cdot s$ such that $a_5 = a_{10} = a_{11} = a_{12} = a_{13} = a_{14} = 0$.
\end{lem}

\begin{proof}
Let $[s]$ for $s = a_0 e_0 + \ldots + a_{14} e_{14}$ be non-stable. Then there is a one-parameter subgroup $\lambda$ with $\mu([s], \lambda) \geq 0$. Up to a change of coordinates, we can assume that $\lambda = \lambda_r$ for $r \in [-\tfrac{1}{2}, 1]$.

The action of $\lambda_r$ on $e_5$ is multiplication by $t^{-3r - 2}$ and since $-3r - 2 < 0$, the numerical criterion implies $a_5 = 0$. We proceed in the same manner for other coefficients: Since $-r - 1 < 0$, we must have $e_{10} = e_{11} = 0$. Due to $-2 < 0$, we obtain $a_{13} = 0$. Finally, $-2r - 3 < 0$ implies $a_{14} = 0$.

Next, let $r < 0$. Then $a_8 = a_9 = 0$. Moreover, $2r - 1 < 0$ implies $a_{12} = 0$, and we are in case (i). On the other hand, if $r > 0$, then $-2r < 0$. This implies $a_4 = 0$ and we are in case (ii).

Lastly, let $r = 0$. If $[s]$ is unstable, then $\lambda_0$ acts trivially on $e_4$, we get $a_4 = 0$, and this is case (ii). If $s$ is strictly-semistable, then we only get $2r - 1 < 0$, i.e., $a_{12} = 0$, and this is the final case.

Assume vice-versa $s = a_0 e_0 + \ldots + a_{14} e_{14}$ such that $a_5 = a_{10} = a_{11} = a_{13} = a_{14} = 0$.
\begin{enumerate}
    \item Assume that $a_8 = a_9 = a_{12} = 0$ as well. Then for any $r \in (-\tfrac{1}{3}, 0)$ the inequality $\mu([s], \lambda_r) = \min \{ r + 3, 2r + 2 , 3r + 1, -r + 2, -2r, 1, 1\} > 0$ holds, i.e., $[s]$ is unstable.
    \item Assume that $a_4 = 0$ as well. Then $\mu([s], \lambda_1) = 1 > 0$ holds, i.e., $[s]$ is unstable.
    \item Assume that $a_{12} = 0$ as well. Then $\mu([s], \lambda_0) = 0$, i.e., $[s]$ is non-stable. \qedhere
\end{enumerate}
\end{proof}

We will have to translate what these unstable and non-stable loci are in terms of geometry. There are three relevant loci:
\begin{align*}
    \widetilde{X}_1 &\coloneqq \overline{\SL_3 \cdot \{ [a_0 e_0 + \ldots + a_4 e_4 + a_6 e_6 + a_7 e_7]: \ a_0, \ldots, a_4, a_6, a_7 \in \C \}}, \\
    \widetilde{X}_2 &\coloneqq \overline{\SL_3 \cdot \{ [a_0 e_0 + \ldots + a_3 e_3 + a_6 e_6 + \ldots + a_9 e_9 + a_{12} e_{12}]: \ a_0, \ldots, a_3, a_6, \ldots, a_9, a_{12} \in \C \}}, \\
    \widetilde{X}_3 &\coloneqq \overline{\SL_3 \cdot \{ [a_0 e_0 + \ldots + a_4 e_4 + a_6 e_6 + \ldots + a_9 e_9]: \ a_0, \ldots, a_4, a_6, \ldots, a_9 \in \C \}}.
\end{align*}

\begin{prop}
\label{prop:locus_x1}
We have $X_1 = \widetilde{X}_1$, i.e., this locus is the closure of ideal sheaves of zero-dimensional length seven subschemes that contain two general reduced points and a third point of length $5$ projectively equivalent to the one cut out by the ideal $J \coloneqq (x^2, xy^2, y^3 + xyz + xz^2)$.
\end{prop}

For any $u \in \C$, we define $K_u \coloneqq (z, ux^2 + y^2)$, and $I_u \coloneqq J \cap K_u$.

\begin{lem}
\label{lem:simplification_x1}
We have $\widetilde{X}_1 = \overline{\SL_3 \cdot \{ [u e_0 + e_2 + e_4 + e_6]: \ u \in \C \}}$. In particular, $\widetilde{X}_1 \subset \widetilde{X}_3$. Moreover, the point $[u e_0 + e_2 + e_4 + e_6]$ corresponds to the ideal $I_u$.
\end{lem}

\begin{proof}
See \cite{code}.
\end{proof}

\begin{proof}[Proof of Proposition \ref{prop:locus_x1}]
Lemma \ref{lem:simplification_x1} implies that $\widetilde{X}_1 \subset X_1$. Since both $X_1$ and $\widetilde{X}_1$ are irreducible, all we have to show is that they have the same dimension. We start by determining the dimension of the orbit of the subscheme cut out by $J$.

Let $A \in \SL_3$ with $A \cdot J = J$. Then $A$ fixes the line cut out by $x$ and the point $(0:0:1)$, i.e.,
\[
A = \begin{pmatrix} 
a & b & c \\
0 & d & e \\
0 & 0 & f
\end{pmatrix}
\]
for $a, b, c, d, e, f \in \C$ with $adf = 1$. We have
\[
A \cdot J = (x^2, xy^2, d^3y^3 + af(d + 2e)xyz + af^2xz^2) = J
\]
which is equivalent to $f = d + 2e$ and $d^3 = af^2$. This means the stabilizer of $J$ has dimension $3$, and its orbit has dimension $5$. Therefore, $\dim X_1 = 9$.

Next, assume there are $u, u' \in \C$ and $A \in \SL_3$ such that $A I_u = I_{u'}$. Then $A$ must be in the stabilizer of $J$ and fix the line cut out by $z$, i.e., 
\[
A = \begin{pmatrix}
a & b & 0 \\
0 & d & 0 \\
0 & 0 & f
\end{pmatrix}
\]
for $a, b, d, f \in \C$ with $adf = 1$, $f = d$, and $d = a$. We can compute
\[
A \cdot K_u = (z, a^2ux^2 + b^2x^2 + 2abxy + a^2y^2) = K_{u'}.
\]
This shows that $b = 0$ and $u = u'$, and thus, the ideals $I_u$ represent distinct orbits.

The same argument with $u = u'$ from the beginning shows that the stabilizer of $I_u$ has dimension $1$, and its orbit has dimension $8$. This implies $\dim \widetilde{X}_1 = 9$ as advertised.
\end{proof}

We are left to deal with the $\widetilde{X}_2$ and $\widetilde{X}_3$.

\begin{prop}
\label{prop:locus_x2}
We have $X_2 = \widetilde{X}_2$, i.e., this locus is the closure of ideal sheaves of zero-dimensional length seven subschemes that contain three general reduced points and a point of length $4$ projectively equivalent to the one cut out by the ideal $(x^2, y^2)$.
\end{prop}

For $t, u \in \C \backslash \{ 0 \}$ with $t + u \neq 0$, we define ideals
\begin{alignat*}{3}
    J_{t, u} &\coloneqq (xy, x^2 + (t + u)y^2), \ &
    K_{t, u} &\coloneqq (x, y + (t + u)z), \\
    L_{t, u} &\coloneqq (y + tz, x^2 + xy + uy^2), \ &
    I_{t, u} &\coloneqq J_{t, u} \cap K_{t, u} \cap L_{t, u}.
\end{alignat*}

The subscheme cut out by $J_{t, u}$ is a length four scheme supported at $(0:0:1)$ projectively equivalent to $(x^2, y^2)$. The subscheme cut out by $K_{t, u}$ is a reduced point, and the subscheme cut out by $L_{t, u}$ is zero-dimensional of length two. The restrictions on $t$ and $u$ are required for the following reasons. The fact that $t + u \neq 0$ implies that the subscheme cut out $K_{t, u}$ is not supported at $(0:0:1)$. Moreover, $t \neq 0$ implies that $L_{t, u}$ is not supported at $(0:0:1)$. Finally, $u \neq 0$ ensures that the support of $K_{t, u}$ and $L_{t, u}$ do not lie on a single line.

\begin{lem}
\label{lem:simplification_x2}
We have $\widetilde{X}_2 = \overline{\SL_3 \cdot \{ [e_2 + e_3 + e_7 + te_8 + ue_9]: \ t, u \in \C \}}$. In particular, $\widetilde{X}_2 \subset \widetilde{X}_3$. Moreover, the point $[e_2 + e_3 + e_7 + te_8 + ue_9]$ corresponds to the ideal $I_{t, u}$.
\end{lem}

\begin{proof}
See \cite{code}.
\end{proof}

\begin{lem}
\label{lem:orbit_x^2_y^2}
The $\SL_3$-orbit of $(x^2, y^2)$ in $\P^{2[n]}$ has dimension four.
\end{lem}

\begin{proof}
Let $A \in \SL_3$ be in the stabilizer of $(x^2, y^2)$. Then $A$ has to fix the point $(0:0:1)$, i.e.,
\[
A = \begin{pmatrix}
a & b & c \\
d & e & f \\
0 & 0 & g
\end{pmatrix}
\]
for $a, b, c, d, e, f, g \in \C$ with $\det(A) = (ae - bd)g = 1$. We have
\[
A \cdot (x^2, y^2) = (a^2x^2 + 2adxy + d^2y^2, b^2x^2 + 2bexy + e^2y^2) = (x^2, y^2).
\]
This is equivalent to $x^2 \wedge y^2$ being a multiple of
\[
2ab(ae - bd) x^2 \wedge xy + (a^2e^2 - b^2d^2) x^2 \wedge y^2 + 2de(ae - bd) xy \wedge y^2.
\]
This happens if and only if $2ab(ae - bd) = 0 = 2de(ae - bd)$ and $a^2e^2 - b^2d^2 \neq 0$. Since $ae - bd \neq 0$ as well, we can conclude that either $a = e = 0$ and $b \neq 0$, $d \neq 0$, or $b = d = 0$ and $a \neq 0$, $e \neq 0$. Overall, the stabilizer of $I$ has dimension four, and thus, the orbit has dimension four as well.
\end{proof}

\begin{lem}
\label{lem:stabilizer_x2}
Let $t, u, t', u' \in \C \backslash \{ 0 \}$ with $t + u \neq 0$, $t' + u' \neq 0$. If there is $A \in \SL_3$ such that $A \cdot J_{t, u} = J_{t', u'}$, $A \cdot K_{t, u} = K_{t', u'}$, and $A \cdot L_{t, u} = L_{t', u'}$, then $A$ is a multiple of the identity.
\end{lem}

\begin{proof}
Such an $A$ has to fix the point $(0:0:1)$ and the line cut out by $x$, i.e.,
\[
A = \begin{pmatrix}
a & b & c \\
0 & d & e \\
0 & 0 & f
\end{pmatrix}
\]
for $a, b, c, d, e, f \in \C$ with $\det(A) = adf = 1$. We have
\begin{align*}
A \cdot J_{t, u} &= (abx^2 + adxy,  (a^2 + b^2(t + u)) x^2 + 2bd(t + u)xy + d^2(t + u)y^2) \\
&= J_{t', u'} = (xy, x^2 + (t' + u')y^2).
\end{align*}
This is equivalent to the wedge products of the generators being scalar multiples of each other:
\[
xy \wedge x^2 + (t' + u') xy \wedge y^2, \ (a^3d - ab^2d(t + u)) xy \wedge x^2 + abd^2(t + u) x^2 \wedge y^2 + ad^3(t + u) xy \wedge y^2.
\]
We obtain $b = 0$ and $a^2 (u' + t') = d^2 (u + t)$. Using this vanishing of $b$, we get
\[
A \cdot L_{t, u} = (ctx + (d + et)y + ftz, a^2x^2 + adxy + d^2uy^2) = L_{t', u'} = (y + t'z, x^2 + xy + u'y^2).
\]
Comparing the two linear generators leads to $c = 0$ and $ft = (d + et)t'$. Moreover, $a^2x^2 + adxy + d^2uy^2$ and $x^2 + xy + u'y^2$ must be multiples of each other as well. This is only possible if $a = d$, and $u = u'$. Since we previously established $a^2 (u' + t') = d^2 (u + t)$, we get $t = t'$.

At this point, we use the final hypothesis
\[
A \cdot K_{t, u} = (x, (a + e(t+u))y + f(t + u)z) = K_{t, u} = (x, y + (t + u)z).
\]
This yields $a + e(t + u) = f$, and since we already showed $f = a + et$, we get $e = 0$ and $f = a$.
\end{proof}

\begin{proof}[Proof of Proposition \ref{prop:locus_x2}]
Lemma \ref{lem:simplification_x2} implies that $\widetilde{X}_2 \subset X_2$. Since both $X_2$ and $\widetilde{X}_2$ are irreducible, all we have to show is that they have the same dimension. By Lemma \ref{lem:orbit_x^2_y^2} the ideal $(x^2, y^2)$ is in an orbit of dimension four, and therefore, $\dim X_2 = 10$. By Lemma \ref{lem:simplification_x2} and Lemma \ref{lem:stabilizer_x2} we know that the points $[e_2 + e_3 + e_7 + te_8 + ue_9]$ for general $t, u \in \C$ all describe pairwise distinct orbits and their stabilizers are zero dimensional. Therefore, $\dim \widetilde{X}_2 = 10$ as well.
\end{proof}

\begin{prop}
\label{prop:locus_x3}
We have $X_3 = \widetilde{X}_3$, i.e., this locus is the closure of ideal sheaves of zero-dimensional length seven subschemes that contain a curvilinear triple point and four general reduced points such that there is a line that intersects the triple point in length two and one more of the reduced points.
\end{prop}

For general $u, v, w \in \C$, we define the ideals
\begin{align*}
J_{u, v, w} &\coloneqq (xy, x^2, uxz + (v + w)y^2), \\
K_{u, v, w} &\coloneqq (x, y + (v + w)z), \\
L_{u, v, w} &\coloneqq (x^2 + y^2 + vyz, x^2 + uxz + wy^2, (1 - w)xy + vxz + uyz + uvz^2), \\
I_{u, v, w} &\coloneqq J_{u, v, w} \cap K_{u, v, w} \cap L_{u, v, w}.
\end{align*}

The scheme cut out by $J_{u, v, w}$ is a curvilinear triple point and the reduced point cut out by $K_{u, v, w}$ has different support as long as $(v + w) \neq 0$. The line cut out by $x$ intersects $J_{u, v, w}$ in length two. The subscheme cut out by $L_{u, v, w}$ is of length three, but it is not clear at this moment whether it is general. This means that the subscheme cut out by $I_{u, v, w}$ is contained in the locus described in Proposition \ref{prop:locus_x3}.

\begin{lem}
\label{lem:simplification_x3}
We have $\widetilde{X}_3 = \overline{\SL_3 \cdot \{ [e_0 + e_2 + e_3 + ue_4 + ve_8 + we_9]: \ u, v, w \in \C \}}$. Moreover, the point $[e_0 + e_2 + e_3 + ue_4 + ve_8 + we_9]$ corresponds to the ideal $I_{u, v, w}$.
\end{lem}

\begin{proof}
See \cite{code}.
\end{proof}

\begin{lem}
\label{lem:stabilizer_x3}
Let $u, v, w, u', v', w' \in \C$ be general and let $A \in \SL_3$ be an upper-triangular matrix. Then 
\[
A \cdot [e_0 + e_2 + e_3 + ue_4 + ve_8 + we_9] = [e_0 + e_2 + e_3 + u'e_4 + v'e_8 + w'e_9]
\]
if and only if either
\begin{enumerate}
    \item $(u', v', w') = (u, v, w)$ and $A$ is a multiple of the identity, or
    
    \item $(u', v', w') = (-u, v, w)$ and $A$ is a multiple of
\[
\begin{pmatrix}
1 & 0  & 0 \\
0 & -1 & 0 \\
0 & 0  & -1
\end{pmatrix}.
\]
\end{enumerate}
\end{lem}

\begin{proof}
See \cite{code}.
\end{proof}

\begin{proof}[Proof of Proposition \ref{prop:locus_x3}]
Lemma \ref{lem:simplification_x3} implies that $\widetilde{X}_3 \subset X_3$. Since both $X_3$ and $\widetilde{X}_3$ are irreducible, all we have to show is that they have the same dimension. 

Assume that $A \cdot I_{u, v, w} = I_{u', v', w'}$ for $A \in \SL_3$ and general $u, v, w, u', v', w' \in \C$. Then $A$ has to fix both the point $(0:0:1)$ and the line cut out by $x$, i.e., $A$ is upper-triangular. By Lemma \ref{lem:stabilizer_x3} there are only finitely many such $A$. We can conclude that $\dim \widetilde{X}_3 = 11$.

To compute $\dim X_3$ note that the locus of four reduced points points together with a curvilinear triple point has dimension $12$. The condition about the line reduces the dimension by $1$.
\end{proof}

\stabilityFinalModelNSeven*

\begin{proof}
By Lemma \ref{lem:simplification_x1} and Lemma \ref{lem:simplification_x2} we know that $\widetilde{X_1} \subset \widetilde{X_3}$ and $\widetilde{X_2} \subset \widetilde{X_3}$. Together with Lemma \ref{lem:stability_final_n7} all we have to show is that $\widetilde{X_1} = X_1$, $\widetilde{X_2} = X_2$, and $\widetilde{X_3} = X_3$. However, these statements were shown in Propositions \ref{prop:locus_x1}, \ref{prop:locus_x2}, and \ref{prop:locus_x3}.
\end{proof}


\subsection{Minimal orbits}
We finish by determining the strictly semistable closed orbits
\begin{prop}
\label{prop:n_7_final_model_closed_orbits}
The closed strictly semistable orbits are given by $\SL_3 \cdot [e_4 + e_8 + we_9]$ for $w \in \C$ and $\SL_3 \cdot [e_4 + e_9]$. Moreover, we have $[e_4 + e_8 - (w + 1) e_9] \in \SL_3 \cdot [e_4 + e_8 + w e_9]$ and no other of pairs of points lie in the same orbit. In particular, the strictly semistable locus in the quotient is isomorphic to $\P^1$.

If $w \neq 0$, then $[e_4 + e_8 + we_9]$ generates an ideal sheaf $\II_Z$ that contains two curvilinear triple points $Z_1$, $Z_2$ not contained in lines with the following property: if $L_i$ for $i = 1, 2$ are the unique lines such that $Z_i \cap L_i$ has length two, then $P = L_1 \cap L_2$ is a single point, $Z_1$ and $Z_2$ are not supported at $P$, and $Z = Z_1 \cup Z_2 \cup (L_1 \cap L_2)$.

The points $[e_4 + e_8]$ and $[e_4 + e_9]$ do not generate ideal sheaves of zero-dimensional schemes.
\end{prop}

\begin{lem}
\label{lem:n_7_final_model_action}
For any $w \in \C$, we have $[e_4 + e_8 - (w + 1) e_9] = T \cdot [e_4 + e_8 + w e_9]$ for
\[
T = \begin{pmatrix}
0  & 0  & 1 \\ 
0  & -1 & 0 \\
-1 & 0  & 0 
\end{pmatrix}.
\]
\end{lem}

\begin{proof}
This is a straightforward computation.
\end{proof}

For any $w \in \C$, let $s_w = e_4 + e_8 + we_9$ and $I_w$ be the ideal generated by $[s_w]$. Similarly, we say $s_{\infty} = e_4 + e_9$ and $I_{\infty}$ is the ideal generated by $[s_{\infty}]$. We define
$J_w = (xy, x^2, (w + 1)y^2 + xz)$
and if $w \neq 0$, then
$K_w = (yz, z^2, wy^2 + xz)$.
Both of these ideals cut out curvilinear length three points supported at $(0:0:1)$, respectively $(1:0:0)$. The line cut out by $x$ is the unique line intersecting the subscheme cut out by $J_w$ in length two. Similarly, the line cut out by $z$ is uniquely determined by $K_w$. Note that the point $(0:1:0)$ lies precisely on the intersection of these two lines.

\begin{lem}
\label{lem:n_7_final_model_primary_decomposition}
For any $w \in \C$ we have
\[
I_w = (xz^2 + (w + 1)y^2z, -xyz, - x^2z -wxy^2)
\]
and $I_{\infty} = (xz^2 + y^2z, -xy^2 - x^2z)$.
If $w \neq 0$ and $w \neq -1$, then
\[
I_w = J_w \cap K_w \cap (x, z)
\]
Moreover, we have
\[
    I_{\infty} = (y^2z + xz^2, -xy^2 - x^2z) = (y^2 + xz) \cap (x, z), \
    I_0 = (z) \cap J_0, \
    I_{-1} = (x) \cap K_{-1}.
\]
\end{lem}

\begin{proof}
This can certainly be computed by hand, but it is a lot simpler and safer to look at \cite{code}.
\end{proof}

\begin{lem}
\label{lem:n_7_final_model_stabilizers}
Let $w, w' \in \C$ and $A \in \SL_3$. Then $A \cdot [s_w] = [s_{w'}]$ if and only if either
\begin{enumerate}
    \item $w = w'$ and $A = \lambda_0(t)$ for some $t \in \C$, or
    \item $w' = -w - 1$, and $A = \lambda_0(t) T$ for some $t \in \C$ and
    \[
    T = \begin{pmatrix}
    0 & 0 & 1 \\
    0 & -1 & 0 \\
    -1 & 0 & 0
    \end{pmatrix}.
    \]
\end{enumerate}
Moreover, $[s_{\infty}]$ is in a separate orbit from all $[s_w]$ and its stabilizer is given by those $A \in \SL_3$ for which either
\begin{enumerate}
    \item $A = \lambda_0(t)$ for some $t \in \C$, or
    \item $A = \lambda_0(t) S$ for some $t \in \C$ and
    \[
    S = \begin{pmatrix}
    0 & 0 & 1 \\
    0 & 1 & 0 \\
    1 & 0 & 0
    \end{pmatrix}.
    \]
\end{enumerate}
\end{lem}

\begin{proof}
By Lemma \ref{lem:n_7_final_model_primary_decomposition}, $[s_0]$, $[s_{-1}]$, and $[s_{\infty}]$ do not correspond to ideals of zero-dimensional schemes and therefore have to be in separate orbits. However, the ideals $I_0$, $I_{-1}$, and $I_{\infty}$ are still uniquely determined by them. Clearly, $I_{\infty}$ cannot be in the same orbit as $I_0$ or $I_{-1}$. We start with the case $[s_{\infty}]$. Then we have $A \cdot I_{\infty} = I_{\infty}$. This implies that $A$ has to fix the point $(0:1:0)$, i.e.,
\[
A = \begin{pmatrix}
a & b & c \\
0 & d & 0 \\
e & f & g
\end{pmatrix}
\]
for some $a, b, c, d, e, f, g \in \C$ with $(ag - ce)d = 1$. Acting on $y^2 + xz$ leads to
\[
(b^2 + ac)x^2 + 2bdxy + d^2y^2 + (ag + 2bf + ce)xz + 2dfyz + (f^2 + eg)z^2.
\]
We want to fix the conic cut out by $y^2 + xz$. Since $d \neq 0$, this is only possible if $b = f = 0$. Further comparision of coefficients leads to the equations
$ac = 0, eg = 0, d^2 = ag + ce$.
The precisely corresponds to the cases described in the statement.

Next, assume that $w = 0$. By Lemma \ref{lem:n_7_final_model_action} we know that $[s_0] = T \cdot [s_{-1}]$. Therefore, it will be enough to deal with the case where $A \cdot I_0 = I_0$. The subscheme cut out by $K_{-1}$ is supported at $(1:0:0)$. Moreover, the line cut out by $z$ is the unique line intersecting it in length two. Overall, $A$ has to fix both the lines cut out by $x$, $z$ and the point $(1:0:0)$. We write
\[
A = \begin{pmatrix}
a & 0 & 0 \\
0 & b & 0 \\
0 & c & d
\end{pmatrix} 
\]
for $a, b, c, d \in \C$ with $abd = 1$ and obtain
\[
A \cdot (yz, z^2, y^2 + xz) = (yz, z^2, b^2y^2 + adxz + 2bcyz + c^2z^2) = (yz, z^2, y^2 + xz).
\]
This is equivalent to $c = 0$, and $b^2 = ad$ which is the same as saying that $A = \lambda_0(t)$ for some $t \in \C$.

For the general case, assume that $w \neq 0$ and $w \neq -1$. The matrix $A$ must fix the point $(0:0:1)$. Moreover, it can either fix both $(1:0:0)$ and $(0:0:1)$, or it can exchange these two points. By Lemma \ref{lem:n_7_final_model_action} we know that $[s_w] = T \cdot [s_{-w-1}]$ and at the same point $T$ is exchanging $(1:0:0)$ and $(0:0:1)$. Therefore, it is enough to deal with the case where both points are fixed by $A$. Then $A$ must be a diagonal matrix, i.e.,
\[
A = \begin{pmatrix}
a & 0 & 0 \\
0 & b & 0 \\
0 & 0 & c
\end{pmatrix} 
\]
for $a, b, c, \in \C$ with $abc = 1$. The fact that $A$ fixes $J_w$ is equivalent to $b^2 = ac$ which is also equivalent to the fact that $A$ fixes $K_w$. Therefore, we get $A = \lambda_0(t)$ for some $t \in \C$. 
\end{proof}

\begin{proof}[Proof of Proposition \ref{prop:n_7_final_model_closed_orbits}]
Let $w \in \C \cup \{ \infty \}$. We start by showing that the orbit of $[s_w]$ is indeed semistable. We already know by the proof of Lemma \ref{lem:stability_final_n7} that no diagonal one-parameter subgroup can destabilize $[s_w]$. However, since $\lambda_0$ is in the stabilizer of $[s_w]$ by Lemma \ref{lem:n_7_final_model_stabilizers}, we know by Remark \ref{rmk:tori_to_check} that is suffices to check diagonal one-parameter subgroups.

Assume that $[s] \in \P(H^0(\Omega(4)))$ is strictly-semistable. Then up to a change of basis we can assume that  $s = a_0 e_0 + \ldots + a_4 e_4 + a_6 e_6 + \ldots + a_9 e_9$ for $a_0, \ldots, a_4, a_6, \ldots, a_9 \in \C$. Since $[s] \notin X_2$, we must have $a_4 \neq 0$. Since $[s] \notin X_1$, we must have $(a_8, a_9) \neq (0, 0)$. We can compute $\lim_{t \to 0} \lambda_0(t) s = a_4 e_4 + a_8 e_8 + a_9 e_9$. We can divide by $a_4$ to reduce to the case $a_4 = 1$. We can with $\lambda_1$ to reduce to $[s_w]$ or $[s_{\infty}]$ depending on whether $a_8 \neq 0$ or $a_8 = 0$. The point cannot be degenerated further without making it unstable. That is again due to the fact that $a_4 = 0$ implies $[s] \in X_1$ and $(a_8, a_9) = (0, 0)$ implies $[s] \in X_2$.

The claim about which points give the same orbit is a consequence of Lemma \ref{lem:n_7_final_model_stabilizers}. The description of the ideals and which points do not correspond to ideals follows from Lemma \ref{lem:n_7_final_model_primary_decomposition}.
\end{proof}


\section{Remaining GIT calculations}
\label{sec:git_calculations}

\subsection{Walls from curvilinear subschemes}
\label{sec:ProofMainThmWall}

The goal of this section is to do the GIT calculations used in the proof of Theorem \ref{thm:MainThmWall}. For that purpose, we first show in Lemma \ref{lem:wall_for_fat_plus_line} that certain ideals $\II_Z$ with positive dimensional stabilizer are strictly GIT semistable. This is done by explicitly calculating the numerical criterion.

\begin{lem}
\label{lem:mu_for_fat_plus_line}
Assume that $l, n$ are integers with $\tfrac{n}{2} \leq l < \tfrac{2n}{3}$ and $r_l(x,y) \in \C[x, y]$ is a homogenous polynomial  of degree $l$. Let $Z \subset \P^2$ be the subscheme cut out by the ideal
\[
\II_Z = (x, y^{n-l}) \cap (z, r_{l}(x, y)) = (xz, y^{n-l}z, r_{l}(x, y))
\]
and let $\lambda = \lambda(a, b)$ be the one-parameter subgroup $\diag\{ a,b,-a-b\}$.
Then, the Hilbert-Mumford index for the $l$-th and $(l+1)$-th Hilbert points are as follows:
\begin{enumerate}
    \item If $r_l = a_0 x^l + \ldots + a_l y^l$ for $a_0, \ldots, a_l \in \C$ with $a_0 \neq 0$ and $a_l \neq 0$, then \begin{align*}
    \mu_{l}(Z, \lambda) &=
    \begin{cases}
    -\frac{a}{2}(4l^2 - 4ln + n^2 - n) - \tfrac{b}{2}(5l^2 - 6ln + 2n^2 + 3l - 2n)
    & \text{if $a < b$,}
    \\
    -\frac{a}{2}(4l^2 - 4ln + n^2 + 2l - n) - \tfrac{b}{2}(5l^2 - 6ln + 2n^2 + l - 2n)
    &\text{if $a \geq b$,}
    \end{cases}
    \end{align*}
    and
    \begin{align*}
    \mu_{l + 1}(Z, \lambda)
    =
    \begin{cases}
    -\frac{a}{2}(4l^2 - 4ln + n^2 + 2l - 3n) - \tfrac{b}{2}(5l^2 - 6ln + 2n^2 + 7l - 4n)
    & \text{if $a < b$,}
    \\
    -\frac{a}{2}(4l^2 - 4ln + n^2 + 6l - 3n) - \tfrac{b}{2}(5l^2 - 6ln + 2n^2 + 3l - 4n)
    & \text{if $a \geq b$.}
    \end{cases}
    \end{align*}
    
    \item If $r_l = a_1 x^{n-1}y + \ldots + a_l y^l$ for $a_1, \ldots, a_l \in \C$ with $a_1 \neq 0$ and $a_l \neq 0$, then
    \begin{align*}
    \mu_{l}(Z, \lambda) &=
    \begin{cases}
    -\tfrac{a}{2}(4l^2 - 4ln + n^2 - n + 2) - \tfrac{b}{2}(5l^2 - 6ln + 2n^2 + 3l - 2n - 2)
    & \text{if $a < b$,}
    \\
    -\tfrac{a}{2}(4l^2 - 4ln + n^2 + 2l - n) - \tfrac{b}{2}(5l^2 - 6ln + 2n^2 + l - 2n)
    & \text{if $ a \geq b$,}
    \end{cases}
    \end{align*}
    and
    \begin{align*}
    \mu_{l + 1}(Z, \lambda)
    =
    \begin{cases}
    -\tfrac{a}{2}(4l^2 - 4ln + n^2 + 2l - 3n + 4) - \tfrac{b}{2}(5l^2 - 6ln + 2n^2 + 7l - 4n - 4)
    & \text{if $a < b$,}
    \\
    -\tfrac{a}{2}(4l^2 - 4ln + n^2 + 6l - 3n) - \tfrac{b}{2}(5l^2 - 6ln + 2n^2 + 3l - 4n)
    & \text{if $a \geq b$.}
    \end{cases}
    \end{align*}
\end{enumerate}
\end{lem}

\begin{proof}
We first find a basis for the relevant global sections that yield the $l$-th and $(l+1)$-th Hilbert points.
Let $r_l(x, y) = a_0 x^l + \ldots + a_l y^l$ for $a_0, \ldots, a_l \in \C$. If $l = \tfrac{n}{2}$ and $a_l \neq 0$, then
\[
(xz, y^{n - l}z, r_l(x, y)) = (xz, r_l(x, y)).
\]
Therefore, the $l$-th and $(l+1)$-th Hilbert points are well-defined for all $\tfrac{n}{2} \leq l < \tfrac{2n}{3}$. 
A direct inspection shows that
\begin{align*}
    H^0(\II_Z(l)) &= xz \cdot \langle x, y, z \rangle^{l - 2} \oplus y^{n - l}z \cdot \langle y, z \rangle^{2l - n - 1} \oplus r_l(x, y) \cdot \C, \\
    H^0(\II_Z(l + 1)) &= xz \cdot \langle x, y, z \rangle^{l - 1} \oplus y^{n - l}z \cdot \langle y, z \rangle^{2l - n} \oplus r_l(x, y) \cdot \langle x, y \rangle.
\end{align*}
If $l = \tfrac{n}{2}$, then $\langle y, z \rangle^{2l - n - 1} = \{ 0 \}$ and the above description makes sense.
Next, we calculate the state polytopes. For that purpose, we introduce the notation:
    \begin{alignat*}{3}
        n_{x} &\coloneqq \frac{l^3 - l}{6}, & \qquad
        m_{x} &\coloneqq \frac{l^3 + 3l^2 + 2l}{6}, \\
        n_{y} &\coloneqq \frac{l^3 - 3l^2 + 6ln - 3n^2 - 4l + 3n}{6}, &
        m_{y} &\coloneqq \frac{l^3 + 6ln - 3n^2 - l + 3n}{6}, \\
        n_{z} &\coloneqq \frac{l^3 + 12l^2 - 12ln + 3n^2 + 5l - 3n}{6}, \ &
        m_{z} &\coloneqq \frac{l^3 + 15l^2 - 12ln + 3n^2 + 20l - 9n + 6}{6}. 
    \end{alignat*}
By taking the wedge product of the basis elements in $H^0(\II_Z(l))$ and $H^0(\II_Z(l+1))$ given above, we obtain the following state polytopes.
\begin{enumerate}
    \item Assume that $a_0 \neq 0$ and $a_l \neq 0$. We can compute
    \begin{align*}
        \State_l(\II_Z) &= \Conv\{(n_x + l, n_y, n_z), (n_x, n_y + l, n_z)\}, \\
        \State_{l + 1}(\II_Z) &= \Conv\{(m_x + 2l + 1, m_y + 1, m_z), (m_x + 1, m_y + 2l + 1, m_z)\}.
    \end{align*}
    \item Assume that $a_0 = 0$, $a_1 \neq 0$, and $a_l \neq 0$. We can compute
    \begin{align*}
        \State_l(\II_Z) &= \Conv\{(n_x + l - 1, n_y + 1, n_z), (n_x, n_y + l, n_z)\}, \\
        \State_{l + 1}(\II_Z) &= \Conv\{(m_x + 2l - 1, m_y + 3, m_z), (m_x + 1, m_y + 2l + 1, m_z)\}.
    \end{align*}
\end{enumerate}
From these state polytopes, the statement is a straightforward computation by pairing the vertices with $\lambda(a,b)$.
     \qedhere
\end{proof}
For the following result, we recall that $\lambda_1$ is the one-parameter subgroup $\diag \{1,1,-2 \}$.
\begin{lem}
\label{lem:wall_for_fat_plus_line}
Given integers $l, n$ with $\tfrac{n}{2} \leq l < \tfrac{2n}{3}$, let $Z$ be the subscheme cut out by the ideal
\[
\II_Z = 
(x, y^{n-l}) \cap (z, r_{l}(x, y)) = (xz, y^{n-l}z, r_{l}(x, y))
\]
where $r_l(x,y)$ is a general polynomial of degree $l$. Then $Z$ is strictly GIT-semistable at 
\[
m_l = \frac{3(n - l)(n - l - 1)}{2(2n - 3l)}
\] 
and destabilized by $\lambda_1$ or $\lambda_1^{-1}$ for all other $m$. 
\end{lem}

\begin{proof}
For the statement to make sense the ideal $\II_Z$ needs to represent a point in the space where $D_{m_l}$ is ample, i.e., $\II_Z$ needs to be Bridgeland-semistable for the corresponding Bridgeland stability condition. By Corollary \ref{cor:ideals_between_-1_-2} we have to show that $m_l \geq l - 1$. This is a consequence of $\tfrac{n}{2} \leq l < \tfrac{2n}{3}$.

Next, we discuss the GIT semistability of our subscheme.
The stabilizer of $Z$ is $\lambda_1$ and by Remark \ref{rmk:tori_to_check} we only need to consider one-parameter subgroups in a special tori
which is not necessarily diagonal. The key idea is that instead of using a non-diagonal one-parameter subgroup, we can apply a change of coordinates  $\bar{x}= x$,  $\bar{y}= sx+y$ and $\bar{z}= z$ with $s \in \C$, so $\lambda$ is now diagonal.  
Within this new coordinate system our ideal has the form
\begin{align*}
\II_Z &= (\bar x, \bar y^{n-l}) \cap (\overline{z}, \bar r_{l}(\bar x, \bar y)) = (\bar x \bar z, \bar y^{n-l}\bar z, \bar r_{l}(\bar x, \bar y)).
\end{align*}
This ideal is of the same form as the one we started with, but there is one issue. For special choices of $s$, we can not assume that $\bar r_l( \bar x, \bar y)$ is general anymore.

If $r_l(x,y)$ is general, then after this base-change we can only be in the two situations described in Lemma \ref{lem:mu_for_fat_plus_line}. The numerical criterion for $\lambda_1$ is the same either way. Together with Lemma \ref{lemma:m0} we get that $\mu_{m_l}(Z, \lambda_1) = 0$ for
\begin{align*}
m_l = \frac{3(n - l)(n - l - 1)}{2(2n - 3l)}.
\end{align*}
Moreover, if $m \neq m_l$, then either $\mu_{m}(Z, \lambda_1) > 0$ or $\mu_{m}(Z, \lambda_1^{-1}) > 0$ and thus, $Z$ is $m$-unstable. 

Next, we calculate $\mu_{m_l}(Z, \lambda(a,b)) \leq 0$ for all $a$ and $b$.
We first observe that 
$D_{m_l} = xD_l + yD_{l+1}$ implies the equations $xl + y(l + 1) = m_l$ and $x + y = 1$. Solving for $x$ and $y$, we obtain
\begin{align*}
x = -\frac{(9l^2 - 10ln + 3n^2 + 9l - 7n)}{2(2n - 3l)}, \
y = \frac{(9l^2 - 10ln + 3n^2 + 3l - 3n)}{2(2n - 3l)}.
\end{align*}
By construction, we have the equality 
$$
    \mu_{m_l}(Z, \lambda(a,b))
    =
    x\mu_{l}(Z, \lambda(a,b))
    +
    y\mu_{l+1}(Z, \lambda(a,b)).
$$
If we assume that case (i) in Lemma \ref{lem:mu_for_fat_plus_line} applies. Then, we obtain
    \begin{align*}
    \mu_{m_l}(Z, \lambda(a,b))
    =
    \begin{cases}
    \frac{(3l^3 - l^2n - 2ln^2 + n^3 - 3l^2 + 3ln - n^2)}{2(2n - 3l)}(a - b)
    & \text{ if $a < b$,}
     \\
    -\frac{(2l - n)(3l^2 - 3ln + n^2 - n)}{2(2n - 3l)}(a - b)
    & \text{ if $a \geq b$}.
    \end{cases}
    \end{align*}
On the  other hand,  if we assume that case (ii) in Lemma \ref{lem:mu_for_fat_plus_line} applies. Then, we obtain
    \begin{align*}
    \mu_{m_l}(Z, \lambda(a,b))
    =
    \begin{cases}
    \frac{(3l^3 - l^2n - 2ln^2 + n^3 - 12l^2 + 13ln - 4n^2 + 3l - n)}{2(2n - 3l)}(a - b)
    & \text{ if $a < b$,}
     \\
    -\frac{(2l - n)(3l^2 - 3ln + n^2 - n)}{2(2n - 3l)}(a - b)
    & \text{ if $a \geq b$}.
    \end{cases}
    \end{align*}
In both cases, the inequalities $\tfrac{n}{2} \leq l < \tfrac{2n}{3}$ imply that $\mu_{m_l}(Z, \lambda(a,b)) \leq 0$ for all $a, b$, and our result follows.
 \qedhere
\end{proof}

\subsection{The largest wall}
\label{sec:ProofThmLargestWall}

In this section, we establish computational statements needed in the proof of  Theorem \ref{thm:largestWall}. To do so, it is important to describe certain monomial ideals that arise as flat limits of our configurations of points. For that purpose, we recall key facts from \cite{OS99:cutting_corners}.

For the moment we will look at monomial ideals $I \subset \C[y,z]$ cutting out subschemes of $\A^2$. For any weight vector $\mathbf{w} = (a, b) \in \R_{\geq 0}^2$ we denote the \emph{initial ideal} of $I$ with respect to the corresponding monomial order by $\initial_{\mathbf{w}}(I)$. We say that two non-negative vectors $\mathbf{w}$ and $\overline{ \mathbf{w} }$ are equivalent if their corresponding initial ideals are the same, that is, 
$\initial_{\mathbf{w}}(I)= \initial_{\overline{\mathbf{w}}}(I)$. The equivalence classes are cones that form a fan covering $\R^2_{\geq 0}$ called the \emph{Gr\"obner fan} of $I$. It is a standard fact that a vector $\mathbf{w}$ lies in an open cell of the Gr\"obner fan if and only if  $\initial_{\mathbf{w}}(I)$ is a monomial ideal.

\begin{defn}
A nonempty finite subset $M$ of the set $\mathbb{N}^2$ of non-negative integer vectors is a staircase if $m \in M$ and $\overline{m} \leq m$ (coordinatewise) implies that $\overline{m} \in M$. A staircase $M$ is called a \emph{corner cut} if for some $\mathbf{w} \in \R^2$ we have $\mathbf{w} \cdot m < \mathbf{w} \cdot m'$ for all $m \in M$ and $m' \in \N^2 \setminus M$.
\end{defn}

In this case, there is $s \in \R_{\geq 0}$ such that $M = \{ m \in \N^2 \; | \; \mathbf{w} \cdot m \leq s \}$ and no other corner cut of length $n$ can be written in this way. Therefore, we denote $M$ as $M_{\mathbf{w}}$ and suppress the $n$ from the notation. For example, for any triangular number $n = 1 + \ldots + s$ we have a corner cut \[
M_{(1, 1)} = \{ (a, b) \in \N^2 \; | \; a + b \leq s \}.
\]
Note that for a fixed length $n$ not any $\mathbf{w}$ induces a corner cut of that length.

\begin{defn}
Let $M$ be corner cut. Then we define the ideal
$I_M \coloneqq \{ y^a z^b \; | \; (a,b) \in \N^2 \setminus M \}$.
\end{defn}

For example, if $M = \{ (0, 0), (1, 0), (0, 1) \}$, then $\II_M = (x^2, xy, y^2)$.

\begin{theorem}[{\cite{OS99:cutting_corners}}]
\label{thm:cutting_corners}
Let $I_n \subset \C[y,z]$ be the ideal of $n$ generic points in $\A^2$ and let $\mathbf{w} \in \R_{\geq 0}^2$ be a weight in an open cone of the Gr\"obner fan of $I_n$. Then  $\initial_{\mathbf{w}}(I_n) = I_{M_{\mathbf{w}}}$.
\end{theorem}

\begin{proof}
If we fix the the number of elements in the corner cut $n = |M|$, there is is only a finite number of them. By \cite[Thm 2.0]{OS99:cutting_corners}, the map $M_{\mathbf{w}} \to \sum_{m \in M} m$ defines a bijection between the corner cuts and the vertex set of a certain polytope called $P^2_n$. The key point is that by \cite[Thm 5.0]{OS99:cutting_corners}, the normal fan of $P^2_n$ equals the Gr\"obner fan of $I_n$ (that is, each vertex of $P^2_n$ is associated to a maximal cone of the Gr\"obner fan). Therefore, we obtain a bijection between corner cuts $M_{\mathbf{w}}$ and monomial ideals $\initial_{\mathbf{w}}(I_n)$. By the equivalence of $(1)$ and $(7)$ in the proof of \cite[Theorem 5.2]{OS99:cutting_corners}, we obtain that
$\initial_{\mathbf{w}}(I_n) = I_{M_{\mathbf{w}}}$.
\end{proof}
Our first step is to calculate the flat limit of $n$ points in general position with respect to a certain one-parameter subgroups $\lambda_r = \diag\{1,r,-1-r\}.$

\begin{lem}
\label{lem:limit_with_cutting_corners}
Let $Z \subset \P^2$ be a subscheme consisting of $n$ general points. We write $n = \Delta(s) + k$ where $\Delta(s)$ is the triangular number $1 + \ldots + s$ and $0 \leq k \leq s$. For small enough $\varepsilon > 0$ we have
\[
    \lim_{t \to 0} \lambda_{-\tfrac{1}{2} - \varepsilon}(t) \cdot \II_Z = (m_i \; | \; 0 \leq i \leq s),
\]
where
\[
    m_i = \begin{cases}
    y^{s - i} z^i & \text{ if } 0 \leq i \leq (s - k), \\
    y^{s - i} z^{i + 1} & \text{ if } (s - k + 1) \leq i \leq s.
    \end{cases}
\]
Additionally, if $n = \Delta(s)$, i.e., $k = 0$, then
\[
    \lim_{t \to 0} \lambda_{-\tfrac{1}{2}}(t) \cdot \II_Z = (y^s, y^{s - 1}z, \ldots, z^s).
\]
\end{lem}

\begin{proof}
Since $Z$ consists of $n$ general points, we may assume that they lie in the affine open subset $\A^2 = \{ (x:y:z) \in \P^2 \; | \; x \neq 0 \}$. For any $r \in [-2, 1]$ we have
\begin{align*}
\lim_{t \to 0} \lambda_r \cdot (x:y:z) = \lim_{t \to 0} (t^{-1} x: t^{-r} y: t^{r + 1} z) = \lim_{t \to 0} (x: t^{1 - r} y: t^{r + 2} z).
\end{align*}
From this it is not difficult to see that the subscheme cut out by $\lim_{t \to 0} \lambda_r \cdot \II_Z$ is given by the ideal $\initial_{\mathbf{w}} \II_Z$ for $\mathbf{w} = (1 - r, r + 2)$ in $\A^2$. See \cite[Cor 3.5]{BM88:standard_bases} for more details on this relation between the limit and the initial ideals.

For $r = -\tfrac{1}{2} - \varepsilon$ where $\varepsilon \geq 0$, we get
\[
\mathbf{w(\varepsilon)} = \left( \frac{3}{2} + \varepsilon, \frac{3}{2} - \varepsilon \right).
\]
If $n = \Delta(s)$ and $\varepsilon = 0$, then $M_{\mathbf{w(0)}} = \{ (a, b) \in \N^2 \; | \; a + b \leq s \}$ is corner cut of length $n$ and together with Theorem \ref{thm:cutting_corners} we get
\[
\lim_{t \to 0} \lambda_{-\tfrac{1}{2}}(t) \cdot \II_Z = (y^s, y^{s - 1}z, \ldots, z^s).
\]
Now assume that $\varepsilon > 0$. Then we examine the corner cut $M_{\mathbf{w(\varepsilon)}}$ and again Theorem \ref{thm:cutting_corners} implies
\[
    \lim_{t \to 0} \lambda_{-\tfrac{1}{2} - \varepsilon}(t) \cdot \II_Z = (m_i \; | \; 0 \leq i \leq s). \qedhere
\]
\end{proof}

\begin{lem}
\label{lem:numerical_criterion_wall_from_collinear}
Let $Z \subset \P^2$ be the subscheme of length $n = \Delta(s) + k + l$ where $s \geq 1$, $\Delta(s) = 1 + \ldots + s$, $0 \leq k \leq s$, and $l \geq s + 1$ cut out by the ideal
\begin{align*}
    \II_Z &= (m_0, \ldots, m_s) \cap (x, p_l(y, z)) = (xm_0, \ldots, xm_s, p_l(y, z)),
\end{align*}
where $p_l(y, z) \in \C[y, z]$ is a general homogeneous polynomial of degree $l$. Then
\[
    \mu_l \left(Z, \lambda_{-\tfrac{1}{2}}\right) = \frac{s^3 - ls^2 + 3ks - ls - s + l^2 - 2kl}{2}.
\]
\end{lem}

\begin{proof}
The global sections $H^0\left( \II_Z(l) \right)$ can be computed as
\begin{align*}
    xy^s \cdot \langle x, y, z \rangle^{l - s - 1} \oplus
    \bigoplus_{i = 1}^{s - k} x y^{s - i} z^i \langle x, z \rangle^{l - s - 1} \oplus
    \bigoplus_{i = s - k + 1}^s x y^{s - i} z^{i + 1} \langle x, z \rangle^{l - s - 2} \oplus
    p_l(y, z) \cdot \C
\end{align*}
If we define
\begin{align*}
    n_x &= \frac{(l - s)(l^2 + ls - 2s^2 - 6k + 3l + 2)}{6}, \\
    n_y &= \frac{l^3 - s^3 - 3k^2 + 3k - l + s}{6}, \\
    n_z &= \frac{l^3 - s^3 + 3k^2 - 6ks - 3k - l + s}{6}, \\
\end{align*}
then the state polytope is given by $\State_l(\II_Z) = \Conv \{ (n_x, n_y + l, n_z), (n_x, n_y, n_z + l) \}$.
The value of the Hilbert-Mumford index is an immediate consequence.
\end{proof}

\subsection{Walls from collinear points}
\label{sec:WallsCollinear}

In this section, we will deal with GIT-walls that destabilize subschemes with $l$ collinear points. As in our previous cases, we first calculate the Hilbert-Mumford index of relevant ideals with positive dimensional stabilizer. 

\begin{lem}
\label{lem:mu_for_collinear_points}
Let $l, n$ be integers such that $n - l$ is the triangular number $\Delta(s) = 1 + \ldots + s$ and $s^2 \leq l < 2\Delta(s)$. Assume further that $p_l(y, z)$ is a homogeneous polynomial of degree $l$ and $Z \subset \P^2$ is the subscheme cut out by the ideal
\[
    \II_Z = (y^s, y^{s-1}z, \cdots, yz^{s-1}, z^{s}) \cap (x, p_l(y, z)) = (xy^s, xy^{s-1}z, \cdots, xyz^{s-1}, xz^{s}, p_l(y, z)).
\]
Then, the Hilbert-Mumford index for the $l$-th and the $(l+1)$-th Hilbert points with respect to $\lambda(a,b) = \diag\{a,b,-a-b\}$
are as follows:
\begin{enumerate}
    \item If $p_l(y, z) = a_0 y^l + \cdot  + a_l z^l$ for $a_0, \ldots, a_l \in C$ with $a_0 \neq 0$ and $a_l \neq 0$, then
    \[
        \mu_l(\II_Z, \lambda(a, b)) =
        \begin{cases}
            \frac{a}{2}(s^3 - ls^2 + l^2 - ls - l - s) - bl & \text{if $a + 2b \geq 0$,} \\
            \frac{a}{2}(s^3 - ls^2 + l^2 - ls + l - s) + bl &  \text{if $a + 2b < 0$,}
        \end{cases}
    \]
    and
    \[
        \mu_{l + 1}(\II_Z, \lambda(a, b)) =
        \begin{cases}
            \frac{a}{2}(s^3 - ls^2 + l^2 - ls - s^2 - l - 2s) - 2bl & \text{if $a + 2b \geq 0$,} \\
            \frac{a}{2}(s^3 - ls^2 + l^2 - ls - s^2 + 3l - 2s) + 2bl & \text{if $a + 2b < 0$.}
        \end{cases}
    \]
    \item If $p_l(y, z) = a_1 y^{l - 1}z + \cdot  + a_l z^l$ for $a_1, \ldots, a_l \in C$ with $a_1 \neq 0$ and $a_l \neq 0$, then
    \[
        \mu_l(\II_Z, \lambda(a, b)) =
        \begin{cases}
            \frac{a}{2}(s^3 - ls^2 + l^2 - ls - l - s) - bl & \text{if $a + 2b \geq 0$,} \\
            \frac{a}{2}(s^3 - ls^2 + l^2 - ls + l - s - 2) + b(l - 2) & \text{if $a + 2b < 0$,}
        \end{cases}
    \]
    and
    \[
        \mu_{l + 1}(\II_Z, \lambda(a, b)) =
        \begin{cases}
            \frac{a}{2}(s^3 - ls^2 + l^2 - ls - s^2 - l - 2s) - 2bl & \text{if $a + 2b \geq 0$,} \\
            \frac{a}{2}(s^3 - ls^2 + l^2 - ls - s^2 + 3l - 2s - 4) + 2b(l - 2) & \text{if $a + 2b < 0$.}
        \end{cases}
    \]
\end{enumerate}
\end{lem}

\begin{proof}
The proof is analogous to that of Lemma \ref{lem:mu_for_fat_plus_line}. The Hilbert points are given by 
\begin{align*}
    H^0\left( \II_Z(l) \right) &=
    xy^s \cdot \langle x, y, z \rangle^{l - s - 1} \oplus 
    xy^{s - 1}z \cdot \langle x, z \rangle^{l - s - 1} \oplus \cdots \oplus
    xz^s \cdot \langle x, z \rangle^{l - s - 1} \oplus 
    p_l(y, z) \cdot \C
    \\
    H^0\left( \II_Z(l + 1) \right) &= 
    xy^s \cdot \langle x, y, z \rangle^{l - s} \oplus 
    xy^{s - 1}z \cdot \langle x, z \rangle^{l - s} \oplus \cdots \oplus 
    xz^s \cdot \langle x, z \rangle^{l - s} \oplus 
    p_l(y, z) \cdot \langle y, z \rangle.
\end{align*}
If we define
    \begin{alignat*}{3}
        n_{x} &\coloneqq \frac{(l + 2s + 2)(l - s + 1)(l - s)}{6}, \ \ &\qquad & 
        m_{x} \coloneqq \frac{(l + 2s + 3)(l - s + 2)(l - s + 1)}{6}, \\
        n_{yz} &\coloneqq \frac{l^3 - l - s^3 + s}{6}, &&
        m_{yz} \coloneqq \frac{l^3 + 3l^2 + 2l - s^3 + s}{6},
    \end{alignat*}
then the state polytopes are given as follows:
\begin{enumerate}
    \item If $p_l(y, z) = a_0 y^l + \cdot  + a_n z^l$ for $a_0, \ldots, a_l \in C$ with $a_0 \neq 0$ and $a_l \neq 0$. Then,
    \begin{align*}
        \State_l(\II_Z) &= \Conv \{(n_x, n_{yz} + l, n_{yz}), (n_x, n_{yz}, n_{yz} + l)\}, \\
        \State_{l + 1}(\II_Z) &= \Conv \{(m_x, m_{yz} + 2l + 1, m_{yz} + 1), (m_x, m_{yz} + 1, m_{yz} + 2l + 1)\}.
    \end{align*}
    \item If $p_l(y, z) = a_1 y^{l - 1}z + \cdot  + a_n z^l$ for $a_1, \ldots, a_l \in C$ with $a_1 \neq 0$ and $a_l \neq 0$. Then, 
    \begin{align*}
        \State_l(\II_Z) &= \Conv \{(n_x, n_{yz} + l - 1, n_{yz} + 1), (n_x, n_{yz}, n_{yz} + l)\}, \\
        \State_{l + 1}(\II_Z) &= \Conv \{(m_x, m_{yz} + 2l - 1, m_{yz} + 3), (m_x, m_{yz} + 1, m_{yz} + 2l + 1)\}.
    \end{align*}
\end{enumerate}
The statement is a direct computation from here. 
\qedhere
\end{proof}

\begin{lem}
\label{lem:bridgeland_for_collinear_points}
Let $l, n$ be integers such that $n - l$ is the triangular number $\Delta(s) = 1 + \ldots + s$ and $s^2 \leq l < 2\Delta(s)$. Assume further that $p_l(y, z)$ is a homogeneous polynomial of degree $l$ and $Z \subset \P^2$ is the subscheme cut out by the ideal
\[
    \II_Z = (y^s, y^{s-1}z, \cdots, yz^{s-1}, z^{s}) \cap (x, p_l(y, z)) = (xy^s, xy^{s-1}z, \cdots, xyz^{s-1}, xz^{s}, p_l(y, z)).
\]
Then $\II_Z$ is Bridgeland-semistable for any stability condition that induces the divisor $D_{m_{l,s}}$ with
\[
m_{l, s} = \frac{s(s + 1)(s - 1)}{s^2 + s - l}.
\] 
\end{lem}
\begin{proof}
By Corollary \ref{cor:ideals_between_-1_-2} we have to show that $m_l \geq l - 1$. This is a consequence of $s^2 \leq l < 2\Delta(s)$.
\end{proof}

\begin{lem}
\label{lem:wall_for_collinear_points}
Let $l, n$ be integers such that $n - l$ is the triangular number $\Delta(s) = 1 + \ldots + s$. Assume further that $s^2 \leq l < 2\Delta(s)$. If $p_l(y, z)$ is a general homogeneous polynomial of degree $l$ and $Z \subset \P^2$ is the subscheme cut out by the ideal
\[
    \II_Z = (y^s, y^{s-1}z, \cdots, yz^{s-1}, z^{s}) \cap (x, p_l(y, z)) = (xy^s, xy^{s-1}z, \cdots, xyz^{s-1}, xz^{s}, p_l(y, z)),
\]
then $Z$ is strictly semistable at $m_{l, s} = \frac{s(s + 1)(s - 1)}{s^2 + s - l}$ and destabilized by $\lambda_{-1/2}$ or $\lambda_{-1/2}^{-1}$ for other $m$.
\end{lem}

\begin{proof}
The strategy of the proof is the exact same as for Lemma \ref{lem:wall_for_fat_plus_line} and we leave some details to the reader. By Lemma \ref{lem:bridgeland_for_collinear_points} the ideal $\II_Z$ is Bridgeland-semistable at the supposed wall, i.e., the statement makes sense to begin with.

The one-parameter subgroup $\lambda_{-1/2}$ stabilizes $Z$, and we can use Remark \ref{rmk:tori_to_check} to reduce which tori to check. Instead, we will again change the coordinates for $Z$ to only deal with diagonal one-parameter subgroups. Since $p_l(y, z)$ is general, everything reduces to the two possibilities in Lemma \ref{lem:mu_for_collinear_points}. Using the values computed there we get $\mu_{m_l}(\II_Z, \lambda(a, b)) \leq 0$. Moreover, $\mu_{m}(\II_Z, \lambda_{-1/2}) = 0$ holds if and only if $m = m_{l, s}$.
\end{proof}

\subsection{Wall from points in a conic}
\label{sec:WallPtsConic}

Lastly, we will establish a few more results used in the proof of Proposition \ref{prop:wall_conic}.

\begin{lem}
\label{lem:wall_conic_even}
Let $Z \subset \P^2$ be the length $n = 2k$ subscheme cut out by $\II_Z = (y^2 + xz, x^k)$ for $k \geq 3$. Then $Z$ is strictly-semistable at $m = k - \tfrac{1}{2}$ and destabilized by $\lambda_0$ or $\lambda_0^{-1}$ for all other $m$.
\end{lem}

\begin{proof}
Note that by Corollary \ref{cor:ideals_between_-1_-2} the ideal $\II_Z$ is destabilized in Bridgeland stability precisely at the wall that induces the divisor $D_{k - 1/2}$. Therefore, $\II_Z$ represents a point in the space for $m \geq k - \tfrac{1}{2}$ and the statement we are trying to prove makes sense.

The subscheme $Z$ is supported at $(0:0:1)$ and has length $2k$. Its stabilizer contains the one-parameter subgroup $\lambda_0$, and by Remark \ref{rmk:tori_to_check} we only have to check diagonal one-parameter subgroups to establish GIT-stability. We compute the relevant global sections:
\[
    H^0(\II_Z(k)) = (y^2 + xz) \cdot \langle x, y, z \rangle^{k - 2} \oplus \langle x^k \rangle, \
    H^0(\II_Z(k + 1)) = (y^2 + xz) \cdot \langle x, y, z \rangle^{k - 1} \oplus x^k \cdot \langle x, y, z \rangle.
\]
Then $\State_k(\II_Z)$ is given by
\begin{align*}
    \Conv \bigg\lbrace &\left(
    \frac{k(k^2 + 5)}{6}, 
    \frac{k(k - 1)(k - 2)}{6}, 
    \frac{k(k - 1)(k + 1)}{6}
    \right), \\
    &\left(
    \frac{k(k^2 - 3k + 8)}{6},
    \frac{k(k - 1)(k + 4)}{6},
    \frac{k(k - 1)(k - 2)}{6}
    \right) \bigg\rbrace.
\end{align*}
Moreover, $\State_{k + 1}(\II_Z)$ is given by
\begin{align*}
    \Conv \bigg\lbrace &\left(
    \frac{k(k^2 + 3k + 20)}{6}, 
    \frac{k^3 - k + 18}{6},
    \frac{k(k + 1)(k + 2)}{6}
    \right), \\
    &\left(
    \frac{k^3 + 17k + 6}{6},
    \frac{k^3 + 6k^2 + 5k + 6}{6},
    \frac{(k + 2)(k^2 - 2k + 3)}{6}
    \right) \bigg\rbrace.
\end{align*}
From this we can compute the Hilbert-Mumford indices
\begin{align*}
    \mu_k(Z, \lambda(a, b)) &=
    \begin{cases}
        -\frac{1}{2}k(bk - b - 2a) & \text{ if $b \geq 0$,} \\
        k(bk - b + a) & \text{ if $b < 0$,}
    \end{cases} \\
    \mu_{k + 1}(Z, \lambda(a, b)) &=
    \begin{cases}
        -\frac{1}{2} (bk^2 + bk - 6b - 6ak) & \text{ if $b \geq 0$,} \\
        k(bk + b + 3a) & \text{ if $b < 0$.}
    \end{cases}
\end{align*}
For $\lambda_0 = \lambda(1, 0)$, we get $\mu_k(Z, \lambda_0) = k > 0$ and $\mu_{k + 1}(Z, \lambda_0) = 3k > 0$. With Lemma \ref{lemma:m0} we obtain that $\mu_m(Z, \lambda_0) = 0$ if and only if $m = k - \tfrac{1}{2}$. In particular, $Z$ is GIT-unstable for $m \neq k - \tfrac{1}{2}$.

We have $D_{k - \tfrac{1}{2}} = \tfrac{3}{2} D_k - \tfrac{1}{2} D_{k + 1}$, and therefore,
\[
    \mu_{k - \tfrac{1}{2}}(Z, \lambda(a, b)) = 
    \begin{cases}
        -\tfrac{1}{2} b(k^2 - 2k + 3) & \text{ if $b > 0$,} \\
        bk(k - 2) & \text{ if $b < 0$.}
    \end{cases}
\]
In either case, $\mu_{k - \tfrac{1}{2}}(Z, \lambda(a, b)) < 0$.
\end{proof}

\begin{lem}
\label{lem:wall_conic_odd}
Let $Z \subset \P^2$ be the length $n = 2k - 1$ subscheme cut out by the ideal
\[
    \II_Z = (y^2 + xz, x^k, x^{k - 1}y)
\]
for $k \geq 3$. Then $Z$ is strictly-semistable at $m = k - 1$ and destabilized by $\lambda_0$ for all other $m$.
\end{lem}

\begin{proof}
Note that by Corollary \ref{cor:ideals_between_-1_-2} the ideal $\II_Z$ is destabilized in Bridgeland stability precisely at the wall that induces the divisor $D_{k - 1}$. Therefore, $\II_Z$ represents a point in the space for $m \geq k - 1$ and the statement we are trying to prove makes sense.

The subscheme $Z$ is supported at $(0:0:1)$ and has length $2k - 1$. Its stabilizer contains the one-parameter subgroup $\lambda_0$, and by Remark \ref{rmk:tori_to_check} we only have to check diagonal one-parameter subgroups to establish GIT-stability.

We define the vector space
\[
V \coloneqq \langle x^r y^s z^{k - 1 - r - s} : (r, s) \neq (k - 1, 0), (r, s) \neq (k - 2, 1) \rangle.
\]
With this we can describe the relevant global sections:
\begin{align*}
    H^0(\II_Z(k)) &= (y^2 + xz) \cdot \langle x, y, z \rangle^{k - 2} \oplus \langle x^k, x^{k - 1}y \rangle, \\
    H^0(\II_Z(k + 1)) &= (y^2 + xz) \cdot V \oplus \langle x^{k - 2} y^3, x^{k + 1}, x^{k}y, x^{k}z, x^{k - 1}y^2, x^{k - 1}yz \rangle.
\end{align*}
The State polytopes are
\begin{align*}
    \State_k(\II_Z) = \Conv \bigg\lbrace &\left(
    \frac{k^3 + 11k - 1}{6}, 
    \frac{(k + 1)(k^2 - 4k + 6)}{6}, 
    \frac{k(k - 1)(k + 1)}{6}
    \right), \\
    &\left(
    \frac{k^3 - 3k^2 + 14k - 6}{6},
    \frac{k^3 + 3k^2 - 4k + 6}{6},
    \frac{k(k - 1)(k - 2)}{6}
    \right) \bigg\rbrace,
\end{align*}
\begin{align*}
    \State_{k + 1}(\II_Z) = \Conv \bigg\lbrace &\left(
    \frac{k^3 + 3k^2 + 26k - 12}{6}, 
    \frac{k^3 - k + 36}{6}, 
    \frac{k(k + 1)(k + 2)}{6}
    \right), \\
    &\left(
    \frac{k(k^2 + 23)}{6},
    \frac{k^3 + 6k^2 + 5k + 12}{6},
    \frac{k^3 - k + 12}{6}
    \right) \bigg\rbrace.
\end{align*}
From this we can compute the Hilbert-Mumford indices
\begin{align*}
    \mu_k(Z, \lambda(a, b)) &=
    \begin{cases}
        -\tfrac{1}{2}bk^2 + 2ak + \tfrac{1}{2}bk - a + b & \text{ if $b \geq 0$,} \\
        bk^2 + 2ak - bk - a + b & \text{ if $b < 0$,}
    \end{cases} \\
    \mu_{k + 1}(Z, \lambda(a, b)) &=
    \begin{cases}
        -\tfrac{1}{2}bk^2 + 4ak - \tfrac{1}{2}bk - 2a + 6b & \text{ if $b \geq 0$,} \\
        bk^2 + 4ak + bk - 2a & \text{ if $b < 0$.}
    \end{cases}
\end{align*}
For $\lambda_0 = \lambda(1, 0)$, we get $\mu_k(Z, \lambda_0) = 2k - 1 > 0$ and $\mu_{k + 1}(Z, \lambda_0) = 4k - 2 > 0$. With Lemma \ref{lemma:m0} we obtain that $\mu_m(Z, \lambda_0) = 0$ if and only if $m = k - 1$. In particular, $Z$ is GIT-unstable for $m \neq k - 1$.

We have $D_{k - 1} = 2D_k - D_{k + 1}$, and therefore,
\[
    \mu_{k - 1}(Z, \lambda(a, b)) = 
    \begin{cases}
        -\tfrac{1}{2} b(k^2 - 3k + 8) & \text{ if $b > 0$,} \\
        b(k - 1)(k - 2) & \text{ if $b < 0$.}
    \end{cases}
\]
In either case, $\mu_{k - 1}(Z, \lambda(a, b)) < 0$.
\end{proof}

\def\cprime{$'$}

\end{document}